\documentclass[12pt]{article}
\pdfoutput=1
\usepackage[pdftex]{graphicx}
\usepackage{bm}
\usepackage{amsmath,amsthm,amssymb,url}
\usepackage{multirow,bigdelim}
\usepackage{amscd}
\usepackage{enumerate}
\usepackage{cellspace}
\newtheorem{thm}{Theorem}[section]
\newtheorem{theorem}[thm]{Theorem}
\newtheorem{proposition}[thm]{Proposition}
\newtheorem{lemma}[thm]{Lemma}

\newtheorem{corollary}[thm]{Corollary}

\newtheorem{remark}[thm]{Remark}

\renewcommand{\tilde}{\widetilde}

\usepackage{tikz}
\usetikzlibrary{intersections}
\usetikzlibrary{arrows}
\usetikzlibrary{cd}
\usetikzlibrary{math}
\usetikzlibrary{calc}
\usetikzlibrary{patterns}
\usetikzlibrary{decorations.markings}
\tikzset{
    segment/.style={line width=1.3pt},
    vertex/.style={
        inner sep=0.5mm,circle,draw,line width=1.5pt,fill=white
    },
    dsegment/.style={segment, dash pattern = on 5pt off 2pt},
    edge/.style={
        line width=1.3pt,
        decoration={
            markings,
            mark=at position 0.6 with {\arrow{>}}
        }
    },
    cylinder/.style = {line width=0.7pt,dash pattern = on 5pt off 5pt},
    graphfill/.style = {gray!50}
}

\title{Symmetric contact systems of segments, pseudotriangulations and inductive constructions for corresponding surface graphs}
\author{James Cruickshank \thanks{
Department of Mathematics, Statistics and Applied Mathematics, National University of Ireland, Galway,
Ireland. \texttt{james.cruickshank@nuigalway.ie}
}
\and
Bernd Schulze \thanks{
Department of Mathematics and Statistics, Lancaster University, 
Lancaster, UK. \texttt{b.schulze@lancaster.ac.uk}}
}
\date{}

\begin{document}


\maketitle


\begin{abstract}
We characterise the quotient surface graphs arising from symmetric contact systems of line segments in the plane and also from symmetric pointed pseudotriangulations in the case where the group of symmetries is generated by a translation or a rotation of finite order.  These results generalise well known results of Thomassen, in the case of line segments, and of Streinu and Haas et al.,  in the case of pseudotriangulations.  Our main tool is a new inductive characterisation of the appropriate classes of surface graphs. We also discuss some consequences of our results in the area of geometric rigidity theory.
\end{abstract}

\section{Introduction}

There has been much interest recently in adapting  
results of combinatorial geometry in areas such as geometric
 rigidity theory, polyhedral scene analysis, or the theory of packings,
 to a symmetric setting (see \cite[Chapters 2, 61, 62]{HB18}, for example,
 for a summary of recent results). Since symmetry is 
ubiquitous in both natural and artificial structures, much of this
work is motivated by applications in materials science,
biophysics and engineering. The purpose of this paper is to provide
symmetric generalisations of two significant results in 
combinatorial geometry, which we now describe.

The first result is concerned with an analogue of the well-known planar circle packing theorem
of Koebe-Andreev-Thurston, where circles are replaced with 
line segments. A \( 2 \)-contact system of line segments in the plane is a finite 
collection of segments such that any point belongs to at most two segments
and belongs to the interior of at most one segment.
Thomassen \cite{Thomassen93} 
has shown that a graph is the intersection graph of
such a contact system of line segments if and only if it is a subgraph
of a planar Laman graph.

The second result is concerned with pointed pseudotriangulations,
which are plane graphs
 with straight line edges
such that every bounded region is a polygon with exactly three 
convex angles in its interior, 
the boundary of the unbounded region is a convex polygon, and such that 
every vertex has exactly one non-convex incident angle. 
Such objects
have been extensively studied and have found wide-ranging applications,
for example in the solution of the carpenter's rule problem \cite{CDR},
 the art gallery problem \cite{SpTo}, and even in the description of 
unusual structural phenomena such as auxeticity in meta-materials (see \cite{BSaux2,BS_pseudo,BSaux1}, for example). 
 A survey of results can be found in \cite{RFS}.
Streinu \cite{Str05} and Haas et al. \cite{Haasetal} have shown the fundamental
result  that a graph can be realised as a pointed pseudotriangulation if and only if it is a planar Laman graph. 

We will prove symmetric versions of the above results in the 
case where the symmetry group is a cyclic group that is 
generated either by a rotation or a translation in the plane.

In the case of contact systems of segments, we must take care to 
specify carefully the appropriate combinatorial object that is 
analogous to the intersection graph. In the symmetric 
case, orbits of segments can have multiple intersections and can
 self-intersect so the graph that arises is naturally a multigraph.
Also, we must be careful about non-degeneracy conditions and so we require
a very slight modification of the definition of a \( 2 \)-contact system.
We explain the change and its relationship to the one used by Thomassen
in more detail below.

Furthermore it is not immediately obvious which classes of graphs
are analogous to the plane Laman graphs in the symmetric contexts. Once 
we identify the relevant classes, which are surface graphs satisfying certain
gain-sparsity counts, the main technical difficulty 
is to provide appropriate inductive characterisations of these
classes. These inductive characterisations are, we believe, 
of independent interest. They are analogous to a widely used result of 
Fekete, Jord\'an and Whiteley \cite{FJW} which gives an inductive characterisation
of plane Laman graphs.
 However, in our setting the proofs 
require some significant new ideas due to the 
more complicated topological setting.

In the case of pointed pseudotriangulations, we provide a natural
extension of the standard definition in the symmetric setting, and
 apply our inductive characterisations to establish symmetric
  versions of the result mentioned above of Streinu and Haas et. al.
  In particular, this allows us to gain new insights into the rigidity
  and flexibility properties of  bar-joint frameworks with rotational or
  translational symmetry in the plane.



We summarise the main results of the paper as follows:
\begin{enumerate}
    \item We characterise, in terms of gain sparsity properties, 
        the intersection graphs of 
        generic symmetric contact systems of line segments 
        in the case where the symmetry 
        group is generated by 
        a rotation of finite order or by a translation (Theorems \ref{thm_maintranslation}
        and \ref{thm_mainrotation}).
    \item We give an analogous combinatorial characterisation of the 
        graphs of symmetric pointed pseudotriangulations in the 
        case where the symmetry 
        group is generated by 
        a rotation of finite order or by a translation (Theorem~\ref{ppt_covering}).
    \item We show that the relevant gain-sparse surface graphs satisfy a 
    topological extension property, in the sense that they
     can always be completed to gain-tight surface graphs by adding appropriate edges
     (Proposition \ref{prop_232complete}).
    \item We give inductive characterisations based on 
        topological vertex splitting moves of the relevant 
        classes of gain-tight surface graphs (Theorems \ref{thm_cylinder_inductive232}
        and \ref{thm_cylinder_inductive231}).
    \item We show that a realisation of a planar graph as a bar-joint framework
     in the plane that is generic with $k$-fold rotational symmetry, $k\geq 3$,
   is minimally `forced-symmetric' rigid (i.e. 
  has no symmetry-preserving deformation)
   if and only if it can be realised as a pointed pseudotriangulation with this symmetry (Corollary~\ref{cor:rigidity}).
   \end{enumerate}

These results open up a number of obvious further research directions, such as
 possible extensions to other discrete subgroups of the Euclidean group, and we hope that this paper serves as an invitation for the reader to join in these explorations.
   
 \subsection{Comments on the presentation}
We have aimed for a relatively self-contained 
exposition, so some of the minor lemmas presented 
here with proofs are 
variations of known results. 
We have attempted to point out the relevant
literature in these cases.

Also because the paper draws
on concepts from several different 
parts of combinatorics, geometry and topology we find it 
expedient to briefly remind the reader of some elementary concepts 
and fix notation
in Section \ref{sec_terms}.

Finally, the proofs of the results from points 3 and 4 in the list 
above are 
quite long and technical. For that reason we have given 
precise statements of the results in Section \ref{sec_statements}
but deferred the proofs til later in order to present the main 
geometric applications first.

%

\section{Terminology and notation}
\label{sec_terms}

Here we fix some terminology and conventions regarding 
some standard notions of topological graph theory.

\subsection{Graphs}

A {\em graph} is a quadruple \( D = (V,E,s,t) \) where \( V,E \) are
sets (of vertices and edges respectively) and \( s,t \) 
are functions \( E \rightarrow V \). In the literature
such objects are sometimes referred to as multi-digraphs 
or quivers. We shall use graph instead and use adjectives
such as simple or loopless as appropriate. We note that 
graphs can be infinite but all graphs that arise in this paper
will be locally finite in the sense that any vertex will be incident to finitely many edges. If the graph \( D \) is not clear from the 
context we will write \( V(D) \), respectively \( E(D) \), for the
sets of vertices, respectively edges, of \( D \).
For \( V' \subset V \),  we write \( E(V') \) for 
the subset of \( E \) spanned by \( V' \) and 
\( D(V') = (V',E(V')) \) for the subgraph of \( D \) induced
by \( V' \).
Similarly for \( E'\subset E \)
we have \( D(E') = (V(E'),E')\)
where \( V(E') \) is the subset of \( V \) spanned by \( E' \).

The {\em geometric realisation} of \( D \) is 
\[ |D| = (E \times [0,1]) \sqcup V/\sim, \] where 
\( (e,0) \sim s(e) \) and \( (e,1) \sim t(e) \).
Throughout the paper we will often conflate vertices or
edges of \( D \) 
with the corresponding points or subsets of \( |D| \).
Connectivity properties of graphs will play an important 
role later so we specify our particular definitions  here 
carefully. Given a topological space \( X \) and a subset 
\( A \subset X \) we say that \( A \) separates points \( u,v \)
if \( u \) and \( v \) lie in the same path component of \( X \), 
\( u,v \not\in A \) and any continuous path joining \( u \) and \( v \) 
must  pass through \( A \). We will use this topological 
notion of separation both in the context of surfaces and graphs. 
For example a cutvertex of \( D \) 
will mean a vertex that separates any 
pair of points in \( |D| \). In particular, any vertex incident to a loop
edge is automatically a cutvertex.

%

\subsection{Surfaces and surface graphs}\label{sec:surgr}

A {\em surface} \( \Theta \) 
is a real two-dimensional manifold without boundary. 
We will be particularly concerned in later sections with the
open annulus \( \mathbb A = \mathbb R^2-\{(0,0)\} \).
We emphasise here that \( \mathbb A \) is to be thought of purely 
as a topological manifold. We will use different notation for the 
various geometric structures that have \( \mathbb A \) as the 
underlying topological manifold. We note that \( \mathbb A \) has 
two topological ends, one at zero and 
one at infinity. The location of these ends relative 
to various embedded graphs will be of importance later.

A {\em \( \Theta \)-graph} \( G \) is a pair \( (D,\Phi) \)
where \( \Phi:|D| \rightarrow \Theta \) is a continuous 
function that is a homeomorphism onto its image.
We will abuse terminology and refer to a subgraph \( H \) of \( G \) 
rather than a sub-\( \Theta \)-graph. In further abusive behaviour 
we will often conflate vertices and edges of \( D \) with 
their images under \( \Phi \).
We say that \( \Theta \)-graphs \( (D_i,\Phi_i), i = 1,2 \),
are isomorphic if there is a homeomorphism \( h:
\Theta \rightarrow \Theta\) and a graph isomorphism 
\( k: D_1\rightarrow D_2 \) such that \( h\circ \Phi_1 = 
\Phi_2 \circ |k|\) where \( |k|:|D_1|\rightarrow |D_2| \) is the
map induced by \( k \).

A {\em face} \( F \) 
of \( G \) is a component of \( \Theta - \Phi(|D|) \). 
In particular \( F \) is a connected open subset of \( \Theta \).
We say that \( F \) is {\em cellular} if it is homeomorphic to 
\( \mathbb R^2 \).
The {\em boundary \( \partial F \)} is the subgraph of \( G \)
comprising those vertices and edges that are contained within
the topological boundary of \( F \).
The face \( F \) has an associated family of closed {\em boundary 
walks}, one for each topological end of the face. {For a formal description of these walks in the cellular case (which can be readily adapted to the non-cellular setting), see Chapter 3 of \cite{MR1844449}.} 
We say that
\( F \) is {\em degenerate} if
there is either a repeated vertex or a repeated edge among 
all the boundary walks of \( F \) and is {\em non-degenerate}
otherwise. If \( F \) is cellular, the 
{\em degree} of \( F \), denoted \( |F| \),  
is the edge length of its unique boundary 
walk. In general \( |F| \geq |E(\partial F)|\), 
\( |F| \geq |V(\partial F)| \)
and one or both of these inequalities may be strict.
A cellular face of degree \( 3 \), respectively degree \( 4 \), is
called a {\em triangle}, respectively a {\em quadrilateral}.

\section{Contact systems of line segments}
\label{sec_config_line_seg}

A {\em contact system} of line segments in the plane 
is a collection of line segments
such that no point is an interior point of more than
one segment (see \cite{MR2177584} and \cite{MR1644051}). A 
\( k \)-contact system is a contact system such that any point belongs
to at most \( k \) segments. In this scheme the \( 2 \)-contact 
systems are in some sense the `least degenerate' and are 
thus a natural starting point for investigation. For our purposes
we introduce a slightly more restrictive definition as follows.
A collection of line segments in the plane is a 
{\em generic contact system} if no 
point is an interior point of more than one segment and 
no point is an endpoint of more than one segment. Observe 
that a generic contact system is necessarily a \( 2 \)-contact system.
On the other hand, we have the following.

\begin{lemma}
    Let \( \mathcal L \) be a \( 2 \)-contact system.
    There is a generic contact system 
    \( \mathcal L' \) (which can be chosen 
    to be arbitrarily close to \( \mathcal L \)) 
    and a bijection \( \mathcal L \rightarrow 
    \mathcal L', l \mapsto l' \) such that 
    \( l \cap m \neq \emptyset \Leftrightarrow l'\cap m'
    \neq \emptyset \) 
    \label{lem_perturbation}
\end{lemma}


{
\begin{proof}
    Suppose that $l, m \in \mathcal L$ are distinct and have a common endpoint $v$. 
    If they are not parallel, then we can extend the segment $l$ by an arbitrarily small 
    length to create a new segment $l'$ such that $l' \cap m$ is not an endpoint of $l'$.

    If $l$ and $m$ are parallel then we can perturb the common endpoint $v$ by a small amount 
    to create $l',m'$ that share a common endpoint but are not parallel. Note that when we do this we also
    must extend or truncate any other segments $k$ that have an endpoint in $l$ or $m$ to maintain that contact.
    As long as the perturbation of $v$ is sufficiently small this will be possible without creating 
    any new contacts between segments. Now we are in the situation of the previous paragraph and we extend $l'$
    as described there.
    
    By repeated applications of the perturbations described above we can find the desired generic contact system 
    $\mathcal L'$.
\end{proof}

}

Thus, if we are interested in the intersection graphs of such systems,
the slightly more restrictive definition of a generic contact system
versus that of a \( 2 \)-contact system is of no consequence. 
We state the aforementioned result of Thomassen in those terms.
Recall that a graph \( D = (V,E) \) is 
\( (2,3) \)-sparse if and only if for 
every non-empty \( E' \subset E \), 
\( |E'| \leq 2|V(E')|-3 \). 

\begin{thm}[Thomassen, \cite{Thomassen93}]
    \label{thm_thomassen1}
    A graph is the intersection graph of a 
    finite generic contact system of 
    line segments in the plane if and only if 
    it is \( (2,3) \)-sparse.
\end{thm}

{
Note that Lemma \ref{lem_perturbation} will need some modification in the symmetric case, which will be our main concern. Details will be given in Section \ref{sec_symm_config}.
}

\subsection{Embedding the intersection graph}

Let \( \mathcal L \) be a generic contact system.
The 
 intersection graph of \( \mathcal L \), which we denote
by \( I_\mathcal L \), has vertex set \( \mathcal L \) 
and directed edges corresponding to pairs \( (l,m) \) where the 
endpoint of \( l \) lies in the interior of \( m \).
We observe that there is some natural extra structure associated 
to \( I_\mathcal L \): it comes equipped with a plane embedding
as follows.
For each \( l \in \mathcal L \) 
we choose a subset \( c_l \subset l\) such that 
\begin{itemize}
    \item \( c_l \) is a closed sub-segment of \( l \) that does not 
        contain either of the endpoints of \( l \); 
    \item for every \( m \) that touches \( l \), \( c_l \)
        contains the point of contact (i.e the endpoint of \( m \)).
\end{itemize}
By construction, \( c_l \cap c_m = 
\emptyset\) for \( l \neq m \). Thus, if \( X \) is the quotient space of
\( \mathbb R^2 \) obtained by collapsing each \( c_l \) to a point 
\( v_l \), it follows that \( X \) is homeomorphic to \( \mathbb R^2 \).
The map \( l \mapsto v_l \) provides an embedding of the vertex set 
of \( I_\mathcal L \) in \( X \). 
If \( m \) touches \( l \) then the 
component of \( m - c_m \) that contains the point of 
contact maps to a path in \( X \) from \( v_m \) to \( v_l \). 
Thus we have an embedding 
\( |I_\mathcal L| \rightarrow X \) which we compose with the
homeomorphism \( X\rightarrow \mathbb R^2 \) to construct 
the desired plane embedding \( \psi: |I_\mathcal L|\rightarrow
\mathbb R^2\). Let \( G_\mathcal L \) be the plane 
graph \( (I_\mathcal L, \psi) \). See Figure \ref{fig_embedding}
for an illustration of this construction.
Of course 
\( G_\mathcal L \) depends on the particular choices of 
\( c_l \) for each \( l \) and on the choice of homeomorphism
\( X \rightarrow \mathbb R^2 \). However, it is not 
hard to see that the combinatorial embedding (see \cite{MR1844449}, Chapter 4 for definitions)
defined by this construction is uniquely characterised by the description
above. 

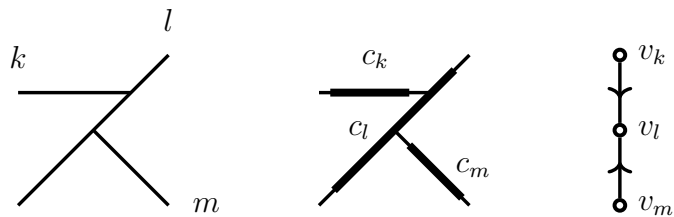
\begin{figure}[htp]
    \centering
  
\begin{tikzpicture}
    \begin{scope}
        \coordinate (A) at (0,0);
        \coordinate (B) at (2,2);
        \coordinate (C) at (2,0);
        \coordinate (D) at (0,2);
        \coordinate (E) at (0,1.5);
        \coordinate (F) at (2,1.5);

        \draw [segment, name path=ps] (A) -- (B) node[label=above:$l$]{};
        \path [name path=n] (C) -- (D);
        \path [name path=h] (E) -- (F);
        \draw [segment,name path=ns, name intersections={of=ps and n, by=x}] (x) --(C) node[label=right:$m$]{};
        \draw [segment,name path=hs, name intersections={of=ps and h, by=y}] (y) --(E)node[label=above:$k$]{};
    ;
    \end{scope}
    \begin{scope}[shift={(4,0)}]
        \coordinate (A) at (0,0);
        \coordinate (B) at (2,2);
        \coordinate (C) at (2,0);
        \coordinate (D) at (0,2);
        \coordinate (E) at (0,1.5);
        \coordinate (F) at (2,1.5);

        \draw [segment, name path=ps] (A) -- (B) node[midway,label=left:$c_l$]{};
        \path [name path=n] (C) -- (D);
        \path [name path=h] (E) -- (F);
        \draw [segment,name path=ns, name intersections={of=ps and n, by=x}] (x) --(C) node[midway,label=right:$c_m$]{};
        \draw [segment,name path=hs, name intersections={of=ps and h, by=y}] (y) --(E)node[midway,label=above:$c_k$]{};
    ;
        \draw [line width = 3pt] ($0.1*(B)+0.9*(A)$) -- ($0.1*(A)+0.9*(B)$);
        \draw [line width = 3pt] ($0.1*(x)+0.9*(C)$) -- ($0.2*(C)+0.8*(x)$);
        \draw [line width = 3pt] ($0.1*(y)+0.9*(E)$) -- ($0.2*(E)+0.8*(y)$);
    \end{scope}
    \begin{scope}[shift={(8,0)}]
        \node [vertex,label=right:{$v_m$}] (M) at (0,0){};
        \node [vertex,label=right:{$v_l$}] (L) at (0,1){};
        \node [vertex,label=right:{$v_k$}] (K) at (0,2){};
        \draw [edge,postaction={decorate}](M) -- (L);
        \draw [edge,postaction={decorate}](K) -- (L);
    \end{scope}

\end{tikzpicture}

\caption{The embedding of $I_\mathcal L$. On the left we have a contact system with segments $l,m,k$. In the centre we have indicated the 
subsegments $c_l,c_m,c_k$ in bold and on the right we have the embedding of the (directed) graph obtained by collapsing each of $c_l,c_m,c_k $ to a point.}
    \label{fig_embedding}
\end{figure}

\section{Symmetric contact systems}
\label{sec_symm_config}

{
    The main objects of interest in this paper are symmetric contact systems. These are contact
    systems that admit a group action induced by some group of symmetries of the plane.
    
Let \( \Gamma \) be a discrete subgroup of the 
Euclidean group of isometries of \( \mathbb R^2 \).
See \cite{Schatt,MR2410150} for a 
discussion of the classification of 
such groups. For $g \in \Gamma$ and $X \subset \mathbb R^2$ we write $g.X$ for the image of $X$ under $g$.
We consider  a contact system of line segments $\mathcal L$ such that
    \begin{itemize}
        \item [(S1)] $g.l \in \mathcal L$ for all $ l \in \mathcal L, g \in \Gamma$
    \end{itemize}
    In general we will seek analogues of Theorem \ref{thm_thomassen1} for various different symmetry groups.
    In order to make the problem more tractable we will impose some conditions relating to finiteness and genericity. In particular we assume that 
    \begin{itemize}
        \item[(S2)] 
        $\mathcal L$ has finitely 
            many $\Gamma$-orbits
    \end{itemize}

    Furthermore, we will assume that 
    \begin{itemize}
        \item[(S3)] $\mathcal L$ is generic.
    \end{itemize}

    Finally we consider an extra condition which is relevant only in the case that $\Gamma$
    does not act freely on $\mathbb R^2$.
    \begin{itemize}
        \item[(S4)] For all $l \in \mathcal L$ and $x \in l$ the stabiliser of $x$ in $\Gamma$ is trivial.
    \end{itemize}

    Given $\mathcal L$ and $\Gamma$ satisfying (S1-4) we say that $\mathcal L$ is a {\em 
        generic $\Gamma$-symmetric 
    contact system}. 
    
    Since the notion of 2-contact system is standard in much of the
    literature we wish to explore the relationship between this notion and 
    that of a generic contact system in the symmetric setting.
    It is clear that a generic $\Gamma$-symmetric contact system is, in particular, a 2-contact system. In the non-symmetric setting 
    Lemma \ref{lem_perturbation} provides a partial converse. In the symmetric
    setting, things are not quite so straightforward.
    
    The following lemma shows that in several cases of interest (S4) is redundant. Recall that $g \in \Gamma$ is primitive if $g = h^m \Rightarrow m = \pm 1$.
    \begin{lemma}
    \label{lem_13imply4}
        Suppose that $\mathcal L, \Gamma$ satisfy (S1), (S2). Furthermore 
        suppose that $\mathcal L$ is a 2-contact system and that $\Gamma$ does not contain a reflection or a primitive rotation of order 2. Then (S4) is also true.
    \end{lemma}
    \begin{proof}
        Suppose that $g.x = x$ for some non-identity element $ g \in \Gamma$,
        $x \in l$, $l \in \mathcal L$. Since $g$ is not a reflection, it must 
        be a rotation and we can assume without loss of generality that 
        $g$ is primitive. Thus $g$ has order at least 3 and it follows that 
        $l,g.l$ and $g^2.l$ are distinct elements of $\mathcal L$ that 
        all contain $x$. This contradicts the assumption that $\mathcal L$
        is a 2-contact system.
    \end{proof}
    Now we prove a symmetric analogue of Lemma 
    \ref{lem_perturbation}.
    \begin{lemma}
        Suppose that $\mathcal L, \Gamma$ satisfy (S1), (S2) and (S4) where $\mathcal L$ is a
        2-contact system. Furthermore suppose that 
        \begin{itemize}
            \item[(S5)] if $g.l \cap l \neq \emptyset$ for any $l \in \mathcal L $ and nonidentity element $g \in \Gamma$, then $g$ is not a translation. 
            \end{itemize}
        Then there is a generic $\Gamma$-symmetric contact system $\mathcal L'$ arbitrarily close to 
        $\mathcal L$ and a bijection $\mathcal L \rightarrow \mathcal L', l \mapsto l'$ such that 
        $l \cap m \neq \emptyset \Leftrightarrow l' \cap m' \neq \emptyset$.
        \label{lem_symm_perturbation}
    \end{lemma}

    \begin{proof}
        We will show how to adapt the argument for Lemma \ref{lem_perturbation}.
        Suppose that $l$ and $m$  are distinct segments in $\mathcal L$ that have a common endpoint. 

        Case 1: $l$ and $m$ lie in distinct 
        $\Gamma$ orbits. Then the perturbation argument of Lemma \ref{lem_perturbation} carries over to this 
        situation with the understanding that the perturbation is carried out for every element of the orbit of 
        a line segment, respecting the $\Gamma$-action. 

        Case 2: $m = g.l$ for some $g \in \Gamma$. Let $l \cap m = \{v\}$. Using (S4) we see that $v, g.v$ are 
        distinct endpoints of $m$. Using (S4) and (S5)
        we infer that $g$ must be either a rotation of order at least 3 or a glide reflection. If $g$ is a rotation of order at least 3 it follows that $m$ and $l$ are not parallel 
        and that the first perturbation described in 
        the proof of Lemma \ref{lem_perturbation} can be also applied in this situation.
        Finally if $g$ is a glide reflection, and $m$ is not parallel to $l$ then we we can extend one end of $l$ by an arbitrarily small length, and make corresponding extensions to all segments in $\Gamma.l$, so that $l$ and $m$ do not share an endpoint. If $l$ and $m$ are parallel then they must both be contained in the axis of the glide reflection $g$. Now, we can perturb the endpoints of $l$ symmetrically with respect to $g$, and make corresponding perturbations to all segments in $\Gamma.l$ and extending or truncating other segments to maintain all other contacts, so that $l$ and $m$ are not parallel and then proceed as before. 
    \end{proof}
    
    Assumption (S5) might seem a little awkward. However we have the following lemma.
    
    \begin{lemma}
    \label{lem_atleast3}
        Suppose that $\mathcal L, \Gamma$ satisfy (S1), $\mathcal L$ is a 2-contact system and $\Gamma$ contains a rotation of order at least 3. Then $\mathcal L, \Gamma$ satisfy (S5).
    \end{lemma}

    
    \begin{proof}
        Suppose that $g.l \cap l$ for some segment $l \in \mathcal L $ and translation $g \in \Gamma$. Then $M = \cup_{i \in \mathbb Z}g^i.l$ is a line in the plane. Now if 
        $h \in \Gamma$ is a rotation of order at least $3$ then $h.M \cap M$ is a single point 
        and it is clear that this contradicts the fact that $\mathcal L$ is a 2-contact system.
    \end{proof}
    
    Later we will focus on the cases where $\Gamma$ is a cyclic group, generated by either a translation or a rotation. 
    We note that in the case where $\Gamma$ is generated by a rotation of (finite) order at least 3,  Lemmas \ref{lem_13imply4},  \ref{lem_symm_perturbation} and \ref{lem_atleast3} allow us to interpret our characterisations of generic $\Gamma$-symmetric contact systems as characterisations of $\Gamma$-symmetric 
    2-contact systems. In the cases where $\Gamma$ is generated by a rotation of order 2 or a translation, things are complicated by the possibility that a symmetric 2-contact system could have a pair of segments that meet at a common endpoint. We defer discussion of such singular orbits to later work as it would significantly add to the length of the present paper and we also believe that our definition of a generic symmetric contact system is reasonably natural and worthy of investigation in its own right.
}



\subsection{The surface graph of a symmetric contact system}

The orbifold \( \mathbb R^2/\Gamma \) is a natural 
geometric object associated to the group \( \Gamma \).
Let \( \mathbf O \) be the 
set of non-singular points of  \( \mathbb R^2/\Gamma \). 
Explicitly 
\( \mathbf O \) is the image of the set of points 
with trivial stabiliser under the quotient map 
\( p: \mathbb R^2 
\rightarrow \mathbb R^2/\Gamma \). 
We observe that, geometrically, \( \mathbf O\) is a
flat surface (i.e. with constant curvature zero).
Let \( \Sigma \) be the 
underlying topological space of \( \mathbf O \). 
Later we will be particularly interested in the cases 
where \( \Gamma  \) is  generated by a translation or 
by a rotation and we note that in both of these cases
\( \Sigma \) is homeomorphic to \( \mathbb A \).

Now suppose that $\mathcal L$ is a generic \( \Gamma \)-symmetric contact system  in the plane. Using property (S4)
of $\mathcal L$
we see that 
each \( \Gamma \)-orbit of \( \mathcal L \) defines a local geodesic in 
\( \mathbf O \). For \( l \in \mathcal L \), 
let \( \overline l :[0,1]\rightarrow \mathbf O\) 
be a constant speed geodesic 
such that \(\overline l([0,1]) = p(l) \) and
let \( \overline{\mathcal L}  = \{ \overline l : l \in \mathcal L\}\).
We refer to \( \overline{\mathcal L} \) as the
contact system in \( \mathbf O \) corresponding to \(\mathcal L\). 

More generally, let \( \alpha:[0,1]\rightarrow \mathbf O \)
be a constant speed local geodesic. We say that \( x \) is 
a point of self intersection of \( \alpha \) if there 
exist \( t_1 \neq t_2 \) such that \( \alpha(t_1) = \alpha(t_2) \). 
Let \( \mathcal N \) be a finite set of constant speed geodesics in 
\( \mathbf O \). We say that \( \mathcal N \) is a 
generic contact system in \( \mathbf O \) if 
\begin{itemize}
    \item 
any point of intersection of 
\( \alpha \neq \beta \in \mathcal N \)
is an endpoint of precisely 
one of \( \alpha,\beta \) and is an interior point of 
precisely one of \( \alpha,\beta \), and
\item
any point of self intersection of 
\( \alpha \in \mathcal N \) occurs precisely once as an endpoint
of \( \alpha \) and precisely once as an interior point of 
\( \alpha \).
\end{itemize}

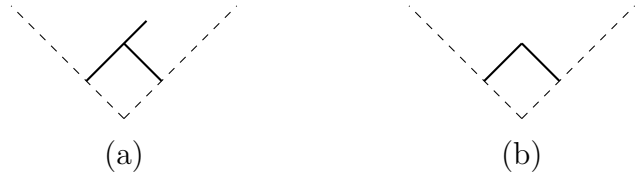
\begin{figure}[htp]
\begin{center}
  \begin{tikzpicture}[very thick,scale=1]
\tikzstyle{every node}=[vertex]
   \draw[dashed,thin] (0,0)  --  (1.5,1.5);
      \draw[dashed,thin] (0,0)  --  (-1.5,1.5);
        \draw[thick] (0.5,0.5)  --  (0,1);
      \draw[thick] (-0.5,0.5)  --  (0.3,1.3);
         \node [rectangle, draw=white, fill=white] (a) at (0,-0.5) {(a)}; 
       
              \end{tikzpicture}
        \hspace{2cm}
         \begin{tikzpicture}[very thick,scale=1]
 \tikzstyle{every node}=[vertex]
   \draw[dashed,thin] (0,0)  --  (1.5,1.5);
      \draw[dashed,thin] (0,0)  --  (-1.5,1.5);
        \draw[thick] (0.5,0.5)  --  (0,1);
      \draw[thick] (-0.5,0.5)  --  (0,1);
         \node [rectangle, draw=white, fill=white] (b) at (0,-0.5) {(b)}; 
        \end{tikzpicture} 
                \caption{{Two contact systems  in \( \mathbf O \) where \( \mathbf O \) is a flat cone. Each of the contact systems consists of a single segment which self-intersects. The contact system in (a) is generic, whereas the one in (b) is not.}}
 \label{fig:gen_ln_seg}               
\end{center}
\end{figure}

{See Figure~\ref{fig:gen_ln_seg} for some examples.}

Given \( \alpha \in \mathcal N \)
there is a \( \Gamma \)-invariant collection of 
line segments in the plane, \( \mathcal L_\alpha \), such that 
\( p^{-1}(\alpha([0,1])) = \cup_{l \in \mathcal L_\alpha}l\)
and \( p(l) = \alpha([0,1]) \) for each \( l \in \mathcal L_\alpha \).
Let \( \tilde{\mathcal N} = 
\cup_{\alpha \in \mathcal N} \mathcal L_\alpha\).
In the case \( \mathcal N = \overline{\mathcal L} \)
it is clear that \( \mathcal L = \tilde{\mathcal N} \).
Indeed the  following lemma is a straightforward observation 
concerning the definitions of a generic \( \Gamma \)-symmetric 
contact system and a contact system in \( \mathbf O \).

{
\begin{lemma}
    The mappings \( \mathcal L \mapsto \overline{\mathcal L}  \)
    and \( \mathcal N \mapsto \tilde{\mathcal N}\) are mutually
    inverse bijections between the set of generic \( \Gamma \)-symmetric
    contact systems in the plane and the set of generic contact
    systems in \( \mathbf O \).
    \label{lem_lifting}
\end{lemma}

\begin{proof}
Suppose that $\mathcal L$ is a generic $\Gamma$-symmetric contact system the plane. Suppose that $\alpha, \beta \in \overline{\mathcal L}$, $\alpha \neq \beta$ and $x$ is a point of intersection of $\alpha$ and $\beta$. Using property (S3) of $\mathcal L$ it follows that $x$ is an endpoint of precisely one of $\alpha, \beta$ and an interior point of the other. 
If $x$ is a point of self-intersection of $\alpha$
then for some $l \in \mathcal L$, $g \neq e \in \Gamma$ we must have $\alpha([0,1]) = p(l) = p(g.l)$. Since $g \neq e$, we have $g.l \neq l$ 
and again using (S3) we see 
$l$ and $g.l$ must intersect at an point that is 
an endpoint of precisely one of $l, g.l$. Thus
$x$ is both an endpoint and an interior point of 
$\alpha$ as required. The finiteness of $\overline{\mathcal L}$ follows from 
(S2). Thus $\overline{\mathcal L}$ is a generic contact system $\mathbf O$.

On the other hand suppose that $\mathcal N$ is a generic contact system in 
$\mathbf O$. It is clear from the construction of $\tilde{\mathcal N}$ that 
it satisfies (S1), (S2) and (S4). Suppose that $l,m \in \tilde{\mathcal N}$
have a point in common and that $\alpha([0,1]) = p(l)$ and $\beta([0,1]) = p(m)$ for $\alpha,\beta \in \mathcal N$. Then $\alpha, \beta$ have a point of 
intersection which is an endpoint of precisely one of $\alpha, \beta$. It follows that the
point of intersection of $l,m$ is an endpoint of precisely one of $l,m$. 
Thus $\tilde{\mathcal N}$ satisfies (S3). The fact that $\mathcal N$
is finite implies that $\tilde{N}$ satisfies (S2). 
Therefore $\tilde{\mathcal N}$ is a generic $\Gamma$-symmetric contact 
system as required. 
\end{proof}
}

Now given a contact system \( \mathcal N \) in \( \mathbf O \)
we can define a graph \( I_\mathcal N \) with vertex set \( \mathcal N \) and edges corresponding to quadruples  \( (\alpha,\beta,x,y) \) 
where \( \alpha,\beta \in \mathcal N \), \( x \in \{0,1\} \), 
\( y \in (0,1) \) and \( \alpha(x) = \beta(y) \). Here
we allow \( \alpha = \beta \) and moreover it is possible that 
we could have distinct edges \( (\alpha,\beta,x_1,y_1) \)
and \( (\alpha,\beta,x_2,y_2) \). (Note 
that for some \( \mathbf O \) this can happen 
even if \( \alpha = \beta \).)

Now
we define a \( \Sigma \)-graph 
whose underlying graph is \( I_\mathcal N \) as follows.
For each \( \alpha \in \mathcal N \) choose a non-empty closed interval 
\( c_\alpha \subset (0,1) \) 
such that \( y \in c_\alpha \) 
for every edge 
\( (\beta,\alpha,x,y) \). 
Now collapse each \(\alpha(c_\alpha)  \) to a single point
\( p_\alpha \). The resulting
quotient space of \( \mathbf O \) is homeomorphic to
\( \Sigma \). The map \( \alpha \mapsto p_\alpha\)
gives an embedding of the vertex set of \( I_\mathcal N \)
and we use the restriction of \( \alpha \) to the appropriate
component of \( [0,1]-c_\alpha \) to construct embeddings 
of the  edges of \( I_\mathcal N \). 
Let \( G_\mathcal N \) denote the resulting \( \Sigma \)-graph.

So if \( \mathcal L \) is a generic \( \Gamma \)-symmetric
contact system in the plane, then \( \Gamma \) acts 
by directed graph automorphisms on \( I_\mathcal L \) and 
it is easy to see that \( I_{\mathcal L}/\Gamma \) is 
canonically isomorphic to \( I_{\overline {\mathcal L}} \).
Now let 
\( \tilde \Sigma = 
\{ x \in \mathbb R^2: \text{Stab}_\Gamma(x) = 1_\Gamma\}\).
It is well known that the restriction \( p:\tilde\Sigma
\rightarrow \Sigma\) is a regular covering projection.
Using standard results of covering space theory it follows
that we can choose the embedding \( \psi: |I_\mathcal L|
\rightarrow \tilde \Sigma\) so that the following diagram
commutes
\begin{equation}\label{diag_covering}
        \begin{tikzcd}
            {|I_\mathcal L|} \arrow[r,"\psi"] \arrow[d]& \tilde \Sigma \arrow[d] \arrow[r,hook] & \mathbb R^2 \arrow[d,"p"]\\
            {|I_{\overline{\mathcal L}}|} \arrow[r,"\overline\psi"]& \Sigma \arrow[r,hook]& \mathbb R^2/\Gamma
        \end{tikzcd}
\end{equation}
where \( \overline \psi:|I_{\overline {\mathcal{L}}} |\rightarrow \Sigma \)
is the embedding constructed above.
We note that the left and middle vertical arrows in 
(\ref{diag_covering}) represent regular covering projections.

In summary a generic \( \Gamma \)-symmetric contact
system, \( \mathcal L \), gives rise to a surface graph, denoted \(
G_{\overline{\mathcal L}}\), which describes the combinatorial 
structure of the contact system. 
See Figure \ref{fig_examples} for some examples of generic symmetric contact 
systems and their corresponding surface graphs.

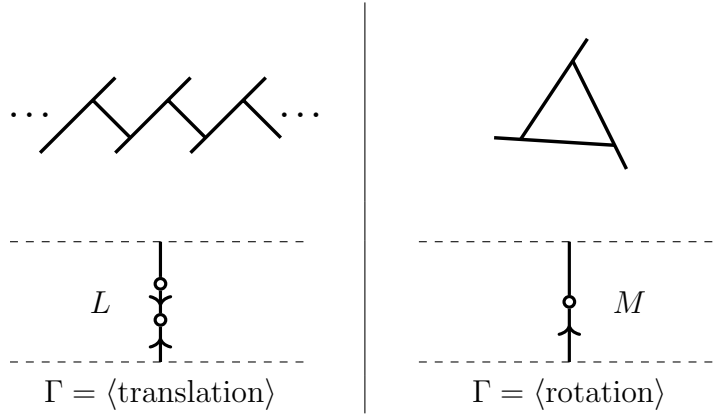
\begin{figure}[htp]
    \centering

    \begin{tabular}{c|c}

\begin{tikzpicture}
    \coordinate (BL) at (-2.5,-1);
    \coordinate (TR) at (2.5,1.5);

    \clip (BL) rectangle (TR);
    \node at (-1.7,0){\bf \dots};
    \foreach \x in {-1.6,-.6,0.4}{
        \draw [segment](\x,-0.5) -- +(1,1) ++(0.7,0.7) -- +(0.5,-0.5) ;
    }
    \node at (1.9,0){\bf \dots};
\end{tikzpicture}
&

\begin{tikzpicture}
    \clip (BL) rectangle (TR);
        \foreach \x in {0,120,240}{
            \draw [segment,rotate=\x] (-1,-0.3) -- (0.6,-0.4);
        }
\end{tikzpicture}

\\





\\

\begin{tikzpicture} [scale=0.8]
    \tikzstyle{every node}=[circle, draw=black,  fill=white, inner sep=0pt, minimum width=4pt];
    \begin{scope}[very thick,decoration={
    markings,
    mark=at position 0.7 with {\arrow{>}}}
    ] 
    \draw[dashed,thin] (-2.5,-0.5)-- (2.5,-0.5);
    \draw[dashed,thin] (-2.5,1.5)-- (2.5,1.5);
        \path (0,0.2) node (p1) {} ;
        \path (0,0.8) node (p2) {} ;
      \draw[postaction={decorate}] (p2)--(0,0.5)--(p1);
      \draw[postaction={decorate}] (0,-0.5)--(p1);
      \draw(0,1.5)--(p2);
      \end{scope}
                              \node [draw=white, fill=white] (a) at (-1,0.5) {$L$};
                              
      \end{tikzpicture}

&

\begin{tikzpicture} [scale=0.8]
    \tikzstyle{every node}=[circle, draw=black,  fill=white, inner sep=0pt, minimum width=4pt];
    \begin{scope}[very thick,decoration={
    markings,
    mark=at position 0.7 with {\arrow{>}}}
    ] 
    \draw[dashed,thin] (-2.5,-0.5)-- (2.5,-0.5);
    \draw[dashed,thin] (-2.5,1.5)-- (2.5,1.5);
        \path (0,0.5) node (p1) {} ;
          \draw[postaction={decorate}] (0,-0.5)--(p1);
      \draw(0,1.5)--(p1);
      \end{scope}
                              \node [draw=white, fill=white] (a) at (1,0.5) {$M$};
                              
      \end{tikzpicture}
\\ 

$\Gamma = \langle \text{translation} \rangle$ 
&
$\Gamma = \langle \text{rotation} \rangle$ 
        
    \end{tabular}

    \caption{{Two examples of generic symmetric contact systems (top row) and their corresponding surface graphs (second row). In both cases the surface \( \Sigma \) is homeomorphic to \( \mathbb A \) which we represent topologically by a horizontal strip with top and bottom edges (the dotted lines) identified. The \( \mathbb A \)-graphs \( L \), respectively $M$, on the left, respectively right, arise as base graphs in the inductive characterisations described in Theorems \ref{thm_cylinder_inductive232} and \ref{thm_cylinder_inductive231}.}}
    \label{fig_examples}
    \label{fig_basecases}
\end{figure}

\section{Gain sparsity counts}
\label{sec_nec}

For the remainder of the paper we will specialise to the 
case where the symmetry group \( \Gamma \) is cyclic and 
orientation preserving. So \( \Gamma \) is either 
generated by a rotation or a translation. In either of these 
cases \( \Sigma \) is homeomorphic to 
the punctured plane \( \mathbb A = \mathbb R^2 -\{(0,0)
\}\) so from now on we will primarily be concerned 
with properties of an \( \mathbb A \)-graph, that is to say 
a graph together with an embedding of its geometric realisation in 
\( \mathbb A \).

For a graph \( D = (V,E) \) we define \[ f(D) = 2|V| - |E|. \] 
Thus \( D \) is \( (2,3) \)-sparse
if and only if \( f(C) \geq 3 \) for every subgraph \( C \) of \( D \) that 
contains at least one edge. If, in addition, \( f(D) = 3 \) or \( D \) is 
an isolated vertex, we say that \( G \) is \( (2,3) \)-tight or is a Laman graph.

Now suppose that \( G \) is an \( \mathbb A \)-graph. 
We say that 
\( G \) is {\em balanced} if some face of \( G \) contains 
both ends of \( \mathbb A \), and {\em unbalanced} otherwise. 
If \( F \subset E(G) \) then we say that \( F \) is 
balanced, respectively unbalanced, if \( G(F) \) is 
balanced, respectively unbalanced.

Suppose that \( l \in \{1,2\} \). We say that 
\( G \) is {\em \( (2,3,l) \)-sparse} if 
\( f(H) \geq l \) for every subgraph \( H \) of \( G \),
and \( f(K) \geq 3 \) for every balanced subgraph \( K \) 
of \( G \) with at least one edge. 
If in addition, either \( f(G) = l \), or 
\( G \) is balanced and \( f(G) =3\), or \( G \) is an 
isolated vertex, then we say that \( G \) is
{\em \( (2,3,l) \)-tight}. 
On the other hand if \( H \) is a subgraph of \( G \) 
such that either \( f(H) <l \), or, \( H \) is balanced, has an
edge and \( f(H) <3 \), then we say that \( H \) 
{\em violates the \( (2,3,l) \)-sparsity count}.
Since any subgraph of a balanced graph is also balanced 
it is 
clear that  a balanced \( \mathbb A \)-graph is \( (2,3,l) \)-tight
if and only if it is a Laman graph.
\label{rem_gain1}

\begin{figure}[htp]
    \centering

\begin{tabular}{ccc}
\begin{tikzpicture}
\tikzstyle{every node}=[circle, draw=black,  fill=white, inner sep=0pt, minimum width=4pt];
    \clip (-0.2,-0.2) rectangle (3.2,2.2);
    \draw[cylinder] (0,0) -- (3,0) (0,2) -- (3,2);
    \coordinate [vertex] (v1) at (1.5,1.7);
    \coordinate [vertex] (v2) at (2,1);
    \coordinate [vertex] (v3) at (1.5,0.3);
    \coordinate [vertex] (v4) at (1,1);
    \draw[edge] (v1) -- (v2) -- (v3) -- (v4) -- (v1) (v2) -- (v4); 
      
\end{tikzpicture}
&
\begin{tikzpicture}
\tikzstyle{every node}=[circle, draw=black,  fill=white, inner sep=0pt, minimum width=4pt];
    \clip (-0.2,-0.2) rectangle (3.2,2.2);
    \draw[cylinder] (0,0) -- (3,0) (0,2) -- (3,2);
    \coordinate [vertex] (v1) at (1.5,0);
    \coordinate [vertex] (v2) at (2,1);
    \coordinate [vertex] (v4) at (1,1);
    \coordinate [vertex] (v3) at (1.5,1.5);
    \draw[edge] (v1) -- (v2);
    \draw[edge] (1.5,1.5) -- (1.25,2);
    \draw[edge] (1.25,0) -- (v4);
    \draw[edge] (v3) -- (v2);
    \draw[edge] (v3) -- (v4);
    \draw[edge] (v4) -- (v1);
    \draw[edge] (v2) -- (1.5,2);
     
\end{tikzpicture}
&
\begin{tikzpicture}
\tikzstyle{every node}=[circle, draw=black,  fill=white, inner sep=0pt, minimum width=4pt];
    \clip (-0.2,-0.2) rectangle (3.2,2.2);
    \draw[cylinder] (0,0) -- (3,0) (0,2) -- (3,2);
    \coordinate [vertex] (v1) at (1.5,0);
    \coordinate [vertex] (v2) at (2,1);
    \coordinate [vertex] (v4) at (1,1);
    \coordinate [vertex] (v3) at (1.5,2);
    \draw[edge] (v1) -- (v2);
    \draw[edge] (v4) -- (1,2);
    \draw[edge] (v4) -- (1,0);
    \draw[edge] (v3) -- (v2);
    \draw[edge] (v3) -- (v4);
    \draw[edge] (v4) -- (v1);
      
   \end{tikzpicture}
   \\(a)  & (b) & (c)
\end{tabular}

\caption{{Examples of \( (2,3,l) \)-tight \( \mathbb A \)-graphs (with directions of edges omitted). (a) is   \( (2,3,l) \)-tight for $l=1,2$ since it is balanced and \( (2,3)\)-tight. (b) and (c) are unbalanced and \( (2,3,2) \)- and  \( (2,3,1) \)-tight, respectively.}}
    \label{fig_gain_sparsity_ex}
    \end{figure}
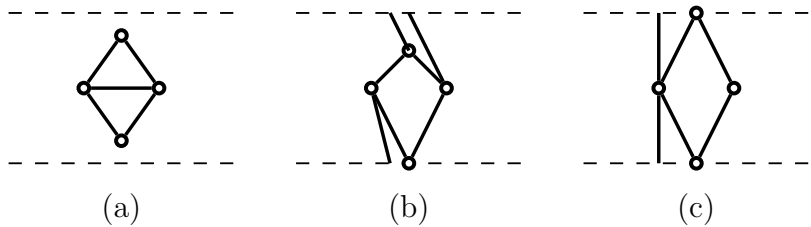

\begin{remark}
    Those familiar with gain graphs
    will  observe that our definition of a balanced 
    \( \mathbb A \)-graph and subsequent definition of 
    \( (2,3,l) \)-sparsity 
    are particular cases of more general
    notions. See, for example, \cite{Zas89},  \cite{ST1},
    \cite{JKT} and \cite{BS10}.
    \label{rem_gain2}
\end{remark}

{For some examples, see Figure~\ref{fig_gain_sparsity_ex}. We can also consider the \( \mathbb A\)-graphs
shown in Figures \ref{fig_basecases}, \ref{fig_degenquads} and 
\ref{fig_231examples}.} In these diagrams, and elsewhere in the paper,  
we represent \( \mathbb A \) 
by a horizontal strip with top and bottom edges identified. {Moreover, the directions of the edges of the \( \mathbb A\)-graphs are removed in the diagrams, as they are irrelevant to the sparsity properties of the \( \mathbb A\)-graphs.}
Specifically the three \(\mathbb A \)-graphs shown in Figure \ref{fig_231examples} are all \( (2,3,1) \)-tight.
Of the \( \mathbb A \)-graphs in Figure \ref{fig_degenquads}, (a) is \( (2,3,2) \)-tight, (b) and (d) are \( (2,3,1) \)-tight and (c) is \( (2,3,1) \)-sparse but not tight.
Finally we note that the \( \mathbb A \)-graph \( L \), respectively \( M \), from Figure \ref{fig_basecases} is \( (2,3,2) \)-tight, respectively \( (2,3,1) \)-tight.

Now we give the statements of our main results for 
generic symmetric contact systems.

\begin{theorem}
    \label{thm_maintranslation}
    Let \( \Gamma \) be the subgroup of the Euclidean group
    generated by a translation or by a rotation of order \( 2 \).
    An \( \mathbb A \)-graph \( G \) is the graph of
    a generic \( \Gamma \)-symmetric contact system of line segments
    in the plane
    if and only if \( G \) is \( (2,3,2) \)-sparse.
\end{theorem}

\begin{theorem}
    \label{thm_mainrotation}
    Let \( \Gamma \) be a subgroup of the Euclidean group
    generated by a rotation of order at least \( 3 \).
    An \( \mathbb A \)-graph \( G \) is the graph of
    a generic \( \Gamma \)-symmetric contact system of line segments
    in the plane
    if and only if \( G \) is \( (2,3,1) \)-sparse.
\end{theorem}

\subsection{Necessity of the gain sparsity counts}

In the remainder of this section we show that the contact 
systems in Theorems \ref{thm_maintranslation} 
and \ref{thm_mainrotation} have graphs with the required sparsity 
properties. 
First, since we need it later and to make our presentation more self-contained,  we give a proof of the corresponding part of
Thomassen's result in the non-symmetric case 
(Theorem \ref{thm_thomassen1}). 

Let \( \mathcal L \) be a generic contact system
of line segments in the plane and 
let \( deg^+(l) \) denote the outdegree 
of a vertex \( l  \in V(I_\mathcal L) = \mathcal L\). 
We say that an endpoint of 
\( l \) is \emph{free} if it does not lie in the interior of 
any other segment. 
Thus \( 2-deg^+(l) \) is the number of free endpoints of 
\( l \). So \( f(I_\mathcal L) = 2|V| - |E| = \sum_{v \in V} (2-deg^{+}(v))\) is the total number of free endpoints in \( \mathcal L \).

\begin{thm} \label{thm:necplane}
    If \( \mathcal L \) is a non-empty collection of line segments in the
    plane, then \( I_\mathcal L \) is \( (2,3) \)-sparse.
\end{thm}

\begin{proof}
    Note that it suffices to check the sparsity condition for all induced
    subgraphs $(V,E)$ of \(I_\mathcal L\) (i.e. subgraphs corresponding to 
    subsets \(\mathcal L'\) of \(\mathcal L\)) with at least one edge. 
    Thus we must show that there are at least three free endpoints in 
    \( \mathcal L' \). 

    Choose a segment \( l \in \mathcal L'\).
    Without loss of 
    generality we may assume that \( l \) is parallel to the \( y \)-axis:
    just rotate the entire configuration until that is true.
    Now pick some interior point of \( l \) as a starting point 
    and move downwards along \( 
    l\) until we come to the endpoint. If it is free then we have found 
    a free endpoint. If it is not free then it belongs to the 
    interior of some 
    other segment \( m \) since \( \mathcal L' \) is a 
    generic contact system. Move along \( m \) in a direction that 
    does not increase the \( y \)-coordinate and continue in this way. 
    Eventually we must arrive at our first free endpoint \( p_1 \). 
    By applying the same argument but moving upward from our starting point
    we find another free endpoint \( p_2 \). Now \( p_1 \neq p_2 \) since
    the \( y \)-coordinate of \( p_1 \), respectively \( p_2 \),
    is strictly less, respectively greater, 
    than that of the starting point.

    Now let \( M \) be the line containing \( p_1 \) and 
    \( p_2 \). Since \( E \) is not empty, not every line segment in 
    \( \mathcal L' \) is contained in \( M \). Thus there is a point
    in some segment of \( \mathcal L' \) that is not in \( M \). Start
    at this point and move along segments always in a direction that
    does not decrease the perpendicular distance to \( M \). Eventually
    we must arrive at a free endpoint \( p_3 \) that does not lie in 
    \( M \) and therefore is not \( p_1 \) or \( p_2 \).
\end{proof}

\subsection{The symmetric cases}

Now suppose that \( \Gamma \) is generated by a 
single rotation or a single translation and let 
\( \mathcal L \) be a \( \Gamma \)-symmetric contact 
system of line segments. As noted above 
\( \Sigma \equiv \mathbb A \).

\begin{proposition}
    \label{prop_balanced}
    Suppose that \( G_{\overline{\mathcal L}} \) is a balanced
    \( \mathbb A \)-graph. Then there is a transversal 
    \( \mathcal F \)
    of the \( \Gamma \)-orbits of \( \mathcal L \)
    such that \( (l,m) \in E(I_\mathcal L) \) if and
    only if there is some \( g \in \Gamma \) such that 
    \( (g.l,g.m) \in E(I_\mathcal F) \). In particular
    \( I_{\overline{\mathcal L}} \cong I_\mathcal F \).
    \label{prop_section}
\end{proposition}

\begin{proof}
    Since \( G_{\overline{\mathcal L}}  = (I_{\overline{\mathcal L}},
    \overline \psi)\) is a balanced \( \mathbb A \)-graph,
    it is clear that \( \overline\psi: |I_{\overline{\mathcal L}}|
    \rightarrow \mathbb A\) 
    induces a trivial homomorphism
    of fundamental groups. 
    Using the commutativity of the left hand square of 
    (\ref{diag_covering}) and some 
    standard results of 
    covering space theory it follows that the covering 
    projection \( |I_\mathcal L| \rightarrow |I_{\overline{\mathcal L}}| \)
    has a global 
    section \( \sigma: | I_{\overline{\mathcal L}}| 
    \rightarrow |I_\mathcal L| \). Now \( \mathcal F = \sigma(V(
    I_{\overline{\mathcal L}}))\) yields the required transversal of the 
    orbits of \( \mathcal L \).
\end{proof}

\begin{corollary}
    \label{cor_balanced}
    Suppose that \( \Gamma  = \langle z \rangle\) where \( 
    z\) is either a translation or a rotation and that 
    \( \mathcal L \) is a generic \( \Gamma \)-symmetric
    contact system of line segments in the plane. If \( H \) is a balanced
    subgraph of \( G_{\overline{\mathcal L}} \) that 
    contains at least one edge then \( f(H) \geq 3 \).
\end{corollary}

\begin{proof}
    If \( H \) is an induced 
    subgraph then it is clear 
    that \( H = G_{\overline{\mathcal L'}} \)
    where \( \mathcal L' \) is some \( \Gamma \) invariant
    subset of \( \mathcal L \). If \( H \) is not induced 
    then it can be obtained from some \( G_{\overline{\mathcal L'}} \)
    by shortening some of the segments in \( \mathcal L' \)
    so as to remove the necessary edges. In either 
    case \( H = G_{\overline{\mathcal M}} \) 
    for some generic \( \Gamma \)-symmetric
    contact system \( \mathcal M \). 
    Using Proposition \ref{prop_balanced}, since \( H \) is balanced, 
    we see that 
    \( H \cong G_\mathcal F \) for some finite contact 
    system 
    \( \mathcal F \). By Theorem \ref{thm:necplane} it follows that 
    \( f(H) \geq 3 \).
\end{proof}

To identify the necessary sparsity condition for unbalanced subgraphs, we first observe that
\( f(G_{\overline{\mathcal L}}) \) is equal to the
number of \( \Gamma \)-orbits of free ends in \( \mathcal L \).

\begin{thm}
    \label{thm:neccyl}
    Suppose that \( \Gamma \) is generated by a translation and 
    let \( \mathcal L \) be a  
    generic \( \Gamma \)-symmetric contact
    system of line segments in the plane. Then the graph 
    \( G_{\overline{\mathcal L}} \) is \( (2,3,2) \)-sparse.
\end{thm}

\begin{proof}
    By Corollary \ref{cor_balanced} it suffices to show that 
    \( f(K) \geq 2 \) for any non-empty 
     subgraph \( K \) of \( 
    G_{\overline{\mathcal L}}\).
    Clearly we can assume that 
    \( K \) is an induced subgraph.
    Also we can assume that 
    the generator of \( \Gamma \) is the translation 
    \( (x,y) \mapsto (x+1,y) \). 
    
    Let \( \mathcal M \) be the \( \Gamma \)-invariant
    subset of \( \mathcal L \) that corresponds to \( V(K) \).
    If all line segments in 
    \( \mathcal M \) are horizontal then, since we assume that 
    \( \mathcal L \) is generic, it is clear that 
    \( K \) has no edges and since it has 
    at least one vertex, it follows that \( f(K)
    \geq 2\).

    So we may assume that some segment \( l \in \mathcal M \)
    is not horizontal. Now starting at an interior point 
    of \( l \) we can move along segments so that the 
    \( y \)-coordinate is non-decreasing. 
    Since \( \mathcal L \)
    must be contained within some horizontal strip, as 
    it has finitely many \( \Gamma \)-orbits and each
    orbit is bounded in the \( y \)-direction, we must
    eventually arrive at a free endpoint.
    Similarly there is another free endpoint obtained by moving away
    from the starting point along line segments so that 
    the \( y \)-coordinate is non-increasing. These 
    two free endpoints do not lie in the same orbit of 
    \( \Gamma \) since they have different \( y \)-coordinates.
    Thus \( f(K) \geq 2 \) as required.
\end{proof}

Next we consider the case where \( \Gamma \) is generated
by a rotation. 

\begin{thm}
    \label{thm_rotation}
    Suppose that \( \Gamma \) is generated by a rotation of 
    order \( n \) and that \( \mathcal L \) is a 
    generic \( \Gamma \)-symmetric contact system of 
    line segments in the plane. Then 
    \begin{enumerate}
        \item \(  G_{\overline{\mathcal L}} \) is \( (2,3,1) \)-sparse.
        \item if \( n = 2 \) then \( G_{\overline{\mathcal L}}\) is \( (2,3,2) \)-sparse.
    \end{enumerate}
\end{thm}

\begin{proof}
    As in the proof of Theorem \ref{thm:neccyl} 
    let \( K \) be a non-empty induced subgraph of \( 
    G_{\overline{\mathcal L}}\) and let \( \mathcal M \) be the 
    \( \Gamma \)-invariant subset of \( \mathcal L \) 
    that corresponds to \( V(K) \). We must show that 
    \( \mathcal M \) has at least two free endpoints that 
    lie in different \( \Gamma \)-orbits.
    
    Let \( o \) be the fixed point of \( \Gamma \). 
    Choose \( l \in \mathcal M \). Starting at an interior point
    of \( l \), move along segments in \( \mathcal M \) 
    so that the distance to \( o \) is always increasing. 
    Since \( \mathcal M \) lies inside some bounded region 
    of the plane 
    we must eventually arrive at a free endpoint, \( p \). 
    Thus \( f(K) \geq 1 \). This proves the first statement.

    Now suppose that \( n =2 \) and let \( L \) be 
    the line containing the points \( p \) and \( o \). 
    If \( K \) has no edges then, since it has 
    at least one vertex, we have \( f(K) \geq 2 \). On the other hand,
    if \( K \) has an edge then not every segment in \( \mathcal M
    \) is contained in \( L \). Starting at some point in a 
    segment of \( \mathcal M \) that is not in \( L \), 
    move along segments in such a way that the perpendicular distance
    to \( L \) is non-decreasing. Again we must eventually arrive
    at some free endpoint \( q \). Now since \( q \not\in L \) and 
    \( \Gamma.p \subset L \) (here is where we use the 
    \( n=2 \) hypothesis) we have found two free 
    endpoints that lie in different \( \Gamma \) orbits. 
    Thus \( f(K) \geq 2 \) in this case. 
\end{proof}

\section{Properties of \( (2,3,l) \)-tight \( \mathbb A \)-graphs: statements}
\label{sec_statements}

In the previous section we established one direction of 
Theorems \ref{thm_maintranslation} and \ref{thm_mainrotation}.
In order to establish the other direction we need to investigate
various properties of \( (2,3,l) \)-sparse graphs. In this section 
we present the statements of the necessary results.
Since the proofs are quite long and are
essentially independent of the geometric applications, we defer those til later.

Let \( \Theta \) be a surface and
suppose that \( G \) is a \( \Theta \)-graph with a triangular face \( T \) and 
an edge \( e = uv \) that belongs to the boundary walk of \( T \). We also suppose that 
\( u \neq v \). Let \( G_{e,T} \) be the \( \Theta \)-graph obtained by collapsing 
\( |e| \) (the image of \( e \) in \( \Theta \)) to a point and deleting one of the other
edges of the facial walk of \( T \). 
We say that \( G_{e,T} \) is obtained from \( G \) by 
a topological contraction of \( T \) along \( e \). 
On the other hand we say that \( G \) 
is obtained from \( G_{e,T} \) by a triangular vertex split.
It is to be emphasised that, apart from the case explicitly 
specified in the definition, parallel edges or loop edges that 
are created by the edge contraction are retained in \( G_{ e,T} \).
In \cite{FJW} Fekete, Jord\'an and Whiteley prove 
the following inductive characterisation of plane Laman graphs.

\begin{theorem}[Fekete, Jord\'an and Whiteley]
    Suppose that \( G \) is a plane Laman graph with at least \( 3 \)
    vertices. 
    Then \( G \) can be constructed from a single edge 
    by a sequence of triangular vertex splits.
    \label{thm_FJW}
\end{theorem}

We would like to prove results analogous to Theorem 
\ref{thm_FJW} for \( (2,3,l) \)-tight \( \mathbb A \)-graphs
for \( l\in \{1,2\} \). 
However 
it is easy to see that there are 
infinitely many pairwise non-isomorphic \( (2,3,l) \)-tight
\( \mathbb A \)-graphs that have no triangular faces. Thus we will need to consider an 
additional contraction move to deal with quadrilateral faces. 

Suppose that \( Q \) is a quadrilateral face of 
\( G \) with boundary walk 
\( v_1\), \(e_1\), \(v_2\), \(e_2\), \(v_3\), \(e_3\), \(v_4\), \(e_4\), \(v_1 \) such that \( 
v_1 \neq v_3\). Let \( \delta \) be a Jordan arc joining 
\( v_1 \) and \( v_3 \) whose relative interior lies in 
\( Q \). We can view \( \delta \) as a (topological) edge 
that is not in \( G \). We will refer to \( \delta \) 
as the diagonal of \( Q \) joining \( v_1 \) and \( v_3 \).
Let \( G_{v_1,v_3,Q} \) be the 
\( \Theta \)-graph obtained 
from \( G \cup \delta \) by contracting \( \delta \) 
to a point and then 
deleting one of the edges \( e_1,e_2 \) and one of the 
edges \( e_3,e_4 \). 
We call \( G_{v_1,v_3,Q} \) a quadrilateral contraction of 
\( G \). 
On the other hand we say that \( G \) is obtained from 
\( G_{v_1,v_3,Q} \) be a quadrilateral vertex split.
Again we emphasise that,  apart from the cases explicitly 
specified in the definition, parallel edges or loop edges that 
are created by the contraction of \( \delta \)
are retained in \( G_{ v_1,v_3,Q} \).
See Figure \ref{fig_contractions} for illustrations of triangle
and quadrilateral contractions.

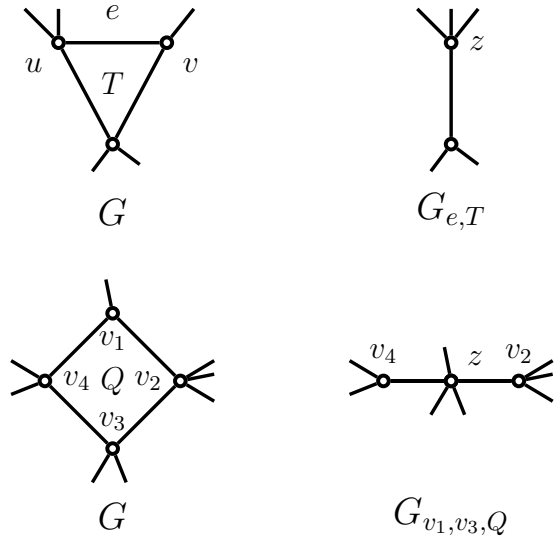
\begin{figure}[htp]
    \centering
    
\begin{tikzpicture}[scale=0.9]
    \begin{scope}
        \coordinate [vertex,label=below left:$u$] (u) at (-0.8,1.5);
        \coordinate [vertex,label=below right:$v$] (v) at (0.8,1.5);
        \coordinate [vertex] (w) at (0,0);
        \draw [edge] (u) --  (v)node [midway,label=above:$e$]{} -- (w) -- (u);
        \draw [edge] (u) -- +(-0.5,0.5);
        \draw [edge] (u) -- +(0.0,0.5);
        \draw [edge] (v) -- +(0.4,0.5);
        \draw [edge] (w) -- +(0.4,-0.3);
        \draw [edge] (w) -- +(-0.3,-0.4);
        \node at (0,0.9) {$T$};
        \node at (0,-1) {\large $G$};
    \end{scope}
    \begin{scope}[shift={(5,0)}]
        \coordinate [vertex,label=right:$z$] (z) at (0,1.5);
        \coordinate [vertex] (w) at (0,0);
        \draw [edge] (z) --  (w);
        \draw [edge] (z) -- +(-0.5,0.5);
        \draw [edge] (z) -- +(0.0,0.5);
        \draw [edge] (z) -- +(0.4,0.5);
        \draw [edge] (w) -- +(0.4,-0.3);
        \draw [edge] (w) -- +(-0.3,-0.4);
        \node at (0,-1) {\large $G_{e,T}$};
    \end{scope}
    \begin{scope}[shift={(0,-3.5)}]
        \coordinate [vertex,label=right:$v_4$] (v4) at (-1,0);
        \coordinate [vertex,label=below:$v_1$] (v1) at (0,1);
        \coordinate [vertex,label=left:$v_2$] (v2) at (1,0);
        \coordinate [vertex,label=above:$v_3$] (v3) at (0,-1);
        \node at (0,0){$Q$};
        \draw[edge] (v1) -- (v2) -- (v3) -- (v4) -- (v1);
        \draw[edge] (v4) -- +(-0.5,0.3);
        \draw[edge] (v4) -- +(-0.5,-0.2);
        \draw[edge] (v1) -- +(-0.1,0.5);
        \draw[edge] (v2) -- +(0.5,0.3);
        \draw[edge] (v2) -- +(0.5,0.1);
        \draw[edge] (v2) -- +(0.5,-0.3);
        \draw[edge] (v3) -- +(-0.3,-0.5);
        \draw[edge] (v3) -- +(0.2,-0.5);
        \node at (0,-2) {\large $G$};
    \end{scope}
    \begin{scope}[shift={(5,-3.5)}]
        \coordinate [vertex,label=above:$v_4$] (v4) at (-1,0);
        \coordinate [vertex,label=above right:$z$] (v1) at (0,0);
        \coordinate [vertex,label=above:$v_2$] (v2) at (1,0);
        \coordinate [vertex] (v3) at (0,0);
        \draw[edge] (v1) -- (v2) -- (v3) -- (v4) -- (v1);
        \draw[edge] (v4) -- +(-0.5,0.3);
        \draw[edge] (v4) -- +(-0.5,-0.2);
        \draw[edge] (v1) -- +(-0.1,0.5);
        \draw[edge] (v2) -- +(0.5,0.3);
        \draw[edge] (v2) -- +(0.5,0.1);
        \draw[edge] (v2) -- +(0.5,-0.3);
        \draw[edge] (v3) -- +(-0.3,-0.5);
        \draw[edge] (v3) -- +(0.2,-0.5);
        \node at (0,-2) {\large $G_{v_1,v_3,Q}$};
    \end{scope}
\end{tikzpicture}
    \caption{Triangle and quadrilateral contractions of a surface graph}
    \label{fig_contractions}
\end{figure}

We note here that in all the cases that arise in our later discussion the quadrilateral \( Q \) 
will have at least three distinct edges. So we will always assume (tacitly) that 
 the deleted edge from the set \(\{e_1,e_2\}\) is distinct from the deleted edge 
from the set \( \{e_3,e_4\} \). This assumption is not necessarily 
vacuous in the case where \( Q \) is a degenerate quadrilateral.

Let \( K \), respectively \( L \), be the unique balanced, respectively unbalanced, 
\( (2,3,2) \)-tight \( \mathbb A \)-graph with two vertices.
Let \( M \) be the unique unbalanced \( (2,3,1) \)-tight \( \mathbb A \)-graph with one vertex.
See Figure \ref{fig_basecases}.

\begin{theorem}
    Suppose that \( G \) is a \( (2,3,2) \)-tight \( \mathbb A \)-graph
    with at least one edge. Then there is 
    a sequence of \( (2,3,2) \)-tight 
    \( \mathbb A \)-graphs \( G_0,G_1,\dots,G_n \cong G \) where \( |V(G_0)| = 2 \) and 
    for \( i =1,\dots,n \), \( G_i \) is obtained from \( G_{i-1} \) by either
    a triangular vertex split or a quadrilateral vertex split. 
    Moreover, if \( G \) is unbalanced then \( G_0 \cong L \),
    whereas if \( G \) is balanced then \( G_0 \cong K \) and 
    only triangular vertex splits are required.
    \label{thm_cylinder_inductive232}
\end{theorem}

\begin{proof}
    See Section \ref{sec_inductive}.
\end{proof}

\begin{theorem}
    Suppose that \( G \) is a \( (2,3,1) \)-tight \( \mathbb A \)-graph
    with at least one edge.
    Then there is 
    a sequence of \( (2,3,1) \)-tight 
        \( \mathbb A \)-graphs \( G_0,G_1,
    \dots,G_n \cong G \) where \( |V(G_0)| \leq 2  \),  
    and  
    for \( i =1,\dots,n \), \( G_i \) is obtained from \( G_{i-1} \) by either
    a triangular vertex split or a quadrilateral vertex split.
    Moreover, if \( G \) is unbalanced then \( G_0 \cong M \), 
    whereas if \( G \) is balanced then \( G_0 \cong K\) and 
    only triangular vertex splits are required. 
    \label{thm_cylinder_inductive231}
\end{theorem}

\begin{proof}
    See Section \ref{sec_inductive}.
\end{proof}

The final piece of the puzzle, at least on the combinatorial 
side of things, is to clarify the relationship between
sparse and tight graphs.
It is well known that any $(2,3)$-sparse graph $D$ can be completed 
to a $(2,3)$-tight graph by adding appropriate edges to $D$. 
This follows from the fact that the edge sets of $(2,3)$-sparse 
subgraphs of the complete graph $K_{|V|}$ form an independence 
structure of a matroid, called the generic $2$-rigidity matroid 
(see, e.g., \cite{WW}). 
Similarly, it is known that the $(2,3,\ell)$-sparsity count induces a matroid for $\ell\in\{1,2\}$ (see \cite{ST1,NS}, for example). 

However, in the context of surface graphs it is not always clear that
these matroidal augmentation properties respect the topological 
embedding.  For example it is known that 
for any simple graph the \( (2,0) \)-sparse edge sets form 
the independent 
sets of a matroid.
Now consider the complete graph \( K_5 \), 
which is \( (2,0) \)-tight and can be embedded in the torus. 
However if
\( e \) is an edge of \( K_5 \) we observe that there is an
embedding of the \( (2,0) \)-sparse graph \( K_5-e \) in the torus 
that cannot be extended to an embedding of \( K_5 \).
By way of analogy we draw the reader's attention to 
the fact, as observed for example by Diestel (see \cite{MR3822066}, Chapter 4), 
that it is not immediately
obvious that a maximal plane graph is maximally planar. 
Diestel provides a careful proof that this is indeed the case in 
loc. cit.
One might view Proposition \ref{prop_232complete} as an analogue of 
that classical fact for certain classes of surface graphs.

\begin{proposition}
    Let \( l \in \{1,2\} \) and 
    let \( G \) be a \( (2,3,l) \)-sparse \( \mathbb A \)-graph.
    There 
    exists a \( (2,3,l) \)-tight \( \mathbb A \)-graph \( G' \)
    such that \( G \) is a spanning subgraph of \( G' \).
    \label{prop_232complete}
\end{proposition}

\begin{proof}
    See Section \ref{sec_completion}.
\end{proof}

\section{Sufficiency of the counts}
\label{sec_suff}

In this section we complete the proofs of Theorems \ref{thm_maintranslation} 
and \ref{thm_mainrotation}. We have already shown in Section 
\ref{sec_nec} the necessity of the sparsity conditions in each of these 
theorems. So this section is devoted to proving the sufficiency. 
Suppose that \( G \) is an \( \mathbb A \)-graph 
associated to a generic \( \Gamma \)-symmetric contact system and that 
\( e \in E(G) \). Suppose that \( e = \Gamma.(l,m) \) where 
\( l \) and \( m \) are segments in the contact system. By shortening
all the segments \( \Gamma.l \) a little bit we obtain a 
generic symmetric contact system whose \( \mathbb A \)-graph is \( G-e \). 
Thus, in light of Proposition \ref{prop_232complete} it suffices to 
prove the following result.

\begin{proposition}
    Let \( G \) be a \( (2,3,l) \)-tight \( \mathbb A \)-graph
    where \( l \in \{1,2\} \). For \( l = 1 \), respectively 
    \( l=2 \), let \( \Gamma \) be a discrete subgroup of 
    the Euclidean group generated by a rotation of order 
    at least 3, respectively a translation or a rotation of order 2.
    Then there is some generic \( \Gamma \)-symmetric
    contact system \( \mathcal L \)
    such that \( G_{\overline{\mathcal L}} \cong G \). 
    \label{prop_suff}
\end{proposition}
\begin{proof}
    First observe that by Lemma \ref{lem_lifting} it suffices to 
    show that there is some contact system \( \mathcal N \)
    in \( \mathbf O \) such that \( G_{\mathcal N} \cong G \). 
    The proof is by induction on \( |V(G)| \). For \( |V(G)| \leq 2 \)
    see Figure \ref{fig_basecases} for illustrations of the required
    contact systems. Now suppose that \( |V(G)| \geq 3 \). By Theorems
    \ref{thm_cylinder_inductive232} and \ref{thm_cylinder_inductive231}
    we can find some \( (2,3,l) \)-tight \( \mathbb A \)-graph \( G' \)
    such that \( G' \) is either a triangle contraction of \( G \)
    or a quadrilateral contraction of \( G \). By induction the $\mathbb A$-graph
    \( G' \) has a representation by a contact system in \( \mathbf O \), 
    say \( 
    \mathcal N'\). So it suffices
    to show that the corresponding triangle splitting or 
    quadrilateral splitting moves are representable by
    contact systems in \( \mathbf O \). In both cases we must replace a 
    single geodesic \( \overline l \in \mathcal N'\)
    by a pair of segments so that the contacts of the new segments 
    correspond to the appropriate subsets of the neighbours of \( 
    \overline l\). 
    
{
    In Figure \ref{fig_triangle_split_realisations} we illustrate the possibilities for triangle splits. In each case the solid segment $\overline{l}$ is replaced by 
    two segments $\overline{l}'$ and $\overline{l}''$ which 
    both contact the segment $w$. We observe that 
    \begin{itemize}
        \item The segments $\overline{l}', \overline{l}''$ can be chosen to lie in an arbitrarily small neighbourhood of the original segment $\overline{l}$. Thus if $m$ is a segment that is adjacent to $\overline{l}$ (considered as vertices in $G'$) then we can ensure that there is a corresponding intersection point between $m$ and one of $\overline{l}',\overline{l}''$
        \item The point of contact between $\overline{l}'$ and $\overline{l}''$ can be chosen to ensure that the set of segments intersecting with $\overline{l}'$, respectively $\overline{l}''$ correspond
        to the neighbour sets of 
        $\overline{l}'$, respectively $\overline{l}''$ in $G_{\mathcal{N}}$.
        
    \end{itemize}
    
    In Figure \ref{fig_quad_split_realisations} we see the corresponding diagrams for quadrilateral splits. In this case we observe that the line segment $\overline{l}$ can be replaced by two parallel segments $\overline{l}', \overline{l}''$ that are arbitrarily close to the original segment $\overline{l}$ and that realise the given quadrilateral splitting move.
    
    Finally, we observe that these geometric constructions, both for triangle splits and quadrilateral splits, apply even in the case where one or more of the split edges involving $\overline{l}$ is a loop edge. 
}

{
    \begin{figure}
        \centering
        
\begin{tabular}{|c|c|c|c|c|c|c|}

\hline

Surface graph 
&
\multicolumn{6}{Sc|}{Contact systems}

\\
\hline

\begin{tikzpicture}
    \coordinate [vertex,label=above:$\overline{l}$] (z) at (0,0);
    \coordinate [vertex,label=below:$w$] (w) at (0,-1);
    \draw [edge] (z) -- (w);
\end{tikzpicture}
&

\multicolumn{4}{Sc|}{
\begin{tikzpicture}
    \draw[segment] (-0.7,0.1) coordinate (z1) -- (0.7,0) coordinate (z2) coordinate[pos=0.5] (w1);
    \draw [dsegment](0,-1.5) coordinate (w2) -- (w1);
\end{tikzpicture}

}
&
\multicolumn{2}{Sc|}{
\begin{tikzpicture}
    \draw [dsegment](-0.6,-1.1) coordinate (z1) -- (0.6,-1) coordinate (z2) coordinate[pos=0.5] (w1);
    \draw[segment] (0,0.5) coordinate (w2) -- (w1);
    
\end{tikzpicture}
}
\\
\hline

\begin{tikzpicture}
    \coordinate [vertex,label=above:$\overline{l}'$] (u) at (-0.2,0);
    \coordinate [vertex,label=above:$\overline{l}''$] (v) at (0.2,0);
    \coordinate [vertex,label=below:$w$] (w) at (0,-1);
    \draw [edge] (u) -- (v) -- (w) -- (u);
\end{tikzpicture}

& 

\begin{tikzpicture}
    \draw [segment] (-0.7,0.1) coordinate (u1) -- (0.7,0) coordinate (u2) coordinate[pos=0.4] (w1) coordinate[pos=0.85] (v1) ;
    \draw [dsegment](-0.2,-1.5) coordinate (w2) -- (w1) coordinate[pos=0.9] (v2);
    \draw [segment] (v1)--(v2);
\end{tikzpicture}

& 
\begin{tikzpicture}
    \draw [dsegment](0,-1.5) coordinate (w1) -- (0,0) coordinate (w2) coordinate[pos=0.9] (v1);
    \draw[segment] (-0.6,0) coordinate (u1) -- (w2) 
    coordinate[pos=1.7] (u2); 
    \draw [segment] (u2) -- (w2);
    \draw [segment] (v1) -- (u2) coordinate[pos=1.8] (v2);
    \draw [segment] (u2) -- (v2);
\end{tikzpicture}

&

\begin{tikzpicture}
    \draw [dsegment](0,-1.5) coordinate (w1) -- (0,0) coordinate (w2) coordinate[pos=0.9] (u2);
    \draw [segment] (-0.8,0.1) coordinate (u1) -- (u2) coordinate[pos=0.5] (v1);
    \draw [segment] (v1) -- (w2) coordinate[pos=2.0] (v2);
    \draw [segment] (w2) -- (v2);
\end{tikzpicture}

&

\begin{tikzpicture}
    \draw [segment] (-0.7,0.1) coordinate (u1) -- (0.7,0) coordinate (u2) coordinate[pos=0.6] (w1) coordinate[pos=0.15] (v1) ;
    \draw [dsegment](0.2,-1.5) coordinate (w2) -- (w1) coordinate[pos=0.9] (v2);
    \draw [segment] (v1)--(v2);
\end{tikzpicture}

&

\begin{tikzpicture}
    \draw  [dsegment](-0.6,-1.1) coordinate (z1) -- (0.6,-1) coordinate (z2) coordinate[pos=0.5] (w1) coordinate[pos=0.3] (u1);
    \draw [segment] (0,0.5) coordinate (w2) -- (w1) coordinate[pos=0.3] (u2);
    \draw[segment] (u1)--(u2);
    
\end{tikzpicture}

&

\begin{tikzpicture}
    \draw [dsegment](-0.6,-1.1) coordinate (z1) -- (0.6,-1) coordinate (z2) coordinate[pos=0.5] (w1) coordinate[pos=0.7] (u1);
    \draw [segment] (0,0.5) coordinate (w2) -- (w1) coordinate[pos=0.3] (u2);
    \draw[segment] (u1)--(u2);
    
\end{tikzpicture}
\\ \hline

\end{tabular}

        \caption{Realising triangle splits with contact systems. In the diagram above the solid segments correspond to the vertices $\overline{l}$, $\overline{l}'$ or $\overline{l}''$, There are essentially six different ways in which we can realise a triangle split along the edge $\overline{l}w$, depending on which orientation of the edge is induced by the contact of the segment $\overline{l}$ and $w$ and also depending on how the neighbours of $\overline{l}$ are split among the vertices $\overline{l}'$ and $\overline{l}''$.}
        \label{fig_triangle_split_realisations}
    \end{figure}

\begin{figure}
    \centering
    
\begin{tabular}{|c|c|c|c|c|}
\hline

Surface graph 
&
\multicolumn{4}{Sc|}{Contact systems} 
\\
\hline

\begin{tikzpicture}
    \coordinate [vertex] (u) at (0,1);
    \coordinate [vertex,label=left:$\overline{l}$] (z) at (0,0);
    \coordinate [vertex] (v) at (0,-1);
    \draw [edge] (u) -- (z)--(v);
\end{tikzpicture}
&

\begin{tikzpicture}
    \draw [dsegment](-0.8,0.7) coordinate (v1) -- (0.8,0.7) coordinate[pos=0.5] (z1);
    \draw [dsegment](-0.8,-0.7) coordinate (w1) -- (0.8,-0.7) coordinate[pos=0.5] (z2);
    \draw [segment] (z1) -- (z2); 
\end{tikzpicture}
    
&

\begin{tikzpicture}
    \draw [dsegment](-0.8,0.7) coordinate (v1) -- (0,0.7);
    \draw [dsegment](-0.8,-0.7) coordinate (w1) -- (0.8,-0.7);
    \draw [segment] (0,-0.7) -- (0,1.5); 
\end{tikzpicture}

&

\begin{tikzpicture}
    \draw [dsegment](-0.8,0.7) coordinate (v1) -- (0,0.7);
    \draw [dsegment](-0.8,-0.3) coordinate (w1) -- (0,-0.3);
    \draw [segment] (0,-0.7) -- (0,1.5); 
\end{tikzpicture}

&

\begin{tikzpicture}
    \draw [dsegment](-0.8,0.7) coordinate (v1) -- (0,0.7);
    \draw [dsegment](0.8,-0.3) coordinate (w1) -- (0,-0.3);
    \draw [segment] (0,-0.7) -- (0,1.5); 
\end{tikzpicture}
    
\\ 
\hline

\begin{tikzpicture}
    \coordinate [vertex] (u) at (0,1);
    \coordinate [vertex,label=left:$\overline{l}'$] (z1) at (-0.3,0);
    \coordinate [vertex,label=right:$\overline{l}''$] (z2) at (0.3,0);
    \coordinate [vertex] (v) at (0,-1);
    \draw [edge] (u) -- (z1)--(v)--(z2)--(u);
\end{tikzpicture}
&

\begin{tikzpicture}
    \draw [dsegment](-0.8,0.7) coordinate (v1) -- (0.8,0.7)
    coordinate[pos=0.4] (zz1)
    coordinate[pos=0.5] (z1);
    \draw [dsegment](-0.8,-0.7) coordinate (w1) -- (0.8,-0.7) 
    coordinate[pos=0.4] (zz2)
    coordinate[pos=0.5] (z2);
    \draw [segment] (z1) -- (z2); 
    \draw [segment] (zz1) -- (zz2); 
\end{tikzpicture}
    
&

\begin{tikzpicture}
    \draw [dsegment](-0.8,0.7) coordinate (v1) -- (0,0.7);
    \draw [dsegment](-0.8,-0.7) coordinate (w1) -- (0.8,-0.7);
    \draw [segment] (0,-0.7) -- (0,1.5); 
    \draw [segment] (-0.2,0.7) -- (-0.2,-0.7);
\end{tikzpicture}

&

\begin{tikzpicture}
    \draw [dsegment](-0.8,0.7) coordinate (v1) -- (0,0.7);
    \draw [dsegment](-0.8,-0.3) coordinate (w1) -- (0,-0.3);
    \draw [segment] (0,-0.7) -- (0,1.5); 
    \draw [segment] (-0.2,0.7) -- (-0.2,-0.3);
\end{tikzpicture}

&

\begin{tikzpicture}
    \draw [dsegment](-0.8,0.7) coordinate (v1) -- (0,0.7);
    \draw [dsegment] (0.8,-0.3) coordinate (w1) -- (-0.2,-0.3);
    \draw [segment] (0,-0.3) -- (0,1.5); 
    \draw [segment] (-0.2,0.7) -- (-0.2,-0.7);
\end{tikzpicture}
\\
\hline

\end{tabular}
    
    \caption{Realising quadrilateral splits with contact systems. There are essentially four different possibilities for the realisation of a quadrilateral split depending on the orientation of the two edges that are split.}
    \label{fig_quad_split_realisations}
\end{figure}

%

}

    So in all cases the required spitting moves are realisable 
    by generic contact systems and the result follows by induction.
\end{proof}

\section{Pseudotriangulations} \label{sec:pseudo}

In this section we give another application of our 
combinatorial results to pseudotriangulations on flat surfaces, which
naturally arise from symmetric pseudotriangulations in the plane.
 For a comprehensive survey on pseudotriangulations and their applications we
refer  the reader to \cite{RFS}. See also the recent work by Borcea and Streinu on \emph{periodic} pseudotriangulations (see \cite{BS_pseudo}, for example). 

Note that while in other sections of the paper a graph is understood to 
be directed, throughout 
Section~\ref{sec:pseudo} exceptionally we understand graphs as {\em undirected}.

\subsection{Pseudotriangulations on flat surfaces} \label{subsec:pseudo}

A simple undirected   graph with straight line edges is called a \emph{geometric graph}. For a discrete subgroup $\Gamma$ of the Euclidean group, we say that a geometric  graph $G=(V,E)$ is \emph{$\Gamma$-symmetric} if for all $e\in E$ and all $g\in \Gamma$, we have $g.e\in E$, where the edge $e$ is considered as a line segment in the plane and $g.e$ denotes the image of the line segment $e$ under the linear transformation defined by $g$. {Throughout this section, we assume that $G$ is a \emph{plane} geometric graph (i.e., $G$ has no crossing edges) and that $\Gamma$ acts freely on the vertices and edges of $G$.}

Recall from Section~\ref{sec_symm_config} that the flat surface consisting of the non-singular points of 
$\mathbb{R}^2/\Gamma$
is denoted by $\mathbf{O}$. We assume throughout this section that $\Gamma$ is either the trivial group, or is generated by a translation or rotation, and hence $\mathbf{O}$ is the plane, a 
flat cylinder, or a flat cone (with the cone point removed) with cone angle $2\pi/k$, $k\geq 2$. Note that under this quotient map each $\Gamma$-orbit of edges of a $\Gamma$-symmetric geometric graph is mapped to a locally geodesic line segment in $\mathbf{O}$. Thus, a $\Gamma$-symmetric geometric graph $G$ naturally gives rise to an embedding of the quotient graph of $G$ with locally geodesic edges in $\mathbf{O}$, which we call a  {\emph{geometric $\mathbf{O}$-graph}.} Note that if $\Gamma$ is non-trivial then the underlying
 surface graph is an $\mathbb{A}$-graph.


For a (possibly degenerate or non-cellular)  face $F$ of a {geometric $\mathbf{O}$-graph}, we say that a vertex in the boundary of $F$ is \emph{convex} if the internal  angle (with respect to $F$) of the boundary at this vertex is convex, that is, strictly smaller than $\pi$. A \emph{pseudotriangle} is a  cellular face of $G$ with exactly three convex vertices. A vertex $v$ of $G$ is called \emph{pointed} if there are two consecutive edges incident with $v$ which form an angle that is strictly larger than $\pi$.

A {geometric $\mathbf{O}$-graph} is called a \emph{pointed pseudotriangulation} in $\mathbf{O}$ if
$G$ is connected, every vertex of $G$ is pointed, and every face of $G$ has a minimum number of convex vertices in its boundary. Note that this implies that every cellular face of a pointed pseudotriangulation $G$ in $\mathbf{O}$ is a pseudotriangle. Moreover, if $\mathbf{O}$ is the plane, then the unbounded face of $G$ has no convex vertices. Similarly, if $\mathbf{O}$ is a flat cylinder, then each unbounded face of $G$ has exactly one convex vertex, and if 
$\mathbf{O}$ is a cone with cone angle $2\pi/k$, $k\geq 2$, then the unbounded face of $G$ has no convex vertices, whereas the face of $G$ containing the cone point has exactly two convex vertices if the cone angle is $\pi$, and exactly one  convex vertex otherwise. Finally, if $\mathbf{O}$ is a flat cylinder or cone and the $\mathbb{A}$-graph of $G$ is balanced, then the non-cellular face must have no convex angles. {See Figure~\ref{fig:pseudo_ex} for some examples.}
 
 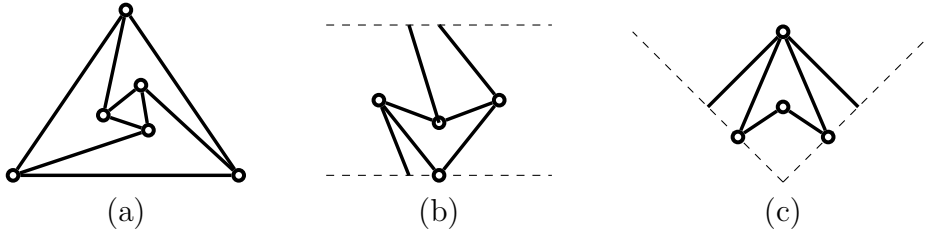
\begin{figure}[htp]
    \centering

\begin{tabular}{ccc}
\begin{tikzpicture}
\tikzstyle{every node}=[circle, draw=black,  fill=white, inner sep=0pt, minimum width=4pt];
    \coordinate [vertex] (v1) at (1.5,2);
    \coordinate [vertex] (v2) at (0,-0.2);
    \coordinate [vertex] (v3) at (3,-0.2);
    \coordinate [vertex] (v4) at (1.7,1);
     \coordinate [vertex] (v5) at (1.2,0.6);
      \coordinate [vertex] (v6) at (1.8,0.4);
    \draw[edge] (v1) -- (v2) -- (v3) -- (v1); 
      \draw[edge] (v4) -- (v5) -- (v6)--(v4);
      \draw[edge] (v1) -- (v5);
     \draw[edge] (v2) -- (v6); 
      \draw[edge] (v3) -- (v4);
\end{tikzpicture}
& \hspace{0.5cm}
\begin{tikzpicture}
\tikzstyle{every node}=[circle, draw=black,  fill=white, inner sep=0pt, minimum width=4pt];
   \draw[edge, dashed, thin] (0,0)--(3,0);
   \draw[edge, dashed, thin] (0,2)--(3,2);
    \coordinate [vertex] (v1) at (1.5,0);
    \coordinate [vertex] (v2) at (2.3,1);
    \coordinate [vertex] (v4) at (0.7,1);
    \coordinate [vertex] (v3) at (1.5,0.7);
    \draw[edge] (v1) -- (v2);
    \draw[edge] (1.5,0.7) -- (1.1,2);
    \draw[edge] (1.1,0) -- (v4);
    \draw[edge] (v3) -- (v2);
    \draw[edge] (v3) -- (v4);
    \draw[edge] (v4) -- (v1);
    \draw[edge] (v2) -- (1.5,2);
     
\end{tikzpicture}
&\hspace{0.5cm}
\begin{tikzpicture}
\tikzstyle{every node}=[circle, draw=black,  fill=white, inner sep=0pt, minimum width=4pt];
    \draw[edge, dashed, thin] (1.5,0) -- (3.5,2);
    \draw[edge, dashed, thin] (1.5,0) -- (-0.5,2);
    \coordinate [vertex] (v1) at (1.5,1);
    \coordinate [vertex] (v2) at (2.1,0.6);
    \coordinate [vertex] (v4) at (0.9,0.6);
    \coordinate [vertex] (v3) at (1.5,2);
    \draw[edge] (v1) -- (v2);
    \draw[edge] (v1) -- (v4);
     \draw[edge] (v3) -- (v2);
    \draw[edge] (v3) -- (v4);
     \draw[edge] (v3) -- (0.5,1);
    \draw[edge] (v3) -- (2.5,1);
   \end{tikzpicture}
   \\(a)  &\hspace{0.5cm} (b) &\hspace{0.5cm} (c)
\end{tabular}

\caption{{Examples of pointed pseudotriangulations in $\mathbf{O}$, where $\mathbf{O}$ is the plane (a), the flat cylinder (b) and the flat cone with angle $\frac{\pi}{2}$ (c).}}
    \label{fig:pseudo_ex}
    \end{figure}

 It was shown in \cite{Str05} that the graph of any pointed pseudotriangulation in the plane is $(2,3)$-tight. Conversely, it was shown in \cite{Haasetal} that every planar $(2,3)$-tight graph can be embedded as a pointed pseudotriangulation in the plane (see also \cite{FJW}). Using Theorems~\ref{thm_cylinder_inductive232} and \ref{thm_cylinder_inductive231} we can extend these results to pointed pseudotriangulations on other flat surfaces.


We have the following necessary condition for a  geometric graph to be a  pointed pseudotriangulation in $\mathbf{O}$.  The proof of this result adapts a counting argument in \cite{Str05}.
 
\begin{proposition}\label{prop:pt_nec} Let $\mathbf{O}$ be a flat 
cylinder or flat cone with cone angle $2\pi/k$, $k\geq 2$. Then the $\mathbb{A}$-graph of a pointed pseudotriangulation in $\mathbf{O}$ is $(2,3,2)$-tight if $\mathbf{O}$ is a cylinder or a cone with cone angle $\pi$, and $(2,3,1)$-tight otherwise. 
\end{proposition}
\begin{proof}   Let $G$ be a pointed pseudotriangulation in $\mathbf{O}$, and let $n$, $m$ and $f$ be the number of  vertices, edges and faces of $G$, respectively. If the $\mathbb{A}$-graph of $G$ is balanced, the result follows from \cite[Theorem 2.3]{Str05}, because in this case $G$ is isometric to a pointed pseudotriangulation in the plane (as we may cut $\mathbf O$ along a path joining the ends of $\mathbf O$ that does not meet $G$). So we may assume that the $\mathbb{A}$-graph of $G$ is unbalanced. We count the number $c$ of   convex angles of $G$  in two different ways.

Suppose first that $\mathbf{O}$ is a cylinder or a cone with cone angle $\pi$. Then, by definition of a pointed pseudotriangulation in $\mathbf{O}$, we have $c=3(f-2)+2$. On the other hand, since every vertex is pointed, we have $c=\sum_{v\in V(G)} (deg(v)-1)=2m-n$. Since $\mathbf{O}$ has genus zero, Euler's formula gives $n-m+f=2$, and we obtain $3f-4=3(m-n+2)-4=2m-n$. Thus, we have $m=2n-2$. 

For the sparsity counts, 
let $G'$ be a subgraph of $G$ and let $m'$, $n'$ and $f'$ be the number of vertices, edges and faces of $G'$. Suppose first that $G'$ is unbalanced. 
Since pointedness is a hereditary property, and since each cellular face of $G'$ must have at least three convex angles and the non-cellular faces must have at least two convex angles in total, we have $2m'-n'\geq 3f'-4$. This implies that $m'\leq 2n'-2$. 
If $G'$ is balanced, then as above it is isometric to a {geometric plane-graph} 
 and we have \( 2m'-n' \geq 3(f'-1) \). Hence \( m' \leq 2n'
-3\).


Note that if $\mathbf{O}$ is a cone with cone angle  $2\pi/k$, where $k\geq 3$, then  $c=3(f-2)+1$. By the same argument as above, it then follows that $G$ is $(2,3,1)$-tight.
\end{proof}

We will now show that the converse of Proposition \ref{prop:pt_nec} holds.

\begin{proposition} \label{prop_pseudo}
If $\mathbf{O}$ is a flat cylinder or flat cone with cone angle $\pi$, 
 then for any  $(2,3,2)$-tight $\mathbb{A}$-graph $G$ there exists a pointed pseudotriangulation in $\mathbf{O}$ whose $\mathbb{A}$-graph is $G$. 
Similarly, if $\mathbf{O}$ is a flat cone with cone angle $2\pi/k$, $k\geq 3$,
 then for any  $(2,3,1)$-tight $\mathbb{A}$-graph $G$ there exists a pointed pseudotriangulation in $\mathbf{O}$ whose $\mathbb{A}$-graph is $G$. 
\end{proposition}
\begin{proof}
 Let $G$ be a $(2,3,2)$-tight ($(2,3,1)$-tight, respectively) $\mathbb{A}$-graph. If $G$ is balanced, then 
 it follows from \cite[Theorem 1]{Haasetal} that $G$ can be realised as a pointed pseudotriangulation in the  plane, and hence (via an isometric embedding
 of the corresponding subset of the plane) also as a pointed pseudotriangulation   in $\mathbf{O}$.
 Let $G_0,  \ldots,G_n=G$ be the construction sequence for $G$ from Theorem~\ref{thm_cylinder_inductive232} (if $G$ is $(2,3,2)$-tight) or  Theorem~\ref{thm_cylinder_inductive231} (if $G$ is $(2,3,1)$-tight). In each case we may clearly construct  a  pointed pseudotriangulation in $\mathbf{O}$ whose  $\mathbb{A}$-graph is $G_0$. See Figure~\ref{fig:pseudo_basecases} for an illustration.

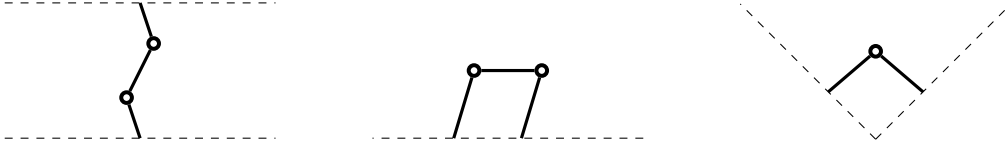
\begin{figure}[htp]
\begin{center}
  \begin{tikzpicture}[very thick,scale=0.9]
\tikzstyle{every node}=[circle, draw=black, fill=white, inner sep=0pt, minimum width=3pt];
    
   \draw[dashed,thin] (-2,0)  --  (2,0);
      \draw[dashed,thin] (-2,-2)  --  (2,-2);
      \path (0.2,-0.6) node[vertex] (v)  {} ;
    \path (-0.2,-1.4) node[vertex] (w)  {} ;
                    
\draw (v)  --  (w);
        \draw (v)  --  (0,0);
    \draw (w)  --  (0,-2);
       
              \end{tikzpicture}
        \hspace{1cm}
         \begin{tikzpicture}[very thick,scale=0.9,rotate=180]
\tikzstyle{every node}=[circle, draw=black, fill=white, inner sep=0pt, minimum width=3pt];
  \draw[dashed,thin] (-2,0)  --  (2,0);
        \path (-0.5,-1) node[vertex] (v)  {} ;
    \path (0.5,-1) node[vertex] (w)  {} ;
    \draw (v)  --  (w);
        \draw (v)  --  (-0.2,0);
    \draw (w)  --  (0.8,0);
     \node [rectangle, draw=white, fill=white] (a) at (0,-1.5) {$\quad$};   
               \end{tikzpicture}
         \hspace{1cm}
         \begin{tikzpicture}[very thick,scale=0.9,rotate=180]
\tikzstyle{every node}=[circle, draw=black, fill=white, inner sep=0pt, minimum width=3pt];
  \draw[dashed,thin] (0,0)  --  (-2,-2);
    \draw[dashed,thin] (0,0)  --  (2,-2);
        \path (0,-1.3) node[vertex] (v)  {} ;
           \draw (v)  --  (-0.7,-0.7);
    \draw (v)  --  (0.7,-0.7);
                  \end{tikzpicture}          
                \caption{Pointed pseudotriangulations in $\mathbf{O}$, where $\mathbf{O}$ is a cylinder, a cone with cone angle $\pi$, and a cone with cone angle $\pi/2$, respectively.}
                \label{fig:pseudo_basecases}
\end{center}
\end{figure}

In each step of the construction sequences, $G_{i}$ is obtained from $G_{i-1}$ by a triangular or quadrilateral vertex split.
We suppose that $G_{i-1}$ is embedded as a pointed pseudotriangulation in $\mathbf{O}$, and  show that the position of the new vertex can be chosen in such a way that the resulting  {geometric $\mathbf{O}$-graph} is again a  pointed pseudotriangulation in $\mathbf{O}$ whose $\mathbb{A}$-graph is $G_i$.

Suppose first that $G_i$ is obtained from $G_{i-1}$ by a triangular vertex split applied to the vertex $v$ along the edge $e=vw$. More precisely, if we write the edges of $G_{i-1}$ that are incident with $v$ in counterclockwise order as  $(e,f_1,\ldots, f_t)$, then  without loss of generality $G_i$ is obtained from $G_{i-1}$ by adding a new vertex $v'$ so that $v'$ is adjacent to $v$ and $w$ via two new edges, the edges   $f_1,\ldots, f_s$   for $1\leq s\leq t$ (or none of the edges $f_i$) are replaced with the edges $f'_i$, where $f_i'$ is obtained from $f_i$ by changing the end vertex $v$ to $v'$, and the remaining edges $f_{s+1},\ldots, f_t,e$ remain incident with $v$.
Note that any loop appears twice in the list $(e,f_1,\ldots, f_t)$. In particular, if $e$ is a loop, then it appears again as an edge $f_{s+\ell}$ for some $\ell\geq 1$ since $e$ remains unchanged by the triangular vertex split. If $f_i$ is a loop, however, which appears again in the list as $f_{i'}$ with $i<i'$, then we may have $i\leq s$ and $i'>s$, i.e.  the triangular vertex split may change the loop edge $f_i$ to a non-loop edge.

Since $G_{i-1}$ has been embedded as a pointed pseudotriangulation in $\mathbf{O}$, there exists a line segment $L$ containing  $v$ so that all edges incident with $v$ emanate from $v$ on the same side of $L$. We consider three distinct cases depending on the position of the edges incident with $v$ in their counterclockwise order from $L$.

Case 1: the order is: $f_{s+1},\ldots, f_t, e, f_1, \ldots, f_s$. In this case we choose the position of the new vertex $v'$ so that it lies sufficiently close to $v$ within the open conical region bounded by $L$ and the edge $f_s$ (or $e$ if there are no edges $f_1,\ldots, f_s$).

Case 2: the order is: $f_{s+1+\ell},\ldots, f_t, e, f_1, \ldots, f_s, f_{s+1}, \ldots, f_{s+\ell}$ for some $\ell\geq 1$.  In this case we choose the position of the new vertex $v'$ so that it lies sufficiently close to $v$ within the open conical region bounded by the edges $f_s$ and  $f_{s+1}$.
If there are no edges $f_1,\ldots, f_s$, then we choose $e$  instead of $f_s$ (where $e=f_{s+i}$ for $1\leq i \leq \ell$ if $e$ is a loop), and if there is no $f_{s+1}$ then we are back in Case 1.

Case 3: the order is: $f_{s+1-\ell},\ldots,f_s, f_{s+1},\ldots, f_t, e, f_1, \ldots, f_{s-\ell}$ for some $\ell\geq 1$.  In this case we choose the position of the new vertex $v'$ so that it lies sufficiently close to $v$ within the open conical region bounded (in counterclockwise order) by $L$ and the line segment obtained by inverting the edge $f_{s+1}$ in $v$. (If there is no $f_{s+1}$, then we choose $e$ instead.)

 In each case it is straightforward to see that the  resulting {geometric $\mathbf{O}$-graph} is a  pointed pseudotriangulation in $\mathbf{O}$. See also Figure~\ref{fig:pseudo_trian}.

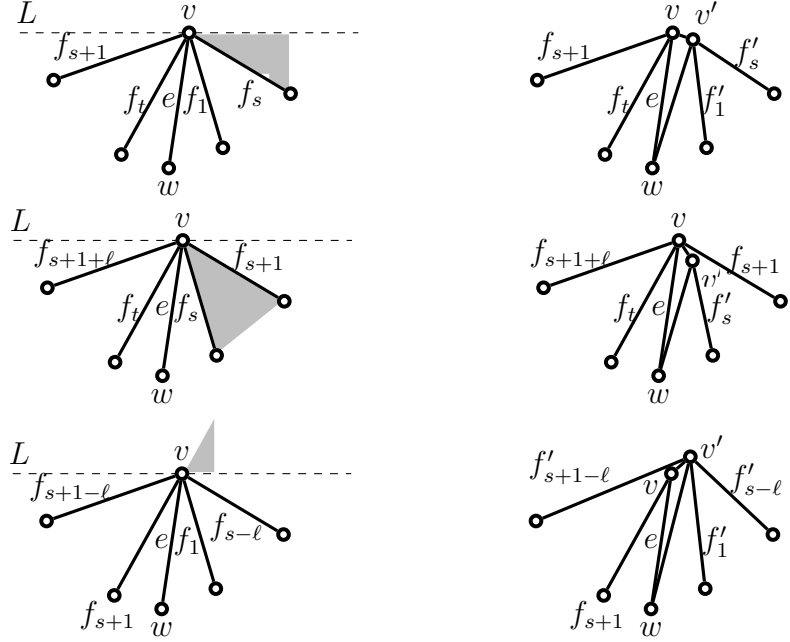
\begin{figure}[htp]
\begin{center}
  \begin{tikzpicture}[very thick,scale=0.9]
      \tikzstyle{every node}=[vertex];

\draw[fill=gray!50!white,draw=white](0,0)--(1.5,-0.9)--(1.5,0)--(0,0);

     \node [rectangle, draw=white, fill=white] (a) at (-2.4,0.3) {$L$};  
  \node [rectangle, draw=white, fill=white] (a) at (0,0.3) {$v$};    
      \node [rectangle, draw=white, fill=white] (a) at (-0.3,-2.3) {$w$}; 
       \node [rectangle, draw=white, fill=white] (a) at (-0.3,-1) {$e$}; 
       \node [rectangle, draw=white, fill=white] (a) at (0.1,-1) {$f_1$}; 
       \node [rectangle, draw=white, fill=white] (a) at (0.9,-0.9) {$f_s$}; 
         \node [rectangle, draw=white, fill=white] (a) at (-1.6,-0.2) {$f_{s+1}$}; 
            \node [rectangle, draw=white, fill=white] (a) at (-0.8,-1) {$f_{t}$};   
      
   \draw[dashed,thin] (-2.5,0)  --  (2.5,0);
        \path (0,0) node (v)  {} ;
    \path (-2,-0.7) node (p1)  {} ;
         \path (-1,-1.8) node (p3)  {} ;
     \path (-0.3,-2) node (p4)  {} ;
     \path (0.5,-1.7) node (p5)  {} ;
     \path (1.5,-0.9) node (p6)  {} ;
           
\draw (v)  --  (p1);
        \draw (v)  --  (p3);
       \draw (v)  --  (p4);
        \draw (v)  --  (p5);
       \draw (v)  --  (p6);

              \end{tikzpicture}
        \hspace{2cm}
         \begin{tikzpicture}[very thick,scale=0.9]
      \tikzstyle{every node}=[vertex];

  \node [rectangle, draw=white, fill=white] (a) at (0,0.3) {$v$};   
    \node [rectangle, draw=white, fill=white] (a) at (0.5,0.3) {$v'$}; 
      \node [rectangle, draw=white, fill=white] (a) at (-0.3,-2.3) {$w$}; 
       \node [rectangle, draw=white, fill=white] (a) at (-0.3,-1) {$e$}; 
       \node [rectangle, draw=white, fill=white] (a) at (0.6,-1) {$f'_1$}; 
       \node [rectangle, draw=white, fill=white] (a) at (1.1,-0.3) {$f'_s$}; 
         \node [rectangle, draw=white, fill=white] (a) at (-1.6,-0.2) {$f_{s+1}$}; 
            \node [rectangle, draw=white, fill=white] (a) at (-0.8,-1) {$f_{t}$};   
 
   \path (0.3,-0.1) node (v')  {} ;
        \path (0,0) node (v)  {} ;
    \path (-2,-0.7) node (p1)  {} ;
         \path (-1,-1.8) node (p3)  {} ;
     \path (-0.3,-2) node (p4)  {} ;
     \path (0.5,-1.7) node (p5)  {} ;
     \path (1.5,-0.9) node (p6)  {} ;
           
         \draw (v)  --  (v');  
            \draw (p4)  --  (v');  
\draw (v)  --  (p1);
        \draw (v)  --  (p3);
       \draw (v)  --  (p4);
        \draw (v')  --  (p5);
       \draw (v')  --  (p6);

        \end{tikzpicture}
\hspace{1cm}
  \begin{tikzpicture}[very thick,scale=0.9]
      \tikzstyle{every node}=[vertex];
       \node [rectangle, draw=white, fill=white] (a) at (0.05,-1) {$f_s$}; 
       \node [rectangle, draw=white, fill=white] (a) at (1.1,-0.25) {$f_{s+1}$}; 

\draw[fill=gray!50!white,draw=white](0,0)--(0.5,-1.7)--(1.5,-0.9)--(0,0);

     \node [rectangle, draw=white, fill=white] (a) at (-2.4,0.3) {$L$};  
  \node [rectangle, draw=white, fill=white] (a) at (0,0.3) {$v$};    
      \node [rectangle, draw=white, fill=white] (a) at (-0.3,-2.3) {$w$}; 
       \node [rectangle, draw=white, fill=white] (a) at (-0.3,-1) {$e$}; 
         \node [rectangle, draw=white, fill=white] (a) at (-1.6,-0.2) {$f_{s+1+\ell}$}; 
            \node [rectangle, draw=white, fill=white] (a) at (-0.8,-1) {$f_{t}$};   
      
   \draw[dashed,thin] (-2.5,0)  --  (2.5,0);
        \path (0,0) node (v)  {} ;
    \path (-2,-0.7) node (p1)  {} ;
         \path (-1,-1.8) node (p3)  {} ;
     \path (-0.3,-2) node (p4)  {} ;
     \path (0.5,-1.7) node (p5)  {} ;
     \path (1.5,-0.9) node (p6)  {} ;
           
\draw (v)  --  (p1);
        \draw (v)  --  (p3);
       \draw (v)  --  (p4);
        \draw (v)  --  (p5);
       \draw (v)  --  (p6);

              \end{tikzpicture}
        \hspace{2cm} 
         \begin{tikzpicture}[very thick,scale=0.9]
      \tikzstyle{every node}=[vertex];

  \node [rectangle, draw=white, fill=white] (a) at (0,0.3) {$v$};   
    \node [rectangle, draw=white, fill=white] (a) at (0.5,-0.55) {$v'$}; 
      \node [rectangle, draw=white, fill=white] (a) at (-0.3,-2.3) {$w$}; 
       \node [rectangle, draw=white, fill=white] (a) at (-0.3,-1) {$e$}; 
       \node [rectangle, draw=white, fill=white] (a) at (0.6,-1) {$f'_s$}; 
       \node [rectangle, draw=white, fill=white] (a) at (1.1,-0.3) {$f_{s+1}$}; 
         \node [rectangle, draw=white, fill=white] (a) at (-1.6,-0.2) {$f_{s+1+\ell}$}; 
            \node [rectangle, draw=white, fill=white] (a) at (-0.8,-1) {$f_{t}$};   
 
   \path (0.2,-0.3) node (v')  {} ;
        \path (0,0) node (v)  {} ;
    \path (-2,-0.7) node (p1)  {} ;
         \path (-1,-1.8) node (p3)  {} ;
     \path (-0.3,-2) node (p4)  {} ;
     \path (0.5,-1.7) node (p5)  {} ;
     \path (1.5,-0.9) node (p6)  {} ;
           
         \draw (v)  --  (v');  
            \draw (p4)  --  (v');  
\draw (v)  --  (p1);
        \draw (v)  --  (p3);
       \draw (v)  --  (p4);
        \draw (v')  --  (p5);
       \draw (v)  --  (p6);

        \end{tikzpicture}
        \hspace{1cm}
 \begin{tikzpicture}[very thick,scale=0.9]
      \tikzstyle{every node}=[vertex];
  \node [rectangle, draw=white, fill=white] (a) at (0,0.3) {$v$};    

\draw[fill=gray!50!white,draw=white](0,0)--(0.5,0)--(0.5,0.9)--(0,0);

     \node [rectangle, draw=white, fill=white] (a) at (-2.4,0.3) {$L$};  
      \node [rectangle, draw=white, fill=white] (a) at (-0.3,-2.3) {$w$}; 
       \node [rectangle, draw=white, fill=white] (a) at (-0.3,-1) {$e$}; 
       \node [rectangle, draw=white, fill=white] (a) at (0.08,-1) {$f_1$}; 
       \node [rectangle, draw=white, fill=white] (a) at (0.8,-0.8) {$f_{s-\ell}$}; 
         \node [rectangle, draw=white, fill=white] (a) at (-1.65,-0.2) {$f_{s+1-\ell}$}; 
            \node [rectangle, draw=white, fill=white] (a) at (-1.1,-2.1) {$f_{s+1}$};   
      
   \draw[dashed,thin] (-2.5,0)  --  (2.5,0);
        \path (0,0) node (v)  {} ;
    \path (-2,-0.7) node (p1)  {} ;
         \path (-1,-1.8) node (p3)  {} ;
     \path (-0.3,-2) node (p4)  {} ;
     \path (0.5,-1.7) node (p5)  {} ;
     \path (1.5,-0.9) node (p6)  {} ;
           
\draw (v)  --  (p1);
        \draw (v)  --  (p3);
       \draw (v)  --  (p4);
        \draw (v)  --  (p5);
       \draw (v)  --  (p6);

              \end{tikzpicture}
        \hspace{2cm} 
         \begin{tikzpicture}[very thick,scale=0.9]
      \tikzstyle{every node}=[vertex];
 
  \node [rectangle, draw=white, fill=white] (a) at (0.6,0.4) {$v'$};    
  \node [rectangle, draw=white, fill=white] (a) at (-0.3,-0.2) {$v$}; 
      \node [rectangle, draw=white, fill=white] (a) at (-0.3,-2.3) {$w$}; 
       \node [rectangle, draw=white, fill=white] (a) at (-0.3,-1) {$e$}; 
       \node [rectangle, draw=white, fill=white] (a) at (0.65,-1) {$f'_1$}; 
       \node [rectangle, draw=white, fill=white] (a) at (1.25,-0.1) {$f'_{s-\ell}$}; 
         \node [rectangle, draw=white, fill=white] (a) at (-1.5,0.1) {$f'_{s+1-\ell}$}; 
            \node [rectangle, draw=white, fill=white] (a) at (-1.1,-2.1) {$f_{s+1}$};   
      
        \path (0,0) node (v)  {} ;
             \path (0.28,0.25) node (v')  {} ;
    \path (-2,-0.7) node (p1)  {} ;
         \path (-1,-1.8) node (p3)  {} ;
     \path (-0.3,-2) node (p4)  {} ;
     \path (0.5,-1.7) node (p5)  {} ;
     \path (1.5,-0.9) node (p6)  {} ;
           
\draw (v')  --  (p1);
\draw (v)  --  (v');
\draw (v')  --  (p4);
        \draw (v)  --  (p3);
       \draw (v)  --  (p4);
        \draw (v')  --  (p5);
       \draw (v')  --  (p6);

        \end{tikzpicture}
     \caption{For $x=1,2,3$, row $x$ illustrates Case $x$ for placing the new vertex $v'$ in a triangular vertex split of $G_{i-1}$ to obtain another pointed pseudotriangulation in $\mathbf{O}$.}
 \label{fig:pseudo_trian}    
\end{center}
\end{figure}

Suppose next that $G_i$ is obtained from $G_{i-1}$ by a quadrilateral vertex split of $v$ along the edges $e_1=vw$ and $e_2=vx$. More precisely, if we write the edges of $G_{i-1}$ that are incident with $v$ in counterclockwise order as  $(e_1,f_1,\ldots, f_s,e_2,f_{s+1},\ldots f_t)$, then  without loss of generality $G_i$ is obtained from $G_{i-1}$ by adding a new vertex $v'$ so that $v'$ is adjacent to $w$ and $x$ via two new edges, the edges   $f_1,\ldots, f_s$  (which may not exist) are replaced with the edges $f'_i$, where $f_i'$ is obtained from $f_i$ by changing the end vertex $v$ to $v'$, and the remaining edges  $e_1,e_2,f_{s+1},\ldots, f_t$ remain incident to $v$.

Since $G_{i-1}$ has been embedded as a pointed pseudotriangulation in $\mathbf{O}$, there  exists a line segment $L$ containing  $v$ so that all edges incident to $v$ emanate from $v$ on the same side of $L$. 

Suppose first that the quadrilateral created by the quadrilateral vertex splitting move is non-degenerate, that is, the vertices $v, v', w$ and $x$ are all pairwise distinct. Then we consider two distinct cases depending on the position of the edges incident with $v$ in their counterclockwise order from $L$. 

Case 1: the order is $f_{s+1+\ell},\ldots, f_t, e_1, f_1, \ldots, f_s, e_2, f_{s+1}, \ldots, f_{s+\ell}$ for some $\ell\geq 1$.  In this case, we choose the position of the new vertex $v'$ so that it lies sufficiently close to $v$ within the open conical region $U$ bounded by the edges $e_1$ and $e_2$.

Case 2: the order is $f_{\ell+1},\ldots, f_s, e_2, f_{s+1}, \ldots, f_t, e_1, f_{1}, \ldots, f_{\ell}$ for some $\ell\geq 1$.
In this  case we choose the position of $v'$ so that it lies sufficiently close to $v$ within the open conical region obtained by inverting the region $U$ from Case 1 in $v$. 

In each case it is straightforward to check that the resulting  {geometric $\mathbf{O}$-graph} is  a  pointed pseudotriangulation in $\mathbf{O}$. See also Figure~\ref{fig:pseudo_quad}.

\begin{figure}[htp]
\begin{center}
  \begin{tikzpicture}[very thick,scale=0.9]
\tikzstyle{every node}=[vertex]

\draw[fill=gray!50!white,draw=white](0,0)--(-1,-0.75)--(1,-0.75)--(0,0);

     \node [rectangle, draw=white, fill=white] (a) at (-2.4,0.3) {$L$};  
  \node [rectangle, draw=white, fill=white] (a) at (0,0.3) {$v$};    
      \node [rectangle, draw=white, fill=white] (a) at (-2,-1.2) {$w$}; 
       \node [rectangle, draw=white, fill=white] (a) at (2,-1.2) {$x$}; 
         \node [rectangle, draw=white, fill=white] (a) at (-1.1,-0.6) {$e_1$}; 
       \node [rectangle, draw=white, fill=white] (a) at (1.1,-0.6) {$e_2$};      
                   \node [rectangle, draw=white, fill=white] (a) at (-0.7,-1.3) {$f_1$}; 
       \node [rectangle, draw=white, fill=white] (a) at (0.7,-1.3) {$f_s$};    
                   \node [rectangle, draw=white, fill=white] (a) at (-2,-0.2) {$f_{s+1+\ell}$}; 
       \node [rectangle, draw=white, fill=white] (a) at (2,-0.2) {$f_{s+1}$}; 
      
   \draw[dashed,thin] (-2.5,0)  --  (2.5,0);
        \path (0,0) node (v)  {} ;
    \path (-2,-1.5) node (w)  {} ;
         \path (2,-1.5) node (x)  {} ;
     \path (-2.3,-0.6) node (p1)  {} ;
     \path (-0.5,-1.7) node (p2)  {} ;
     \path (0.5,-1.7) node (p3)  {} ;
        \path (2.3,-0.6) node (p4)  {} ;

\draw (v)  --  (w);
        \draw (v)  --  (x);
       \draw (v)  --  (p1);
        \draw (v)  --  (p2);
       \draw (v)  --  (p3);
        \draw (v)  --  (p4);

              \end{tikzpicture}
        \hspace{2cm}
         \begin{tikzpicture}[very thick,scale=0.9]
\tikzstyle{every node}=[vertex]

  \node [rectangle, draw=white, fill=white] (a) at (0,0.3) {$v$};   
   \node [rectangle, draw=white, fill=white] (a) at (0,-0.8) {$v'$};
      \node [rectangle, draw=white, fill=white] (a) at (-2,-1.2) {$w$}; 
       \node [rectangle, draw=white, fill=white] (a) at (2,-1.2) {$x$}; 
         \node [rectangle, draw=white, fill=white] (a) at (-1.1,-0.6) {$e_1$}; 
       \node [rectangle, draw=white, fill=white] (a) at (1.1,-0.6) {$e_2$};      
                   \node [rectangle, draw=white, fill=white] (a) at (-0.7,-1.3) {$f'_1$}; 
       \node [rectangle, draw=white, fill=white] (a) at (0.7,-1.3) {$f'_s$};    
                   \node [rectangle, draw=white, fill=white] (a) at (-2,-0.2) {$f_{s+1+\ell}$}; 
       \node [rectangle, draw=white, fill=white] (a) at (2,-0.2) {$f_{s+1}$}; 
      
        \path (0,0) node (v)  {} ;
         \path (0,-0.3) node (v')  {} ;
    \path (-2,-1.5) node (w)  {} ;
         \path (2,-1.5) node (x)  {} ;
     \path (-2.3,-0.6) node (p1)  {} ;
     \path (-0.5,-1.7) node (p2)  {} ;
     \path (0.5,-1.7) node (p3)  {} ;
        \path (2.3,-0.6) node (p4)  {} ;

\draw (v)  --  (w);
\draw (x)  --  (v');
\draw (v')  --  (w);
        \draw (v)  --  (x);
       \draw (v)  --  (p1);
        \draw (v')  --  (p2);
       \draw (v')  --  (p3);
        \draw (v)  --  (p4);
 
        \end{tikzpicture} \hspace{1cm}
         \begin{tikzpicture}[very thick,scale=0.9]
\tikzstyle{every node}=[vertex]

\draw[fill=gray!50!white,draw=white](0,0)--(1,0.75)--(-1,0.75)--(0,0);

     \node [rectangle, draw=white, fill=white] (a) at (-2.4,0.3) {$L$};  
  \node [rectangle, draw=white, fill=white] (a) at (0,-0.5) {$v$};    
      \node [rectangle, draw=white, fill=white] (a) at (-2,-1.2) {$w$}; 
       \node [rectangle, draw=white, fill=white] (a) at (2,-1.2) {$x$}; 
         \node [rectangle, draw=white, fill=white] (a) at (-1.1,-0.6) {$e_2$}; 
       \node [rectangle, draw=white, fill=white] (a) at (1.1,-0.6) {$e_1$};      
                   \node [rectangle, draw=white, fill=white] (a) at (-0.8,-1.3) {$f_{s+1}$}; 
       \node [rectangle, draw=white, fill=white] (a) at (0.7,-1.3) {$f_t$};    
                   \node [rectangle, draw=white, fill=white] (a) at (-2,-0.2) {$f_{\ell+1}$}; 
       \node [rectangle, draw=white, fill=white] (a) at (2,-0.2) {$f_{1}$}; 
      
   \draw[dashed,thin] (-2.5,0)  --  (2.5,0);
        \path (0,0) node (v)  {} ;
    \path (-2,-1.5) node (w)  {} ;
         \path (2,-1.5) node (x)  {} ;
     \path (-2.3,-0.6) node (p1)  {} ;
     \path (-0.5,-1.7) node (p2)  {} ;
     \path (0.5,-1.7) node (p3)  {} ;
        \path (2.3,-0.6) node (p4)  {} ;

\draw (v)  --  (w);
        \draw (v)  --  (x);
       \draw (v)  --  (p1);
        \draw (v)  --  (p2);
       \draw (v)  --  (p3);
        \draw (v)  --  (p4);
       
           \end{tikzpicture}
        \hspace{2cm}
         \begin{tikzpicture}[very thick,scale=0.9]
\tikzstyle{every node}=[vertex]

  \node [rectangle, draw=white, fill=white] (a) at (0,-0.5) {$v$};  
     \node [rectangle, draw=white, fill=white] (a) at (0,0.6) {$v'$}; 
      \node [rectangle, draw=white, fill=white] (a) at (-2,-1.2) {$w$}; 
       \node [rectangle, draw=white, fill=white] (a) at (2,-1.2) {$x$}; 
         \node [rectangle, draw=white, fill=white] (a) at (-0.45,-0.6) {$e_2$}; 
       \node [rectangle, draw=white, fill=white] (a) at (0.45,-0.6) {$e_1$};      
                   \node [rectangle, draw=white, fill=white] (a) at (-0.8,-1.3) {$f_{s+1}$}; 
       \node [rectangle, draw=white, fill=white] (a) at (0.7,-1.3) {$f_t$};    
                   \node [rectangle, draw=white, fill=white] (a) at (-2,-0.2) {$f'_{\ell+1}$}; 
       \node [rectangle, draw=white, fill=white] (a) at (2,-0.2) {$f'_{1}$}; 
      
        \path (0,0) node (v)  {} ;
         \path (0,0.3) node (v')  {} ;
    \path (-2,-1.5) node (w)  {} ;
         \path (2,-1.5) node (x)  {} ;
     \path (-2.3,-0.6) node (p1)  {} ;
     \path (-0.5,-1.7) node (p2)  {} ;
     \path (0.5,-1.7) node (p3)  {} ;
        \path (2.3,-0.6) node (p4)  {} ;

\draw (v)  --  (w);
        \draw (v)  --  (x);
        \draw (v')  --  (w);
        \draw (v')  --  (x);
       \draw (v')  --  (p1);
        \draw (v)  --  (p2);
       \draw (v)  --  (p3);
        \draw (v')  --  (p4);
       
           \end{tikzpicture}
                \caption{For $x=1,2$, row $x$ illustrates Case $x$ for placing the new vertex $v'$ in a quadrilateral vertex split of $G_{i-1}$ to obtain another pointed pseudotriangulation in $\mathbf{O}$.}
 \label{fig:pseudo_quad}               
\end{center}
\end{figure}
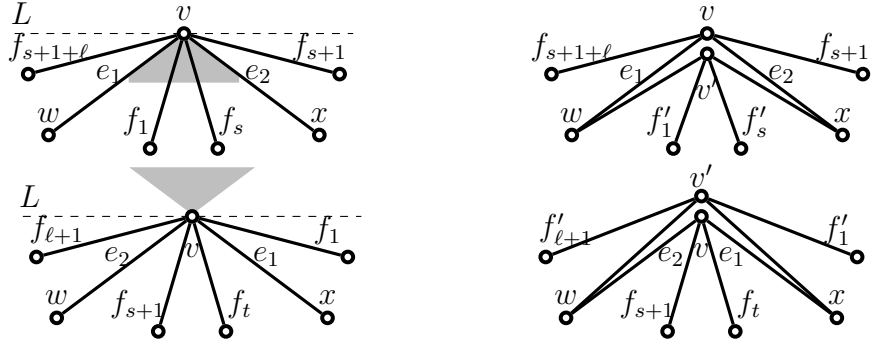

Suppose next that the quadrilateral created by the quadrilateral vertex split is degenerate. Then the $\mathbb{A}$-graph of the quadrilateral is one of the three graphs depicted in Figure~\ref{fig_degenquads}(a),(b),(c). In case (a), we have $w=x$ and the proof above applies. In cases (b) and (c), $\mathbf{O}$ is a flat cone with cone angle $2\pi/k$, $k\geq 3$, and the degenerate quadrilateral is obtained from a loop $e$ at vertex $v$, or a loop $e$ at vertex $v$ with an additional edge $f=vx$, respectively. 

In case (b), there are two cases for the counterclockwise order (from $L$) of the edges incident with $v$ in the pointed pseudotriangulation $G_{i-1}$ in $\mathbf{O}$, as shown on the left hand side of the first and second row in Figure~\ref{fig:deg_quad_ppt}. Similar to the non-degenerate case, it is straightforward to see that if the position of $v'$ is chosen sufficiently close to $v$ in the open conical regions depicted in Figure~\ref{fig:deg_quad_ppt}, and $v'$ is joined to $v$ with a `geometric twist', then the resulting {geometric $\mathbf{O}$-graph} is a pointed pseudotriangulation in $\mathbf{O}$.

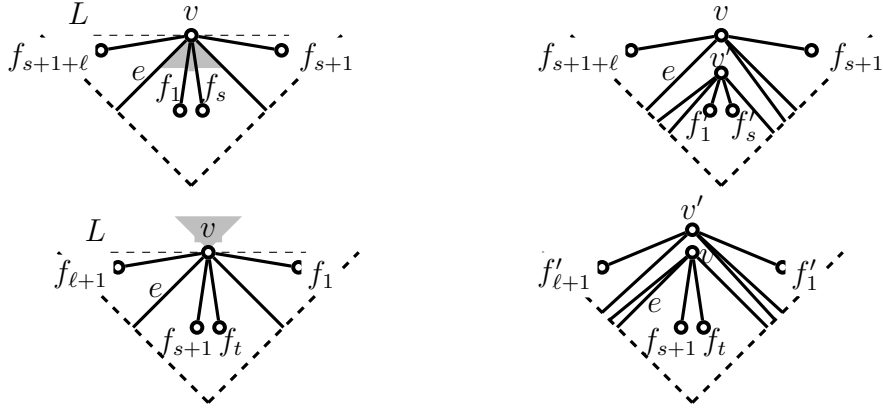
\begin{figure}[htp]
\begin{center}
  \begin{tikzpicture}[very thick,scale=1]
\tikzstyle{every node}=[vertex]

\draw[fill=gray!50!white,draw=white](0,2)--(-0.5,1.5)--(0.5,1.5)--(0,2);

     \node [rectangle, draw=white, fill=white] (a) at (-1.5,2.3) {$L$};  
         \node [rectangle, draw=white, fill=white] (a) at (-0.3,1.3) {$f_1$}; 
       \node [rectangle, draw=white, fill=white] (a) at (0.3,1.3) {$f_s$};    
          \node [rectangle, draw=white, fill=white] (a) at (-0.7,1.5) {$e$};      
            \node [rectangle, draw=white, fill=white] (a) at (0,2.3) {$v$};

   \draw[dashed,thin] (-1.3,2)  --  (1.3,2);  
       
        \path (0,2) node (v)  {} ;
             \path (-0.15,1) node (p1)  {} ;
     \path (0.15,1) node (p2)  {} ;
                 \path (-1.2,1.8) node (s1)  {} ;
     \path (1.2,1.8) node (s2)  {} ;

        \draw[dashed] (0,0)  --  (-2,2);
       \draw[dashed] (0,0)  --  (2,2);
         \node [rectangle, draw=white, fill=white] (a) at (1.8,1.7) {$f_{s+1}$}; 
       \node [rectangle, draw=white, fill=white] (a) at (-1.9,1.7) {$f_{s+1+\ell}$};  
       
        \draw (v)  --  (-1,1);
         \draw (v)  --  (1,1);
            \draw (v)  --  (p1);
       \draw (v)  --  (p2);    
          \draw (v)  --  (s1);
       \draw (v)  --  (s2); 
              \end{tikzpicture}
        \hspace{2cm}
         \begin{tikzpicture}[very thick,scale=1]
\tikzstyle{every node}=[vertex]
 
         \node [rectangle, draw=white, fill=white] (a) at (-0.3,0.8) {$f'_1$}; 
       \node [rectangle, draw=white, fill=white] (a) at (0.3,0.8) {$f'_s$};    
          \node [rectangle, draw=white, fill=white] (a) at (-0.7,1.5) {$e$};      
            \node [rectangle, draw=white, fill=white] (a) at (0,2.3) {$v$};       
                  \node [rectangle, draw=white, fill=white] (a) at (0,1.73) {$v'$};      
                 
       
        \path (0,2) node (v)  {} ;
         \path (0,1.5) node (vp)  {} ;
             \path (-0.15,1) node (p1)  {} ;
     \path (0.15,1) node (p2)  {} ;
                 \path (-1.2,1.8) node (s1)  {} ;
     \path (1.2,1.8) node (s2)  {} ;

        \draw[dashed] (0,0)  --  (-2,2);
       \draw[dashed] (0,0)  --  (2,2);
         \node [rectangle, draw=white, fill=white] (a) at (1.8,1.7) {$f_{s+1}$}; 
       \node [rectangle, draw=white, fill=white] (a) at (-1.9,1.7) {$f_{s+1+\ell}$};  
       
        \draw (v)  --  (-1,1);
         \draw (v)  --  (1,1);
            \draw (vp)  --  (p1);
       \draw (vp)  --  (p2);    
          \draw (v)  --  (s1);
       \draw (v)  --  (s2); 
      \draw (vp)  --  (-0.7,0.7);
 \draw (vp)  --  (0.7,0.7);
 \draw (v)  --  (0.85,0.85);
 \draw (vp)  --  (-0.85,0.85);
        \end{tikzpicture} 
          \begin{tikzpicture}[very thick,scale=1]
\tikzstyle{every node}=[vertex]

\draw[fill=gray!50!white,draw=white](0,2)--(-0.5,2.5)--(0.5,2.5)--(0,2);

     \node [rectangle, draw=white, fill=white] (a) at (-1.5,2.3) {$L$};  
         \node [rectangle, draw=white, fill=white] (a) at (-0.3,0.8) {$f_{s+1}$}; 
       \node [rectangle, draw=white, fill=white] (a) at (0.3,0.8) {$f_t$};    
          \node [rectangle, draw=white, fill=white] (a) at (-0.7,1.5) {$e$};      
            \node [rectangle, draw=gray!50!white, fill=gray!50!white] (a) at (0,2.3) {$v$};

   \draw[dashed,thin] (-1.3,2)  --  (1.3,2);  
       
        \path (0,2) node (v)  {} ;
             \path (-0.15,1) node (p1)  {} ;
     \path (0.15,1) node (p2)  {} ;
                 \path (-1.2,1.8) node (s1)  {} ;
     \path (1.2,1.8) node (s2)  {} ;

        \draw[dashed] (0,0)  --  (-2,2);
       \draw[dashed] (0,0)  --  (2,2);
         \node [rectangle, draw=white, fill=white] (a) at (1.5,1.7) {$f_{1}$}; 
       \node [rectangle, draw=white, fill=white] (a) at (-1.7,1.7) {$f_{\ell+1}$};  
       
        \draw (v)  --  (-1,1);
         \draw (v)  --  (1,1);
            \draw (v)  --  (p1);
       \draw (v)  --  (p2);    
          \draw (v)  --  (s1);
       \draw (v)  --  (s2); 
              \end{tikzpicture}
        \hspace{2cm}
         \begin{tikzpicture}[very thick,scale=1]
\tikzstyle{every node}=[vertex]
 
         \node [rectangle, draw=white, fill=white] (a) at (-0.3,0.8) {$f_{s+1}$}; 
       \node [rectangle, draw=white, fill=white] (a) at (0.3,0.8) {$f_t$};    
          \node [rectangle, draw=white, fill=white] (a) at (-0.5,1.3) {$e$};      
            \node [rectangle, draw=white, fill=white] (a) at (0.2,1.95) {$v$};       
                   \node [rectangle, draw=white, fill=white] (a) at (0,2.6) {$v'$};      
                 
       
        \path (0,2) node (v)  {} ;
             \path (-0.15,1) node (p1)  {} ;
     \path (0.15,1) node (p2)  {} ;
                 \path (-1.2,1.8) node (s1)  {} ;
     \path (1.2,1.8) node (s2)  {} ;     
                 \path (0,2.3) node (vp)  {} ;
                 
        \draw[dashed] (0,0)  --  (-2,2);
       \draw[dashed] (0,0)  --  (2,2);
         \node [rectangle, draw=white, fill=white] (a) at (1.5,1.7) {$f'_{1}$}; 
       \node [rectangle, draw=white, fill=white] (a) at (-1.7,1.7) {$f'_{\ell+1}$};  
       
        \draw (v)  --  (-1,1);
         \draw (v)  --  (1,1);
            \draw (v)  --  (p1);
       \draw (v)  --  (p2);

             \draw (vp)  --  (s2);    
          \draw (vp)  --  (s1);
            \draw (vp)  --  (1.2,1.2);
 \draw (vp)  --  (-1.2,1.2);
 \draw (v)  --  (-1.1,1.1);
 \draw (vp)  --  (1.1,1.1);
        \end{tikzpicture} 
                \caption{The two rows illustrate how to place the new vertex $v'$ in a quadrilateral vertex split of $G_{i-1}$ to obtain another pointed pseudotriangulation in $\mathbf{O}$, where the new quadrilateral has the $\mathbb{A}$-graph shown in Figure~\ref{fig_degenquads}(b). }
 \label{fig:deg_quad_ppt}               
\end{center}
\end{figure}

In case (c), we again consider the list of edges that are incident to $v$ in the pointed pseudotriangulation $G_{i-1}$ in $\mathbf{O}$ in counterclockwise order from $L$. In this list, the edge $f=vx$ may either lie between the two copies of the loop $e$, or between a copy of $e$ and $L$. In the first case, we illustrate how to place the new vertex $v'$ to obtain another pointed pseudotriangulation in $\mathbf{O}$ in Figure~\ref{fig:deg_quad_ppt_2}. The other case is similar. \end{proof}

\begin{figure}[htp]
\begin{center}
  \begin{tikzpicture}[very thick,scale=1]
\tikzstyle{every node}=[vertex]

\draw[fill=gray!50!white,draw=white](0,2)--(0,1.5)--(0.5,1.5)--(0,2);

     \node [rectangle, draw=white, fill=white] (a) at (-1.5,2.3) {$L$};  
         \node [rectangle, draw=white, fill=white] (a) at (-0.5,1.2) {$f_t$}; 
       \node [rectangle, draw=white, fill=white] (a) at (0.5,1.2) {$f_s$};    
          \node [rectangle, draw=white, fill=white] (a) at (-0.7,1.5) {$e$};      
            \node [rectangle, draw=white, fill=white] (a) at (0,2.3) {$v$};       
                      \node [rectangle, draw=white, fill=white] (a) at (-0.15,1) {$f$}; 
                  \node [rectangle, draw=white, fill=white] (a) at (0,0.48) {$x$}; 
                 
   \draw[dashed,thin] (-1.3,2)  --  (1.3,2);  
       
        \path (0,2) node (v)  {} ;
             \path (-0.35,0.8) node (p1)  {} ;
     \path (0.35,0.8) node (p2)  {} ;
                 \path (-1.2,1.8) node (s1)  {} ;
     \path (1.2,1.8) node (s2)  {} ;     
                 \path (0,0.7) node (x)  {} ;
                 
        \draw[dashed] (0,0)  --  (-2,2);
       \draw[dashed] (0,0)  --  (2,2);
         \node [rectangle, draw=white, fill=white] (a) at (1.8,1.7) {$f_{s+1}$}; 
       \node [rectangle, draw=white, fill=white] (a) at (-1.9,1.7) {$f_{s+1+\ell}$};  
       
        \draw (v)  --  (-1,1);
         \draw (v)  --  (1,1);
            \draw (v)  --  (p1);
            \draw(v)--(x);
       \draw (v)  --  (p2);    
          \draw (v)  --  (s1);
       \draw (v)  --  (s2); 
              \end{tikzpicture}
        \hspace{2cm}
         \begin{tikzpicture}[very thick,scale=1]
\tikzstyle{every node}=[vertex]
 
     \node [rectangle, draw=white, fill=white] (a) at (-1.5,2.3) {$L$};  
         \node [rectangle, draw=white, fill=white] (a) at (-0.6,0.85) {$f_t$}; 
       \node [rectangle, draw=white, fill=white] (a) at (0.55,0.8) {$f'_s$};    
          \node [rectangle, draw=white, fill=white] (a) at (-0.7,1.5) {$e$};      
            \node [rectangle, draw=white, fill=white] (a) at (0,2.3) {$v$};       
                      \node [rectangle, draw=white, fill=white] (a) at (-0.15,1) {$f$}; 
                  \node [rectangle, draw=white, fill=white] (a) at (0,0.48) {$x$}; 
                 
   \draw[dashed,thin] (-1.3,2)  --  (1.3,2);  
       
        \path (0,2) node (v)  {} ;
             \path (-0.35,0.8) node (p1)  {} ;
     \path (0.35,0.8) node (p2)  {} ;
                 \path (-1.2,1.8) node (s1)  {} ;
     \path (1.2,1.8) node (s2)  {} ;     
                 \path (0,0.7) node (x)  {} ;
                 \path (0.2,1.6) node (vp)  {} ;

        \draw[dashed] (0,0)  --  (-2,2);
       \draw[dashed] (0,0)  --  (2,2);
         \node [rectangle, draw=white, fill=white] (a) at (1.8,1.7) {$f_{s+1}$}; 
       \node [rectangle, draw=white, fill=white] (a) at (-1.9,1.7) {$f_{s+1+\ell}$};  
       
        \draw (v)  --  (-1,1);
         \draw (v)  --  (1,1);
            \draw (v)  --  (p1);
            \draw(v)--(x);
       \draw (vp)  --  (p2);    
        \draw (vp)  --  (x); 
          \draw (v)  --  (s1);
       \draw (v)  --  (s2); 
        \draw (vp)  --  (0.9,0.9); 
 \draw (v)  --  (-0.9,0.9); 
        \end{tikzpicture} 
          \begin{tikzpicture}[very thick,scale=1]

        \end{tikzpicture} 
                \caption{Illustration of the placement of the new vertex $v'$ in a quadrilateral vertex split of $G_{i-1}$ to obtain another pointed pseudotriangulation in $\mathbf{O}$, where the new quadrilateral has the $\mathbb{A}$-graph shown in Figure~\ref{fig_degenquads}(c). }
 \label{fig:deg_quad_ppt_2}               
\end{center}
\end{figure}
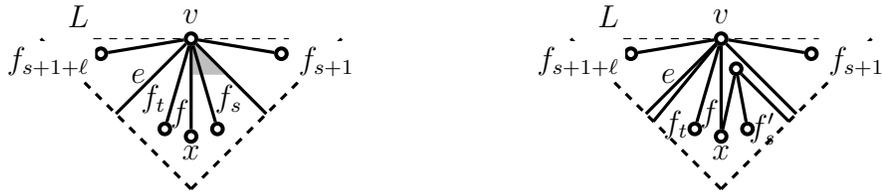

We may reformulate Proposition~\ref{prop:pt_nec} and Proposition~\ref{prop_pseudo} in terms of 
coverings of pointed pseudotriangulations in $\mathbf{O}$ as follows. We say that a $\Gamma$-symmetric geometric graph $G$ is a \emph{$\Gamma$-symmetric pointed pseudotriangulation in the plane} if its quotient graph $G/\Gamma$ is a pointed pseudotriangulation in the flat surface $\mathbf{O}$ of non-singular points of $\mathbb{R}^2/\Gamma$.

\begin{theorem}\label{ppt_covering} Let $\Gamma$ be generated by a translation or a $2$-fold rotation ($n$-fold rotation with $n\geq 3$) in the plane. Then the quotient $\mathbb{A}$-graph of a $\Gamma$-symmetric pointed pseudotriangulation in the plane is $(2,3,2)$-tight ($(2,3,1)$-tight, respectively). Conversely, for any  $(2,3,2)$-tight ($(2,3,1)$-tight, respectively) $\mathbb{A}$-graph $\overline{G}$ there exists a $\Gamma$-symmetric pointed pseudotriangulation $G$ in the plane whose quotient $\mathbb{A}$-graph is $\overline{G}$.
\end{theorem}

\begin{remark} Since the orbifolds considered in this section arise from a discrete subgroup of the Euclidean group acting on the plane,  the cone angle of the flat cone $\mathbf{O}$ in Propositions ~\ref{prop:pt_nec} and \ref{prop_pseudo} is assumed be of the form $2\pi/k$ for $k\geq 2$.  
However, the proofs can easily be adapted to extend these results to flat cones with any cone angle $\alpha$, $0<\alpha \leq 2\pi$. Observe that for the number $c$ of convex angles in a pointed pseudotriangulation with $f$ faces   in  the cone $\mathbf{O}$ with cone angle $\alpha$ we have
$$c=\begin{cases}
   3(f-2)+1& \text{if } 0< \alpha <\pi\\
   3(f-2)+2& \text{if } \pi\leq \alpha <2\pi\\
   3(f-2)+3& \text{if } \alpha =2\pi
\end{cases},$$
so the corresponding sparsity counts change accordingly. 
In fact, this pattern for the counts continues in this fashion for cone angles $\alpha>2\pi$, with $c=3(f-2)+a$ if $(a-1)\pi\leq \alpha <a\pi$.
\end{remark}

\subsection{Applications in geometric rigidity theory} \label{subsec:app}

We now discuss some applications of the results in Section~\ref{subsec:pseudo} to the rigidity and flexibility analysis of symmetric bar-joint frameworks. {We refer the reader to the Handbook chapter on Rigidity and Scene Analysis by Schulze and Whiteley \cite[Chapter 61]{HB18} for a detailed summary of definitions and results in geometric rigidity theory.}

A \emph{(bar-joint) framework} in the plane is a pair $(G,p)$, where $G=(V,E)$ is a  simple undirected graph and $p:V\to \mathbb{R}^2$ is an embedding. We  think of $(G,p)$ as a collection of  fixed-length bars (corresponding to the edges of $G$) which are connected at their ends by pin joints (corresponding to the vertices of $G$). Note that a  framework in which no edges cross each other may be considered as a geometric graph in the plane. Loosely speaking, a  framework $(G,p)$ is \emph{rigid} if all edge-length preserving, continuous motions of  $(G,p)$ are trivial, i.e. rigid body motions in the plane, and \emph{flexible} otherwise. A framework $(G,p)$  is \emph{generic} if the coordinates of the points $p(v)$, $v\in V$, are algebraically independent over $\mathbb{Q}$.

It is well known that a graph $G$ is  $(2,3)$-tight if and only if $G$ is  minimally $2$-rigid, that is, any generic realisation of $G$ as a bar-joint framework in the plane is minimally rigid (in the sense that removing any edge yields a flexible framework) \cite{PG,Laman}.  Pointed pseudotriangulations allow us to give a geometric certificate for a planar graph to be  minimally $2$-rigid, since a planar graph is $(2,3)$-tight if and only if it can be embedded as a pointed pseudotriangulation in the plane.

Using the results in Section~\ref{subsec:pseudo} we may deduce symmetric analogues of this result. We need the following definitions. For an abstract group $C$, we say that a graph $G$ is \emph{$C$-symmetric} if there exists a group action $\theta:C\to \textrm{Aut}(G)$, where $\textrm{Aut}(G)$ denotes the group of automorphisms of $G$. We will assume throughout this section that $\theta$ is \emph{free}, i.e. it acts freely on the vertex and edge set of $G$. Let $G$ be a $C$-symmetric graph and suppose that $C$ also acts on $\mathbb{R}^d$ via a homomorphism $\tau:C\to O(\mathbb{R}^d)$. Then a framework $(G,p)$ is called  \emph{$C$-symmetric} (with respect to $\theta$ and $\tau$) if $$\tau(g)(p(v))=p(\theta(g)(v)) \textrm{ for all } g\in C \textrm{ and } v\in V.$$
 A $C$-symmetric framework $(G,p)$ is called \emph{(forced) $C$-rigid} if there are no non-trivial motions of $(G,p)$ that preserve the full symmetry of $(G,p)$, in the sense that all frameworks along the path are also $C$-symmetric (see \cite{BSflex} for more details). Further, $(G,p)$  is called \emph{$C$-generic} if the representatives for the $C$-orbits of vertices of $G$ are in generic position. A graph $G$ is called \emph{minimally $C$-rigid} (with respect to $\theta$ and $\tau$) if some (or equivalently, every) $C$-generic realisation of $G$ (with respect to $\theta$ and $\tau$) as a bar-joint framework is minimally $C$-rigid in the plane. $C$-rigidity has been studied extensively in various different contexts in recent years; see for example \cite{JKT,MT15,NS}. A detailed summary of results can be found in \cite[Chapter 62]{HB18}.


\begin{corollary} \label{cor:rigidity}
Let $G$ be a $C$-symmetric planar graph with respect to the free action $\theta:C\to \textrm{Aut}(G)$, where $C$ is a cyclic group of order $k\geq 1$, $k\neq 2$.
Then $G$ is  minimally $C$-rigid with respect to $\theta$ and $\tau$, where $\tau(C)$ is generated by a rotation of order $k$ in the plane, if and only if $G$ can be embedded as a
$\tau(C)$-symmetric pointed pseudotriangulation  in the plane (with respect to $\theta$).
\end{corollary}
\begin{proof} Given a $C$-generic (with respect to $\theta$ and $\tau$) minimally $C$-rigid framework with no edges crossing each other, we consider the corresponding geometric quotient graph $\overline{G}$ in the flat cone $\mathbf{O}$ of non-singular points in $\mathbb{R}^2/\tau(C)$. 
By \cite[Theorem 6.3]{JKT}, the $\mathbb{A}$-graph $\overline{G}'$ of $\overline{G}$ is $(2,3,1)$-tight.
Thus, by Corollary~\ref{ppt_covering}, there exists a $\tau(C)$-symmetric pointed pseudotriangulation in the plane (with respect to $\theta$) whose quotient $\mathbb{A}$-graph is 
$\overline{G}'$.
Conversely, if $G$ can be embedded as a $\tau(C)$-symmetric pointed pseudotriangulation in the plane (with respect to $\theta$), then, by Corollary~\ref{ppt_covering}, the quotient $\mathbb{A}$-graph of $G$ is  $(2,3,1)$-tight. Thus, by \cite[Theorem 6.3]{JKT}, $G$ is  minimally $C$-rigid with respect to $\theta$ and $\tau$.
\end{proof}

While rigidity always implies $C$-rigidity, the converse is not true in general. However, note that if $C$ is a cyclic group of order $k=1,3$, then $C$-rigidity is in fact equivalent to rigidity for $C$-generic frameworks. This is trivial for $k=1$ and was shown in \cite[Theorem 6.11]{ST} for $k=3$. 

For the case when $\tau(C)$ is generated by a rotation of order $2$, we conjecture that a result analogous to Corollary~\ref{cor:rigidity} may be established by allowing the action $\theta$ to be non-free on the edges of $G$, and hence allowing an edge of $\overline{G}$ to go through the cone point of $\mathbb{R}^2/\tau(\Gamma)$. By the transfer results for $C$-rigidity established in \cite{CNSW}, this result would then immediately also extend to the case of reflection symmetry. 

Rigidity analyses of periodic frameworks in the plane, or equivalently, frameworks on the flat torus, have also received a significant amount of attention in recent years, both under a fixed torus (see \cite{Ross,ST1,KST19}, for example) and a flexible torus (see \cite{BS10,MTadv}, for example). In particular, it was shown in \cite{BS_pseudo} that  a pointed pseudotriangulation on the flat torus has exactly one non-trivial motion under a fully flexible torus, and that this motion is expansive in the sense that it does not decrease the distance between any pair of vertices. We conjecture that this result extends to the case of pointed pseudotriangulations on the flexible flat cylinder. In the case when the cylinder is fixed, it follows from \cite[Theorem 2.4]{KST19} and the results in Section~\ref{subsec:pseudo} that generic realisations of $\overline{G}$ as frameworks on the cylinder (or equivalently, frameworks in the plane that are periodic in one direction) are minimally rigid if and only if $\overline{G}$ can be embedded  as a pointed pseudotriangulation in the flat cylinder. 

\begin{remark}
    We conclude this part of the paper by noting that the results of 
    Sections \ref{sec_nec}, \ref{sec_suff} and \ref{sec:pseudo}
    suggest several obvious lines of future work. In particular 
    it would be interesting to prove analogues 
    of Theorems \ref{thm_maintranslation},
    \ref{thm_mainrotation} and
    \ref{ppt_covering} for all of the 
    discrete subgroups of the Euclidean group. We have 
    conjectures for various cases. However the 
    inductive characterisations of the appropriate surface graphs
    seem to be significantly more challenging in these cases.
    \label{rem_future}
\end{remark}

\section{Unions and intersections}

The remainder of the paper is devoted to proving the results of 
Section \ref{sec_statements}. In this section we set out some elementary 
properties of balanced and unbalanced \( \mathbb A \)-graphs, and the associated 
gain sparsity counts. 
See \cite{Zas89}, \cite{JKT} and \cite{ST1} for more detail on gain graphs and 
associated matroids.

First we record the fundamental observation that for 
subgraphs  \( B,C \) of \( D \),
\begin{equation}
    f(B\cup C) +f(B \cap C) = f(B) +f(C).
    \label{eqn_incexcl}
\end{equation}
Now suppose that \( D \) is a \( (2,3) \)-sparse graph.
It is easy to see using 
Equation (\ref{eqn_incexcl}) that 
if \( B \) is \( (2,3) \)-tight then \( B \) is connected. 
Furthermore if \( B \) and \( C \) are both \( (2,3) \)-tight 
and \( B\cap C \) contains at least one edge then 
both \( B \cup C \) and \( B \cap C \) are \( (2,3) \)-tight.
We will generalise these observations
to \( (2,3,l) \)-sparse \( \mathbb A \)-graphs.
Throughout the section \( G \) is an \( \mathbb A  \)-graph and 
\( H \) and \( K \) are non-empty subgraphs of \( G \).

The following lemma is a special case of \cite[Lemma 2.4]{JKT}.
\begin{lemma}
    Suppose that \( H \) and \( K \) are both balanced and \( H\cap K \)
    is connected. Then \( H \cup K \) is also balanced.
    \label{lem_svk}
\end{lemma}

\begin{proof}
    Suppose that \( H \cup K\) is unbalanced. Let \( F \) be the 
    face of \( H \) that contains the ends of \( \mathbb A \). 
    There must be a path \( p \) in \( K \) joining vertices
    \( u,v \in \partial F \cap K \)
    such that \( \mathring p \subset 
    F \cap K\) and \( \mathring p \) separates the 
    ends of \( \mathbb A \) in \( F \), where \( \mathring p \)
    is the relative interior of \( p \). 
    Let \( q \) be a path in \( H\cap K \) joining \( u \) and \( v \). 
    The concatenation 
    of \( p \) and \( q \) forms a loop in \( K \) that separates the ends of \( \mathbb A \)
    contradicting the hypothesis that \( K \) is balanced.
\end{proof}

\begin{lemma}
    If \( G \) is \( (2,3,l) \)-tight for \( l \in \{1,2\} \) then 
    \( G \) is connected.
    \label{lem_tightconn}
\end{lemma}

\begin{proof}
    This a straightforward consequence of (\ref{eqn_incexcl}).
\end{proof}

\begin{lemma}
    Suppose that \( G \) is \( (2,3,2) \)-sparse and 
    that \( H,K \) are both \( (2,3,2)  \)-tight. 
    \begin{enumerate}
        \item If \( H \) and \( K \) are both unbalanced and
            \( H \cap K \) is non-empty, then \( H\cap K \) and
            \( H \cup K \) are both \( (2,3,2) \)-tight. Furthermore
            \( H \cap K \) is either unbalanced or consists of a
            single vertex.
        \item If at least one of \( H \) or \( K \) is balanced and 
            \( H \cap K \) has at least two vertices then 
            \( H \cup K \) is \( (2,3,2) \)-tight. Furthermore
            either \( H\cap K \) is \( (2,3,2) \)-tight or, 
            \( H\cap K \) consists of two isolated vertices and 
            \( H\cup K \) is unbalanced.
    \end{enumerate}
    \label{lem_232tightunion}
\end{lemma}

\begin{proof}
    If \( H \) and \( K \) are both unbalanced 
    then, by (\ref{eqn_incexcl}), \( f(H\cap K) + f(H\cup K) = 4 \)
    and therefore \( f(H\cup K) = f(H\cap K) = 2 \). Conclusion 
    1 follows easily.

    If \( H \) is balanced and \( K \) is unbalanced, we see that 
    \( f(H \cup K) + f(H \cap K) = 5 \). Now \( H \cap K \)
    is balanced and since \( H \cap K  \) has at least two vertices, 
    we necessarily have \( f(H \cap K) =3 \) and \( f(H \cup K ) = 2 \). 

    Finally 
    if \( H \) and \( K \) are both balanced then 
    \( f(H \cup K) +f(H \cap K ) = 6 \). Now if \( H\cap K \) is
    connected, then by Lemma \ref{lem_svk}, \( H \cup K \)
    is balanced and then \( f(H\cup K) = 
    f(H\cap K) = 3\). On the other hand, if \( H \cap K  \) 
    is disconnected,
    then we must have \( f(H\cap K) = 4 \) and \( f(H \cup K) = 2 \). 
    But \( H \cap K \) is a disconnected balanced graph, so it must 
    comprise two isolated vertices and \( H \cup K \) must 
    be unbalanced since \( f(H \cup K) <3 \).
\end{proof}

We have a similar statement for \( (2,3,1) \)-sparse graphs. The 
proof is a routine adaptation of the proof of Lemma 
\ref{lem_232tightunion} and we omit the details.

\begin{lemma}
    Suppose that \( G \) is \( (2,3,1) \)-sparse and that
    \( H,K \) are both \( (2,3,1)  \)-tight. 
    \begin{enumerate}
        \item If \( H \) and \( K \) are both unbalanced and
            \( H \cap K \) is non-empty, then \( H\cap K \) and
            \( H \cup K \) are both unbalanced and 
            \( (2,3,1) \)-tight.
        \item If at least one of \( H \) or \( K \) is balanced and 
            \( H \cap K \) has at least one edge then
            \( H \cup K \) is \( (2,3,1) \)-tight.
    \end{enumerate}
     \label{lem_231tightunion}
\end{lemma}

\section{Inductive constructions for tight graphs}
\label{sec_inductive}

Throughout this section \( G \) is a \( (2,3,l) \)-tight
\( \mathbb A \)-graph.
Since a balanced \( \mathbb A \)-graph is equivalent to 
a plane graph with a puncture in the unbounded face 
we can restate 
Theorem \( 4 \) of \cite{FJW}, which we will need later, as follows. 

\begin{theorem}[Fekete, Jord\'an, Whiteley]
    Suppose that \( G \) is a balanced \( (2,3,l) \)-tight
    \( \mathbb A \)-graph with at least 4 vertices.
    Then for each vertex \( v \) of \( G \) there are distinct 
    edges \( e_1, e_2 \), both not incident with \( v \) and
    triangles \( T_i \) containing \( e_i \) such that 
    \( G_{e_i,T_i} \) is also \( (2,3,l) \)-tight for \( i =1,2 \).
    \label{thm_strongFJW}
\end{theorem}

\subsection{Euler counts}

Let \( \mathbb S =\{ x \in \mathbb R^3: \| x\| =1\}\) be 
the standard \( 2 \)-sphere 
and suppose 
that \( G \) is a connected finite \( \mathbb S \)-graph
with at least one edge. In 
particular, all faces of \( G \) are cellular with positive degree.
Let \( f_i \)
be the number of faces of degree \( i \). Since we allow loop edges 
and parallel edges, it is possible that \( f_1 \) or \( f_2 \) 
are non-zero. Using Euler's polyhedral formula, 
\( \sum_{i \geq 1}f_i = 2+|E|-|V| \), together with 
\( \sum_{i\geq 1} if_i = 2|E| \) and \( f(G) = 2|V|-|E| \),
we have 
\begin{equation}\label{eqn_euler} 3f_1 +2f_2 +f_3 = 
8-2f(G) + \sum_{i\geq 5} (i-4)f_i\end{equation}
for a connected \( \mathbb S \)-graph with at least one edge.
From this we can deduce the following for \( \mathbb A \)-graphs. (See Figure~\ref{fig_gain_sparsity_ex} for some examples.)

\begin{lemma}
    Suppose that \( l \in \{1,2\} \)
    and that \( G \) is a \( (2,3,l) \)-tight 
    \( \mathbb A \)-graph with at least three vertices. 
    \begin{enumerate}[(1)]
        \item If \( G \) is balanced, it has at least 
            one triangular face.
        \item If \( G \) is unbalanced and has no triangular face
            then every cellular face has degree \( 4 \).
    \end{enumerate}
    \label{lem_23leulercount}
\end{lemma}

\begin{proof}
    Conclusion (1) is a standard fact about plane Laman graphs.
    So assume that \( G \) is unbalanced.
    For the purposes of the proof (as opposed to the statement), 
    think of \( G \) as an \( \mathbb S \)-graph with 
    two marked faces corresponding to the faces 
    that contain the ends of \( \mathbb A \).

    Suppose that \( l=2 \). In this case loop edges are 
    forbidden, so \( f_1 =0 \). Also if \( F \) is a 
    face of degree two then \( F \) must be one of 
    the marked faces, otherwise the boundary \( \partial F \) (recall Section~\ref{sec:surgr}) is 
    balanced and \( f(\partial F)  = 2 \). Conclusion (2)
    now follows easily from Equation (\ref{eqn_euler}).

    Now suppose that \( l=1 \) and \( f_3 = 0 \). 
    Then, using Equation (\ref{eqn_euler}), we have 
    $3f_1+2f_2 \geq 6$ with equality if and only if \( f_i = 0 \)
    for \( i \geq 5 \). If \( F \) is a face of 
    \( G \) with \( |F| = 2 \) then, 
    since \( G \) has at least three vertices, 
    \( F \) must have non-degenerate boundary otherwise we have
    a vertex with two incident loop edges which is forbidden 
    by \( (2,3,1) \)-sparsity. Thus any face of degree at most two must 
    be one of the marked faces. It follows that \( f_1 =2 \), 
    \( f_2 = 0 \), and thus as remarked above, 
    \( f_i = 0 \) for \( i \geq 5 \).
\end{proof}

\subsection{Triangles}

Suppose that \( G \) is a \( (2,3,l) \)-sparse \( \mathbb A \)-graph and
that \( T \) is a triangular face of \( G \). 
If \( l = 2 \) then
the boundary of \( T \) must be 
non-degenerate (i.e. there is no repeated vertex in the
boundary walk). In the case \( l=1 \), either 
\( T \) is non-degenerate or \( \partial T \) is 
isomorphic to Figure \ref{fig_degenquads}(d).

\begin{lemma}
    Suppose that \( G \) is a \( (2,3,1) \)-sparse \( \mathbb A \)-graph and that 
    \( T \) is a degenerate triangular face. If \( e \) is a non-loop 
    edge of \( T \) then \( G_{e,T} \) is also \( (2,3,1) \)-sparse.
    \label{lem_231degentri}
\end{lemma}

\begin{proof}
    Let \( u,v \) be the vertices of \( \partial T \) and let \( z \)
    be the corresponding contracted vertex of \( G_{e,T} \).
    If \( H \) is a (balanced or unbalanced) 
    subgraph of \( G_{e,T} \) that violates the \( (2,3,1) \)-sparsity
    count then it is easy to see that \( V(H)-\{u,v\}\cup \{z\} \)
    spans a subgraph \( G \) that also violates the \( (2,3,1) \)-sparsity
    count.
\end{proof}

So we can assume from now on that all triangular faces are non-degenerate.

\begin{lemma}
    Suppose that \( G \) is a \( (2,3,l) \)-sparse \( \mathbb A \)-graph
    with a non-degenerate triangular face \( T \).
    Then \( G_{e,T} \) is not \( (2,3,l) \)-sparse if and only if 
    there is a \( (2,3,l) \)-tight 
    subgraph \( B \) of \( G \) such that \( E(B) \cap E(\partial T) 
    =\{e\}\) and \( |E(B)| \geq 2 \).
    \label{lem_23ltriangleblocker}
\end{lemma}

\begin{proof}
    The ``if'' direction is straightforward.
    Suppose that \( V(e) = \{u,v\} \).
    Let \( z \) be the vertex of \( G_{e,T} \) corresponding to 
    \( e \) and let \( e' \) be the edge of \( \partial T \) 
    that remains in \( G_{e,T} \). Now let \( A \) 
    be a subgraph of \( G_{e,T} \) 
    that violates the \( (2,3,l) \)-sparsity count. 
    Clearly \( z \in V(A) \) and \( e' \not\in E(A) \), for
    otherwise \( G \) would have a subgraph that 
    violates the \( (2,3,l) \)-sparsity count. 
    Let 
    \( B \) be the subgraph of \( G \) defined as follows.
    Let \( V(B) = V(A) - \{z\} \cup \{u,v\} \) and \( E(B) =
    E(A) \cup \{e\}\) where we identify any edge of 
    \( G_{e,T}\) with the corresponding edge of \( G \). 
    Now it is clear that \( B \) is balanced if and only 
    if \( A \) is balanced. Moreover \( f(A) = f(B)-1 \)
    and since \( A \) violates the \( (2,3,l) \)-sparsity
    count we have \( |E(A)| \geq 1 \). Therefore
    \( |E(B)| \geq 2 \) and \( B \) is \( (2,3,l) \)-tight.
\end{proof}

Note that if \( l=2 \) then we can assume that the subgraphs 
\( A \) and \( B \) from the proof above are induced subgraphs 
of \( G_{e,T} \) and \( G \) respectively. In particular \( B \)
does not contain the vertex of \( T \) that is not incident to 
\( e \) in this case. This is not necessarily 
true when \( l =1 \). See Figure \ref{fig_231examples} for an 
example.

The graph \( B \) whose existence is asserted by Lemma 
\ref{lem_23ltriangleblocker} is called a {\em blocker}
for the contraction \( G_{e,T} \). Observe that 
\( B \) has a face that properly contains the face \( T \) of \( G \).

\begin{lemma}
    Suppose that \( G \) is \( (2,3,l) \)-tight and that \( B \) is a 
    blocker for \( G_{e,T} \) that is maximal with respect to inclusion
    among all such blockers. If \( F \) is a face of \( B \) that does 
    not contain \( T \) then \( F \) is also a face of \( G \).
    \label{lem_maxblocker23l}
\end{lemma}

\begin{proof}
    We will deal with the case in which  \( G \) is unbalanced. 
    The argument for the balanced case is similar and easier.
    Let \( H \), respectively \( K \), 
    be the subgraph of \( G \) consisting of \( \partial F \)
    together with all edges and vertices of \( G \) 
    that are inside, respectively outside, \( F \).
    Observe that \( H\cup K = G \) and \( H \cap K = \partial F \). 
    Now \(
    l = f(G) = f(H\cup K) = f(H) +f(K) -f(\partial F)\). 
    But \( B\cap H = \partial F \) also, so \( f(B\cup H)
     = f(B) + f(H) - f(\partial F)\). 
    Combining these we see that 
    \begin{equation}
        f(B\cup H) = f(B) +l - f(K) \leq f(B)
        \label{eqn_add}
    \end{equation}
    since  \( f(K) \geq l \).
    Suppose that \( B \) and \( B\cup H \) are both balanced or both unbalanced.
    Since \( B \) is \( (2,3,l) \)-tight 
    it must be that \( f(B \cup H) = f(B) \). Now since \( F \) is
    not the face of \( B \) that contains \( T \) it follows that 
    \( E(B \cup H \cap \partial T) = E(B \cap \partial T) = \{e\}  \) and
    so \( B\cup H \) is a blocker for \( G_{e,T} \). Since 
    \( B \) is maximal it follows that \( H \subset B \) as required.

    The only other possibility is that \( B \) is balanced and \( B \cup H \) is
    unbalanced. 
    In this case, \( F \) must be the face of \( B \) that contains
    both ends of \( \mathbb A \).
    Since \( F \) is also a face of \( K \), 
    it follows that \( K \) must also be balanced 
    and  \( f(K) \geq 3\). Now the first equation in (\ref{eqn_add}) yields 
    \( f(B \cup H) \leq l\).
    Therefore \( B \cup H \) is 
    an unbalanced blocker that strictly contains \( B \), contradicting the maximality 
    of \( B \).
\end{proof}

We note that the case of Lemma \ref{lem_maxblocker23l} 
in which \( G \) is balanced is equivalent to Lemma 9 of 
\cite{FJW}.

In the proof of the next proposition and several times in the remainder of the paper we use the following simple observation. 
Suppose that
\( H \) is a balanced subgraph of \( G \)
and \( F \) is a cellular face of \( G \) such that
\( H \) contains all but one of the edges of \( \partial F \).
Then \( H \cup \partial F \)
is also balanced since adding the remaining edge cannot separate the
ends of \( \mathbb A \).

\begin{proposition}
    Suppose that \( G \) is a \( (2,3,l) \)-tight \( \mathbb A \)-graph that 
    has at least one triangular face. Then \( G \) has a triangular 
    face \( T \) and an edge \( e \in \partial T \) such that 
    \( G_{e,T} \) is \( (2,3,l) \)-tight.
    \label{prop_23ltriangle}
\end{proposition}

\begin{proof}
    By Theorem \ref{thm_strongFJW} we can assume that 
    \( G \) is unbalanced.
    Suppose that \( S \) is a triangular face and that 
    \( G_{e,S} \) is not \( (2,3,l) \)-sparse for all \( e \in 
    \partial S\).
    Let \( B_1 \), respectively \( B_2 \), be blockers for 
    two of the possible contractions of \( S \). It is 
    clear that \( B_1 \cap B_2 \) is non-empty and that 
    \( V(\partial S) \subset V(B_1\cup B_2) \) but 
    \( E(\partial S) \not\subset E(B_1\cup B_2) \). If both 
    \( B_1 \) and \( B_2 \) are unbalanced then, by 
    Lemma \ref{lem_232tightunion} or Lemma \ref{lem_231tightunion}, 
    \( B_1 \cup B_2 \) is
    tight and unbalanced and so must be an induced subgraph
    of \( G \) which is a contradiction. Thus we can conclude that 
    there is some edge \( f \in \partial S \) 
    such that 
    any blocker for \( G_{f,S} \) is balanced. 
    Let \( B \) be a maximal blocker for
    \( G_{f,S} \). 
    So \( B \) is a balanced \( (2,3,l) \)-tight graph. 
    Let \( U \) be the face of \( B \) that contains the 
    ends of \( \mathbb A \).
    By Lemma \ref{lem_maxblocker23l}, and since \( G \) is unbalanced,
    we see that \( U \) contains \( S \) and that all the cellular 
    faces of \( B \) are also faces of \( G \). 
    Let \( u \) be a vertex incident to \( f \).
    Using 
    Theorem \ref{thm_strongFJW}, there is a triangular 
    face \( T \) of \( B \) and \( e \in \partial T \) 
    such that \( e \) is not incident to \( u \) and 
    \( B_{e,T} \) is \( (2,3,l) \)-tight. 
    Since \( T \) is a cellular face of \( B \), it is also a
    face of \( G \).
    
    We will show that \( G_{e,T} \) is \( (2,3,l) \)-sparse. 
    Suppose that \( C \) is a blocker for \( G_{e,T} \). 
    By Lemma \ref{lem_232tightunion} or Lemma \ref{lem_231tightunion}
    we see that both \( B \cup C  \) and 
    \( B \cap C \) are \( (2,3,l) \)-tight. Now \( B_{e,T} \)
    is \( (2,3,l) \)-sparse and therefore \( B \cap C \) 
    cannot be a blocker for this contraction.
    It follows that \( B \cap C = \{e\} \) and in particular 
    \( u \not\in V(C) \).

    Suppose, seeking a contradiction,  that \( E(C)\cap E(\partial S) \) 
    is nonempty. Clearly \( f \not\in E(C)\)
    since \( C \) does not contain \( u \). 
    Therefore \( B \cup C  \) contains two of the edges of \( \partial S\)
    and (as observed above) it follows that 
    \( B \cup C \cup \partial S  \) is balanced if and 
    only if \( B \cup C  \) is balanced.
    But \( B \cup C \) is \( (2,3,l) \)-tight, so  
    \( B \cup C \cup \partial S = B \cup C \). Thus \( C \) must 
    contain two of the edges of \( \partial S \) and hence 
    \( u \in V(C) \) contradicting our earlier deduction. 
    Thus \( E(C) \cap E(\partial S) = \emptyset \) and so 
    \( B \cup C \) is a blocker for \( G_{f,S} \). By the 
    maximality of \( B \), we have \( C \subset B \) and 
    hence \( C = B \cap C = \{e\} \) which contradicts our choice 
    of \( C \) as a blocker for \( G_{e,T} \).
\end{proof}

\subsection{Quadrilaterals}

As previously noted, Proposition \ref{prop_23ltriangle} is not sufficient
to give a useful inductive characterisation of \( (2,3,l) \)-tight 
\( \mathbb A \)-graphs since there are infinitely many pairwise non-isomorphic 
examples that have no triangular faces. 

For the rest of this section 
suppose that \( Q \) is a quadrilateral face of 
\( G \) with boundary walk 
\( v_1\), \(e_1\), \(v_2\), \(e_2\), \(v_3\), \(e_3\), \(v_4\), \(e_4\), \(v_1 \). 
Note that we make no assumptions regarding the non-degeneracy of 
\( Q \). Any such assumption will be explicitly stated as needed.

%

In contrast with the case of triangles, 
for quadrilateral contractions there are sufficiently
many differences between the cases \( l=1 \) and \( l=2 \) to 
warrant separate treatments.

\subsubsection{\( l=2 \)}

\begin{lemma}
    \label{lem_232degen}
    Suppose that \( G \) is a \( (2,3,2) \)-sparse 
    \( \mathbb A \)-graph and that \( Q \) is a degenerate quadrilateral 
    face of \( G \). Then \( Q \) is isomorphic to the 
    \( \mathbb A \)-graph shown in Figure \ref{fig_degenquads}(a).
\end{lemma}

\begin{proof}
    First observe that \( v_i \neq v_{i+1} ,i=1,2,3\) and 
    \( v_1 \neq v_4 \) since loop edges are forbidden.
    Now 
    since \( \mathbb A \) is orientable it is clear that if \( e_1 = e_3 \)
    then \( v_2 = v_3 \) and so \( e_2 \) is a loop edge which is 
    forbidden in a \( (2,2) \)-sparse graph. Thus \( e_1 \neq e_3 \). 
    Similarly \( e_2 \neq e_4 \). 
    If \( e_1 = e_2 \) then 
    it is clear 
    that the walk \( v_1, e_ 3,v_4,e_4,v_1 \) bounds a cellular 
    region in \( \mathbb A \). So \( G(\{e_3,e_4\}) \) is
    balanced but \( f(G(\{e_3,e_4\})) \leq 2\) 
    contradicting the balanced sparsity count of \( G \).
    Thus \( e_1 \neq e_2 \) and similarly \( e_2 \neq
    e_3\), \( e_3 \neq e_4 \) and \( e_4 \neq e_1 \). 
    Thus \(| E(\partial Q) | = 4\). Now \( |V(\partial Q)| 
    \geq \frac12( |E(\partial Q)|+2) = 3\).
    The required conclusion follows easily.
\end{proof}

\begin{corollary}
    Suppose that \( G \) is a \( (2,3,2) \)-sparse 
    \( \mathbb A \)-graph and \( Q \) is a degenerate
    quadrilateral
    face of \( G \). Let \( v \) be the repeated vertex on the
    boundary walk of \( Q \). Then \( v \) is a cutvertex of 
    \( G \). 
    \label{cor_repeatcut}
\end{corollary}

\begin{figure}[htp]
    \centering

\begin{tabular}{cccc}
\begin{tikzpicture}
    \draw[cylinder] (0,0) -- (3,0) (0,2) -- (3,2);
    \coordinate [vertex] (v1) at (1.5,2);
    \coordinate [vertex] (v2) at (2,1);
    \coordinate [vertex] (v3) at (1.5,0);
    \coordinate [vertex] (v4) at (1,1);
    \draw[edge] (v1) -- (v2) -- (v3) -- (v4) -- (v1);
\end{tikzpicture}
&
\begin{tikzpicture}
    \draw[cylinder] (0,0) -- (3,0) (0,2) -- (3,2);
    \coordinate [vertex] (v1) at (1,2);
    \coordinate [vertex] (v2) at (2,2);
    \coordinate [vertex] (v3) at (2,0);
    \coordinate [vertex] (v4) at (1,0);
    \draw[edge] (v1) -- (v2) -- (v3) -- (v4) -- (v1);
\end{tikzpicture}
&
\begin{tikzpicture}
    \draw[cylinder] (0,0) -- (3,0) (0,2) -- (3,2);
    \coordinate [vertex] (v1) at (1,2);
    \coordinate [vertex] (v2) at (2,1.7);
    \coordinate [vertex] (v3) at (2,0.3);
    \coordinate [vertex] (v4) at (1,0);
    \draw[edge] (v1) -- (v2) -- (v3) -- (v4) -- (v1);
\end{tikzpicture}
&
\begin{tikzpicture}
    \draw[cylinder] (0,0) -- (3,0) (0,2) -- (3,2);
    \coordinate [vertex] (v1) at (1,2);
    \coordinate [vertex] (v2) at (2,1);
    \coordinate [vertex] (v4) at (1,0);
    \draw[edge] (v1) -- (v2) -- (v4) -- (v1);
\end{tikzpicture}
\\
(a) & (b) & (c) & (d)
\end{tabular}

\caption{Up to isomorphism these are the possible embeddings of a degenerate cellular face of degree at most four in a \( (2,3,1) \)-sparse \( \mathbb A \)-graph. In a \( (2,3,2) \)-sparse \( \mathbb A \)-graph (a) is the only possibility.}
    \label{fig_degenquads}
\end{figure}
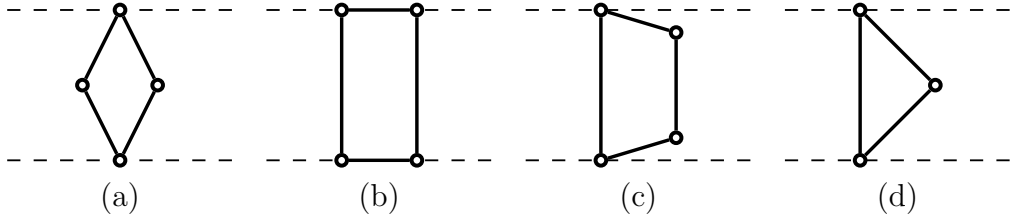

The situation with blockers for 
quadrilateral contractions is a little 
more complicated than for triangle contractions.

\begin{lemma}
    Suppose that \( G \) is a \( (2,3,2) \)-sparse \( \mathbb A \)-graph, 
    \( G \) is a quadrilateral face of \( Q \) with \( v_1 \neq v_3 \) 
    and \( G_{v_1,v_3,Q} \) is not \( (2,3,2) \)-sparse. Then 
    \( v_2 \neq v_4 \) and
    at least one of the following statements is true.
    \begin{enumerate}[(1)]
        \item There is some \( (2,3,2) \)-tight subgraph \( B \) of \( G \) such that \( \{v_1,v_3\} \subset V(B) \),  and 
            \( |\{v_2,v_4\} \cap V(B)| = 1  \).
        \item There is some balanced subgraph \( B \) of \( G \) such that
            \( B \) contains at least one edge, 
            \( f(B) = 4 \), \( \{v_1,v_3\} \subset V(B) \), 
            \( \{v_2,v_4\} \cap V(B) = \emptyset  \) and 
            so that \( B\cup \partial Q \) is a balanced subgraph of \( G \). 
        \item There is some subgraph \( B \) of \( G \) such that 
            \( f(B) = 3 \), \( \{v_1,v_3\} \subset V(B) \) and 
            \( \{v_2,v_4\} \cap V(B) = \emptyset  \) and so that \( B\cup \partial Q \) is an unbalanced subgraph of \( G \).
    \end{enumerate}
    \label{lem_blocker}
\end{lemma}

\begin{proof}
    Let \( z \) be the vertex of \( G_{v_1,v_3,Q} \)  that corresponds
    to \( v_1,v_3 \). Let \( A \) be a subgraph of \( G_{v_1,v_3,Q} \)
    that violates the \( (2,3,2) \)-sparsity count. Clearly
    the subgraph of \( G_{v_1,v_3,Q} \) induced by \( V(A) \) 
    also violates the \( (2,3,2) \)-sparsity count so we 
    assume that \( A \) is in fact an induced subgraph of 
    \( G_{v_1,v_3,Q }\).
 
    Let \( B \) be the induced subgraph of \( G \) on the vertex set 
    \( V(A) - \{z\} \cup \{v_1,v_3\} \). 
    Note that \( B \) contains an edge since \( A \) does. 

    Now suppose, seeking a contradiction, that \( v_2 = v_4 \). 
    By Corollary \ref{cor_repeatcut},
    \( v_2 \) is a cutvertex for \( G \). 
    Also if \( v_2 \in V(B) \) then \( f(B) = f(A) \) 
    and both \( A \) and \( B \) are unbalanced 
    which contradicts the sparsity of \( G \), so \(  v_2 \not\in 
    V(B)\). In particular \( v_1 \) and \( v_3 \) are in 
    different components of \( B \). Say \( v_1 \in B' \) 
    and \( v_3 \in B''\) where \( B =B' \cup B'' \) and 
    \( B' \cap B'' = \emptyset \). Now  \(f(A) = f(B') +f(B'') -2  \)
    and \( A \) is balanced
    if and only if  \( B' \) and \( B'' \) are both 
    balanced. Since one of \( B',B'' \) contains an edge, 
    it follows easily 
    that \( f(A) \geq 3 \) if \( A \) is balanced and 
    \( f(A) \geq 2 \) if \( A \) is unbalanced, contradicting 
    the choice of \( A \). Thus we have shown that \( v_2 \neq v_4 \).

    Since both \( A \) and \( B \) are induced subgraphs, we have
    \begin{equation}
        f(B) = f(A) +2 - n
        \label{eqn_count1}
    \end{equation}
    where \( n = |\{v_2,v_4\} \cap V(B)| \).
    Also we observe that \( A \) is balanced if and only 
    if \( B \cup \partial Q \) is balanced. 

    Now suppose that \( A \) is balanced. Then \( B \) is 
    also balanced and so \( f(B) \geq 3 \).  Together 
    with
    (\ref{eqn_count1}) this yields \( 3 \leq f(B) \leq f(A) +2 - 
    n\). If \( f(B) = 4 \) then
    \( f(A) = 2 \) and \( n = 0 \) and 
    (2) is true. If \( f(B) = 3 \)  and \( f(A) = 2 \)
    then \( n = 1\) and (1) is true. 
    If \( f(B) = 3 \)  and \( f(A) = 1 \)
    then \( n = 0\). Now replace \( B \) by 
    \( B\cup {v_2}\cup\{e_1,e_2\} \), which is also balanced
    since \( A \) 
    is balanced, and (1) is true.

    If \( A \) is unbalanced then 
    \( B \cup \partial Q \) is also unbalanced and \( f(A) \leq 1 \).
    Using (\ref{eqn_count1}) we see that either \( f(A) =n=0 \) or 
    \( f(A) = 1 \) and \( n =0 \) or \( 1 \).
    If \( f(A) = 1 \) and \( n=1 \) then \( f(B) = 2 \) 
    and (1) is true.
    If \( f(A) = 0 \) and \( n=0 \) then \( f(B) =2 \). Now replace 
    \( B \) by \( B\cup {v_2}\cup\{e_1,e_2\} \) 
    and again (1) is true.
    Finally if \( f(A) =1 \) and \( n=0 \), then \( f(B) = 3 \) and 
    (3) is true.
\end{proof}
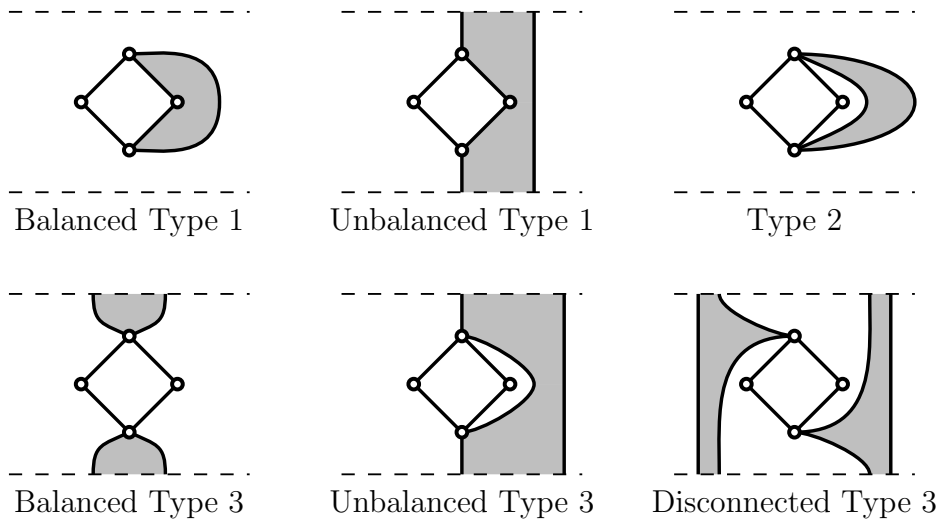
\begin{figure}[htp]
    \centering

    \begin{tabular}{cccc}

\begin{tikzpicture}[x=0.8cm,y=0.8cm]
    \clip (-0.5,-1) rectangle (4.5,3.5);
    \draw[cylinder] (0,0) -- (4,0) (0,3) -- (4,3);
    \coordinate (v4) at (1.2,1.5);
    \coordinate (v1) at (2,2.3);
    \coordinate (v2) at (2.8,1.5);
    \coordinate (v3) at (2,0.7);
    \filldraw[graphfill] (v1) .. controls +(0.5,0) and +(0,1) .. (3.5,1.5) -- (v2) -- (v1) --cycle;
    \filldraw[graphfill] (3.5,1.5) .. controls +(0,-1) and +(0.5,0) .. (v3) -- (v2) -- (3.5,1.5) --cycle;

    \draw[edge] (v1) -- (v2) -- (v3) -- (v4) -- (v1);
    \draw[edge] (v1) .. controls +(0.5,0) and +(0,1) .. (3.5,1.5) .. controls +(0,-1) and +(0.5,0) .. (v3);
    \node[vertex] at (v1){};
    \node[vertex] at (v2){};
    \node[vertex] at (v3){};
    \node[vertex] at (v4){};
    \node (label) at (2,-0.5){ Balanced Type 1};
\end{tikzpicture}

&

\begin{tikzpicture}[x=0.8cm,y=0.8cm]
    \clip (-0.5,-1) rectangle (4.5,3.5);
    \coordinate (v4) at (1.2,1.5);
    \coordinate (v1) at (2,2.3);
    \coordinate (v2) at (2.8,1.5);
    \coordinate (v3) at (2,0.7);
    \fill[graphfill] (v1) -- (2,3) -- (3.2,3) -- (3.2,1.5) -- (v2)-- (v1) -- cycle;
    \fill[graphfill] (v3) -- (2,0) -- (3.2,0) -- (3.2,1.5) -- (v2)-- (v3) -- cycle;
    \draw[edge] (v1) -- (2,3) (3.2,3) -- (3.2,0) (v3) -- (2,0);

    \draw[cylinder] (0,0) -- (4,0) (0,3) -- (4,3);
    \draw[edge] (v1) -- (v2) -- (v3) -- (v4) -- (v1);
    \node[vertex] at (v1){};
    \node[vertex] at (v2){};
    \node[vertex] at (v3){};
    \node[vertex] at (v4){};
    \node (label) at (2,-0.5){Unbalanced Type 1};
    
\end{tikzpicture}

&

\begin{tikzpicture}[x=0.8cm,y=0.8cm]
    \clip (-0.5,-1) rectangle (4.5,3.5);
    \coordinate (v4) at (1.2,1.5);
    \coordinate (v1) at (2,2.3);
    \coordinate (v2) at (2.8,1.5);
    \coordinate (v3) at (2,0.7);
    \filldraw[graphfill] (v1) .. controls +(1,0) and +(0,0.5) .. (4,1.5) -- (3.2,1.5) .. controls +(0,0.3) and +(0.7,-0.3) .. (v1) -- cycle;
    \filldraw[graphfill] (v3) .. controls +(1,0) and +(0,-0.5) .. (4,1.5) -- (3.2,1.5) .. controls +(0,-0.3) and +(0.7,0.3) .. (v3) -- cycle;

    \draw[edge] (v1) .. controls +(1,0) and +(0,0.5) .. (4,1.5) .. controls +(0,-0.5) and +(1,0) .. (v3);
    \draw[edge] (v1) .. controls +(0.7,-0.3) and +(0,0.3) .. (3.2,1.5) .. controls +(0,-0.3) and +(0.7,0.3) .. (v3);

    \draw[cylinder] (0,0) -- (4,0) (0,3) -- (4,3);
    \draw[edge] (v1) -- (v2) -- (v3) -- (v4) -- (v1);
    \node[vertex] at (v1){};
    \node[vertex] at (v2){};
    \node[vertex] at (v3){};
    \node[vertex] at (v4){};
    \node (label) at (2,-0.5){Type 2};
\end{tikzpicture}

\\

\begin{tikzpicture}[x=0.8cm,y=0.8cm]
    \clip (-0.5,-1) rectangle (4.5,3.5);
    \coordinate (v4) at (1.2,1.5);
    \coordinate (v1) at (2,2.3);
    \coordinate (v2) at (2.8,1.5);
    \coordinate (v3) at (2,0.7);
    \fill[graphfill] (v1) .. controls +(155:0.5) and +(-90:0.5) .. (1.4,3) -- (2.6,3) .. controls +(-90:0.5) and +(25:0.5) .. (v1);
    \fill[graphfill] (v3) .. controls +(205:0.5) and +(90:0.5) .. (1.4,0) -- (2.6,0) .. controls +(90:0.5) and +(-25:0.5) .. (v3);

    \draw[edge] (v1) .. controls +(155:0.5) and +(-90:0.5) .. (1.4,3) (2.6,3) .. controls +(-90:0.5) and +(25:0.5) .. (v1);
    \draw[edge] (v3) .. controls +(-155:0.5) and +(90:0.5) .. (1.4,0) (2.6,0) .. controls +(90:0.5) and +(-25:0.5) .. (v3);

    \draw[cylinder] (0,0) -- (4,0) (0,3) -- (4,3);
    \draw[edge] (v1) -- (v2) -- (v3) -- (v4) -- (v1);
    \node[vertex] at (v1){};
    \node[vertex] at (v2){};
    \node[vertex] at (v3){};
    \node[vertex] at (v4){};
    \node (label) at (2,-0.5){Balanced Type 3};
    
\end{tikzpicture}

&

\begin{tikzpicture}[x=0.8cm,y=0.8cm]
    \clip (-0.5,-1) rectangle (4.5,3.5);
    \coordinate (v4) at (1.2,1.5);
    \coordinate (v1) at (2,2.3);
    \coordinate (v2) at (2.8,1.5);
    \coordinate (v3) at (2,0.7);
    \fill[graphfill] (2,0) -- (v3) .. controls +(0:0.3) and +(-90:0.3) .. (3.2,1.5) -- (3.7,1.5) -- (3.7,0) -- cycle;
    \fill[graphfill] (2,3) -- (v1) .. controls +(0:0.3) and +(90:0.3) .. (3.2,1.5) -- (3.7,1.5) -- (3.7,3) -- cycle;
    \draw[graphfill] (3.2,1.5) -- (3.7,1.5);

    \draw[edge] (v1) -- (2,3) (2,0) -- (v3) .. controls +(0:0.3) and +(-90:0.3) .. (3.2,1.5) .. controls +(90:0.3) and +(0:0.3) .. (v1) (3.7,3) -- (3.7,0);  
    \draw[cylinder] (0,0) -- (4,0) (0,3) -- (4,3);
    \draw[edge] (v1) -- (v2) -- (v3) -- (v4) -- (v1);
    \node[vertex] at (v1){};
    \node[vertex] at (v2){};
    \node[vertex] at (v3){};
    \node[vertex] at (v4){};
    \node (label) at (2,-0.5){Unbalanced Type 3};
    
\end{tikzpicture}

&
\begin{tikzpicture}[x=0.8cm,y=0.8cm]
    \clip (-0.5,-1) rectangle (4.5,3.5);
    \coordinate (v4) at (1.2,1.5);
    \coordinate (v1) at (2,2.3);
    \coordinate (v2) at (2.8,1.5);
    \coordinate (v3) at (2,0.7);
    \fill[graphfill] (v1) .. controls +(180:1.4) and +(90:1) .. (0.75,0) -- (0.4,0) -- (0.4,3) -- (0.75,3) .. controls +(-90:0.3) and +(180:0.3) .. (v1);
    \fill[graphfill] (v3) .. controls +(0:1.4) and +(-90:1) .. (3.25,3) -- (3.6,3) -- (3.6,0) -- (3.25,0) .. controls +(90:0.3) and +(0:0.3) .. (v3);

    \draw[edge] (v1) .. controls +(180:1.4) and +(90:1) .. (0.75,0) (0.4,0) -- (0.4,3) (0.75,3) .. controls +(-90:0.3) and +(180:0.3) .. (v1);
    \draw[edge] (v3) .. controls +(0:1.4) and +(-90:1) .. (3.25,3) (3.6,3) -- (3.6,0) (3.25,0) .. controls +(90:0.3) and +(0:0.3) .. (v3);

    \draw[cylinder] (0,0) -- (4,0) (0,3) -- (4,3);
    \draw[edge] (v1) -- (v2) -- (v3) -- (v4) -- (v1);
    \node[vertex] at (v1){};
    \node[vertex] at (v2){};
    \node[vertex] at (v3){};
    \node[vertex] at (v4){};
    \node (label) at (2,-0.5){Disconnected Type 3};
    
\end{tikzpicture}
    \end{tabular}
    \caption{The topology of blockers for a quadrilateral contraction. The shaded region in each diagram stands for a subgraph of \( G \) that is a blocker for the contraction \( G_{v_1,v_3,Q} \) where \( v_1 \) and \( v_3 \) are the top and bottom vertices of the quadrilateral. 
    Note that a balanced type 3 blocker cannot arise in the context of \( (2,3,1) \)-sparsity. On the other hand, a disconnected type 3 blocker cannot arise in the context of \( (2,3,2) \)-sparsity.}
    \label{fig_quadblockertypes}
\end{figure}

We refer to the graph \( B \) whose existence is asserted by Lemma
\ref{lem_blocker} as type 1/2/3 blocker according to whichever case
of the lemma applies. Note that for a given blocker \( B \) 
exactly one of (1)-(3) is true. See Figure \ref{fig_quadblockertypes}
for some schematic diagrams indicating the topological embedding of
various types of blockers. Note that these diagrams and those in Figures 
\ref{fig_quad_prop_232} and \ref{fig_231case12} are meant only as 
aids to the topological intuition of the reader. 
We do not rely on the faithfulness 
of any of these diagrams for the proofs in this section.
We collect some observations about the blockers in the following lemmas.

\begin{lemma}
    In all cases of Lemma \ref{lem_blocker} the blocker is connected.
    \label{lem_quadblockersconnected}
\end{lemma}

\begin{proof}
    This follows easily using Equation (\ref{eqn_incexcl})
    and the \( (2,3,2) \)-sparsity of \( G \).
\end{proof}


\begin{lemma}
    Suppose that \( B \) is a blocker of type 2 or 3 for \( G_{v_1,v_3,Q} \).
    Then \( V(B) \) separates \( v_2 \) and \( v_4 \) in \( G \).
    \label{lem_blockersep}
\end{lemma}
\begin{proof}
    Consider the surface graph \( G\cup \delta \) where \( \delta \) is a new edge 
    embedded as a diagonal of \( Q \) joining \( v_1 \) and \( v_3 \). 
    Since \( B \) is connected by Lemma \ref{lem_quadblockersconnected},
    we can find a cycle \( C \) in \( B \cup \delta \) that contains the edge \( \delta \).
    Since \( \mathbb A \) has genus zero 
    it follows that \( |C| \) is a loop that separates
    \( v_2 \) from \( v_4 \) in the surface. In particular 
    any path in \( G \) from 
    \( v_2 \) to \( v_4 \) must pass through some 
    vertex of \( C \).
    But \( V(C) \subset V(B) \).
\end{proof}


\begin{proposition}
    \label{prop_quad1}
    Suppose that \( G \) is a \( (2,3,2) \)-sparse 
    \( \mathbb A \)-graph and that \( Q \) is a quadrilateral face of \( G \) 
    with boundary vertices \( v_1,v_2,v_3,v_4 \) such that 
    \( v_1 \neq v_3 \). If 
    \( G_{v_1,v_3,Q} \) is not \( (2,3,2) \)-sparse then 
    \( G_{v_2,v_4,Q} \) is \( (2,3,2) \)-sparse.
\end{proposition}
\begin{proof}
    Note that we  proved that \( v_2 \neq v_4 \) in Lemma \ref{lem_blocker}, so \( G_{v_2,v_4,Q} \) is well defined.
    Now suppose that \( G_{v_2,v_4,Q} \) is not \( (2,3,2) \)-sparse. 
    Then by 
    Lemma \ref{lem_blocker} there are blockers \( B_1 \), respectively 
    \( B_2 \), for the contractions \(G_{v_1,v_3,Q} \), 
    respectively 
    \( G_{v_2,v_4,Q} \). 
    Observe that \( V(\partial Q) 
    \subset V(B_1 \cup B_2) \), 
    Therefore \begin{eqnarray}
        f(B_1 \cup B_2 \cup \partial Q) &=& f(B_1 \cup B_2) - d \notag \\
        &=& f(B_1) +f(B_2) -f(B_1 \cap B_2) -d
        \label{eqn_def} 
    \end{eqnarray} 
    where
    \( d \) is the number of edges of \( \partial Q\) 
    that are not in \( B_1 \cup B_2\).
    In fact \( d = 1 \), \( 2 \) or \( 4 \) 
    depending in an obvious way on the types of \( B_1  \)
    and \( B_2 \). 
    Now there are six cases to consider depending 
    on the types of the respective blockers. We will derive a 
    contradiction in each of these.
    In the following list ``Case \( (X, Y) \)'' means that \( B_1 \) is
    a type \(X\) blocker and \( B_2 \) is a type \( Y \) blocker.
    Note that, by Lemmas \ref{lem_quadblockersconnected} and 
    \ref{lem_blockersep}, \( B_1 \cap B_2 \)
    is nonempty in all cases.

    Case \( (1,1) \) :
    In this case \( d = 1 \). First 
    observe that 
    \( B_1 \cup B_2 \cup \partial Q \) is balanced if and only if 
    \( B_1 \cup B_2 \) is balanced. 
    Now \( B_1 \cap B_2 \)
    contains an edge of \( \partial Q \), so 
    by Lemma \ref{lem_232tightunion} \( B_1
    \cup B_2\) is \( (2,3,2) \)-tight. 
    Thus (\ref{eqn_def}) yields the required contradiction.

    Case \( (1,2) \): Without loss of generality suppose 
    \( v_2 \in B_1 \). In this case \( d = 2 \) so  
    using (\ref{eqn_def}) we have 
    \begin{equation}
        f(B_1 \cup B_2 \cup \partial Q) 
        = f(B_1) +2 - f(B_1 \cap B_2) \label{eqn_case12}
    \end{equation}
    Now since \( B_1 \) is tight we have \( f(B_1) 
    \leq 3\). It follows from (\ref{eqn_case12}) that \( f(B_1\cap
    B_2) \leq 3\) and so \( B_1 \cap B_2 \) must be connected.
    Now \( B_1 \cap \partial Q \) is connected and \( 
    (B_1 \cap B_2) \cap (B_1 \cap \partial Q) \) is non-empty: it
    contains  \( v_2 \).
    Thus \( B_1 \cap (B_2 \cup \partial Q) = (B_1 \cap B_2) \cup (B_1 
    \cap \partial Q)\) is also connected. Since \( B_2 \cup \partial Q \)
    is balanced it follows from Lemma \ref{lem_svk} that 
    \( B_1 \cup B_2 \cup \partial Q \) is balanced if and 
    only if \( B_1 \) is balanced. In particular \( f(B_1\cup B_2 \cup 
    \partial Q) \geq f(B_1)\) since \( B_1 \) is \( (2,3,2) \)-tight.
    It follows 
    from (\ref{eqn_case12}) that \( f(B_1 \cap B_2) 
    \leq 2\). 
    But since \( B_1 \cap B_2 \) is balanced, we conclude that
    \( B_1 \cap B_2 = \{v_2\}\).
    In particular, 
    by Lemma \ref{lem_blockersep}, \( v_2 \) is a
    cutvertex for \( B_1 \). 
    Now if \( B_1 \) is balanced it is a Laman
    graph and so it cannot have a cutvertex,
    so \( B_1 \) must be unbalanced.
    Since \( f(B_1) = 2 \), it follows from 
    Equation (\ref{eqn_incexcl}) 
    that \( B_1 = B' \cup B'' \) where \(f(B') =f(B'') =2 \) 
    and \( B' \cap B'' = \{v_2\} \). Also \( B' \) and 
    \( B'' \) each contain one edge of \( \partial Q \), 
    so they are both unbalanced.

    Now \( B_2 \) is connected by Lemma \ref{lem_quadblockersconnected},
    so by concatenating a path in \( B_2 \) joining \( v_2 \) 
    and \( v_4 \) with the diagonal of \( Q \) we form a 
    cycle, \( C \), that separates (in \( \mathbb A \)) any 
    point in \( B' - \{v_2\} \) from any point in \( B''-\{v_2\} \):
    see Figure \ref{fig_quad_prop_232} for an illustration.
    Now \( C \) is balanced since \( B_2 \cup \partial Q \) is balanced.
    It follows that at least one of \( B' \) or \( B'' \) 
    is balanced, contradicting our earlier deduction.
%
%
%

    Case \( (1,3) \): In this case \( d = 2 \) so (\ref{eqn_def})
    yields \( f(B_1 \cup B_2 \cup \partial Q)
    =  f(B_1) - f(B_1 \cap B_2) +1\). 
    Now \( B_1 \cap B_2 \) is non-empty, so \( f(B_1 \cap B_2)
    \geq 2\). Therefore \( f(B_1) \geq 3 \). 
    But since \( B_1 \) is \( (2,3,2) \)-tight it 
    follows that \( B_1 \) must be balanced, 
    and in fact \( f(B_1 \cap B_2) = 2 \). Therefore
    \( B_1 \cap B_2 \) must be a single vertex, which 
    by Lemma \ref{lem_blockersep} is a cutvertex for 
    \( B_1 \). However, we have shown that \( B_1 \)
    is balanced and so is a Laman graph, which cannot 
    have a cutvertex.

    Case \( (2,2) \): In this case, \( d = 4 \) and using 
    (\ref{eqn_def}) we have \( f(B_1 \cup B_2
    \cup \partial Q) = 4 - f(B_1 \cap B_2)\) so 
    \( f(B_1 \cap B_2) = 2 \). 
    Since \( B_1 \cap B_2 \) is balanced, it follows that
    \( B_1 \cap B_2 \) is a single vertex, say \( w \),
    and that \( f(B_1 \cup B_2 \cup \partial Q ) = 2 \).
    Now by 
    Lemma \ref{lem_blockersep}, \( w \)
    is a cutvertex for both \( B_1 \) and \( B_2 \).
    So \( B_1 = B' \cup B'' \) where \( B' \cap B'' = \{w\} \),
    \( v_1 \in B' \), \( v_3 \in B'' \) and 
    both \( B' \) and \( B'' \) are Laman graphs and thus connected.
    Since \( B_2 \cup \partial Q \) is balanced it has a face \( U \)
    that contains both ends of \( \mathbb A \).
    It is clear that one of \( B',B'' \), without loss of 
    generality say \( B' \), is disjoint 
    from \( U \). Therefore \( B_2 \cup \partial Q \cup B' \) is balanced. 
    Now \( B_1 \cup B_2 \cup \partial Q 
    = (B_1 \cup \partial Q) \cup (B_2 \cup \partial Q \cup B')\). 
    But \( B_1 \cup \partial Q \) is balanced, we have 
    just seen that \( B_2 \cup \partial Q \cup B' \) is balanced,
    and \( (B_1 \cup \partial Q) \cap (B_2 \cup \partial Q \cup B')
    = \partial Q \cup B'\) which is connected. By Lemma 
    \ref{lem_svk}
    \( B_1 \cup B_2 \cup \partial Q \) is balanced, contradicting 
    our earlier deduction that \( f(B_1 \cup B_2 \cup \partial Q)
    =2\).

    Cases \( (2,3) \) and \( (3,3) \): In these cases \( d = 4 \) and 
    (\ref{eqn_def}) yields \( f(B_1 \cup B_2 \cup \partial Q ) 
    \leq  3- f(B_1\cap B_2) \). Since \( B_1\cap B_2 \) is non-empty 
    by Lemma \ref{lem_blockersep}, we have \( f(B_1 \cap B_2) 
    \geq 2 \) yielding the desired contradiction.
\end{proof}
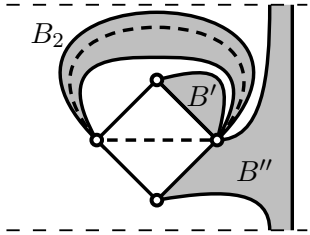
\begin{figure} [htp]
    \centering

\begin{tikzpicture}[x=1.0cm,y=1.0cm]
    \clip (-0.5,-0.5) rectangle (4.5,3.5);
    \coordinate (v4) at (1.2,1.2);
    \coordinate (v1) at (2,2.0);
    \coordinate (v2) at (2.8,1.2);
    \coordinate (v3) at (2,0.4);
    \fill[graphfill] (v1) .. controls +(15:0.9) and +(75:0.9) .. (v2) --cycle;
    \fill[graphfill](v2) .. controls +(45:1.5) and +(0:1) .. (2,2.9) .. controls +(180:1) and + (135:1.5) .. (v4) .. controls +(120:1.3) and +(180:0.5) .. (2,2.3) .. controls +(0:0.5) and +(60:1.3) .. (v2);
    \fill[graphfill] (v3) .. controls +(10:1) and +(90:0.5) .. (3.5,0) -- (3.8,0) -- (3.8,3) -- (3.5,3) .. controls +(-90:1) and +(0:0.6) .. (v2);

    \draw[edge] (v1) .. controls +(15:0.9) and +(75:0.9) .. (v2) --cycle;
    \draw[edge](v2) .. controls +(45:1.5) and +(0:1) .. (2,2.9) .. controls +(180:1) and + (135:1.5) .. (v4) .. controls +(120:1.3) and +(180:0.5) .. (2,2.3) .. controls +(0:0.5) and +(60:1.3) .. (v2);
    \draw[edge] (v3) .. controls +(10:1) and +(90:0.5) .. (3.5,0) (3.8,0) -- (3.8,3) (3.5,3) .. controls +(-90:1) and +(0:0.6) .. (v2);
    \draw[edge, dash pattern = on 4pt off 3pt] (v2) ..controls +(50:1.2) and +(0:1) .. (2,2.7) .. controls +(180:1) and +(130:1.2) .. (v4) -- (v2);

    \draw[cylinder] (0,0) -- (4,0) (0,3) -- (4,3);
    \draw[edge] (v1) -- (v2) -- (v3) -- (v4) -- (v1);
    \node[vertex] at (v1){};
    \node[vertex] at (v2){};
    \node[vertex] at (v3){};
    \node[vertex] at (v4){};
    \node at (2.6,1.8){\small $B'$};
    \node at (3.3,0.8){\small $B''$};
    \node at (0.55,2.6){\small $B_2$};
    
\end{tikzpicture}

\caption{Case \( (1,2) \) from the proof of Proposition \ref{prop_quad1}. The dotted loop is \( C \).} 
    \label{fig_quad_prop_232}
\end{figure}

It is worth noting that analogues of Proposition 
\ref{prop_quad1} fail for other similar classes of 
graphs. For example there are many examples of 
\( (2,2) \)-tight torus graphs with quadrilateral faces
for which both contractions yield graphs that 
are not \( (2,2) \)-sparse.
See \cite{CKPS} for details of this.

\subsubsection{\( l=1 \)}

In this 
subsection \( G \) will be a \( (2,3,1) \)-sparse \( \mathbb A \)-graph.
The general pattern of the arguments is similar to the 
\( (2,3,2) \)-sparse case. However, there are significant differences
in the details of the statements and proofs,
mostly due to the fact that if \( H \) is a balanced  \( (2,3,1) \)-tight 
subgraph of \( G \) then the induced subgraph \( G(H) \) need not 
be \( (2,3,1) \)-tight. This complicates some of the discussion since we 
cannot assume that a blocker is an induced subgraph and so it cannot 
be characterised by its set of vertices.
See Figure \ref{fig_231examples} for some examples.

\begin{figure}[htp]
    \centering

\begin{tabular}{ccc}
\begin{tikzpicture}
    \clip (-0.2,-0.2) rectangle (3.2,2.2);
    \draw[cylinder] (0,0) -- (3,0) (0,2) -- (3,2);
    \coordinate [vertex] (v1) at (1.5,2);
    \coordinate [vertex] (v2) at (2,1);
    \coordinate [vertex] (v3) at (1.5,0);
    \coordinate [vertex] (v4) at (1,1);
    \draw[edge] (v1) -- (v2) -- (v3) -- (v4) -- (v1) (v2) -- (v4); 
\end{tikzpicture}
&
\begin{tikzpicture}
    \clip (-0.2,-0.2) rectangle (3.2,2.2);
    \draw[cylinder] (0,0) -- (3,0) (0,2) -- (3,2);
    \coordinate [vertex] (v1) at (1.5,2);
    \coordinate [vertex] (v2) at (2,1.2);
    \coordinate [vertex] (v3) at (1.5,0);
    \coordinate [vertex] (v4) at (1,1.2);
    \coordinate [vertex] (v5) at (1.5,0.8);
    \draw[edge] (v2) -- (v3) -- (v4) -- (v1) -- (v2) -- (v5) -- (v4) (v5) -- (v3); 
\end{tikzpicture}
&
\begin{tikzpicture}
    \clip (-0.2,-0.2) rectangle (3.2,2.2);
    \draw[cylinder] (0,0) -- (3,0) (0,2) -- (3,2);
    \coordinate [vertex] (v1) at (1,1);
    \coordinate [vertex] (v2) at (1.5,1.5);
    \coordinate [vertex] (v3) at (2,1);
    \coordinate [vertex] (v4) at (1.5,0.5);
    \draw[edge] (v1) -- (v2) -- (v3) -- (v4) -- (v1); 
    \draw[edge] (v1) -- (1,2) (v1) -- (1,0);
    \draw[edge] (v3) -- (2,2) (v3) -- (2,0);
\end{tikzpicture}
\end{tabular}

\caption{All three of these \( \mathbb A \)-graphs 
    are \( (2,3,1) \)-sparse. The graph
    on the left has a blocker for a triangle contraction that is 
    not an induced subgraph. Likewise one of the contractions
    of the quadrilateral face of the middle graph has a blocker 
    that is not induced. The quadrilateral face in the right
    hand graph has a blocker for one of its contractions that is
not connected.}

    \label{fig_231examples}
    \label{fig_disconnblocker}
\end{figure}
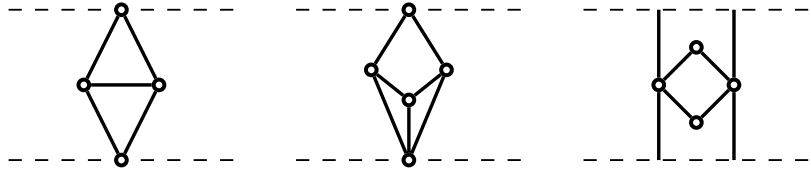

Also note that there are some additional degeneracies possible 
in the boundary of a quadrilateral face of a \( (2,3,1) \)-sparse
\( \mathbb A \)-graph. 

\begin{lemma}
    Suppose that \( Q \) is a degenerate quadrilateral face of 
    a \( (2,3,1) \)-sparse \( \mathbb A \)-graph \( G \). 
    Then \( \partial Q \) is isomorphic to one of the 
    three \( \mathbb A \)-graphs 
    shown in Figure \ref{fig_degenquads}.
    \label{lem_quaddegen231}
\end{lemma}

\begin{proof}
    If \( \partial Q \) has no loop edges then as in the proof of 
    Lemma \ref{lem_232degen} we can show that \( \partial
    Q\) is isomorphic to Figure \ref{fig_degenquads}(a). 
    On the other hand, if \( \partial Q \) has a loop edge then 
    such an edge must span an unbalanced subgraph and it is easy 
    to see then that \( \partial Q \) must be isomorphic to 
    (b) or (c) in Figure \ref{fig_degenquads}.
\end{proof}

As before we will assume in the discussion below that 
the boundary walk of a quadrilateral \( Q \) is 
\( v_1,e_1,v_2,e_2,v_3,e_3,v_4,e_4,v_1 \).
We note with respect to the examples in Figure \ref{fig_degenquads}
that 
(b) and (c) satisfy \( v_1 \neq v_3 \) and \( v_2 \neq v_4 \).
In particular both \( G_{v_1,v_3,Q} \) and \( G_{v_2,v_4,Q} \)
are defined in those cases. Note that even in these degenerate
cases we still delete one vertex and two edges in the construction 
of the contracted graph.
Let \( \delta \) be a Jordan arc joining \( v_1 \) and \( v_3 \)
whose interior lies in \( Q \). We can think of \( \delta \) as 
an edge that can be added to subgraphs of \( G \). 
\begin{lemma}
    Suppose that \( G \) is a \( (2,3,1) \)-sparse \( \mathbb A \)-graph,
    \( Q \) is a quadrilateral face of \( G \) with \( v_1 \neq v_3 \)
    and \( G_{v_1,v_3,Q} \) is not \( (2,3,1) \)-sparse. Then 
    at least one of the following statements is true.
    \begin{enumerate}[(1)]
        \item There is some \( (2,3,1) \)-tight subgraph \( B \) 
            of \( G \) such that \( E(B) \cap E(\partial Q) 
            =\{e_1,e_2\} \) or \( E(B) \cap E(\partial Q) 
            =\{e_3,e_4\}\).

        \item There is some subgraph \( B \) of \( G \) 
            such that \( B \) contains at least one edge, 
            \( f(B) \leq 4 \), \( \{v_1,v_3\} \subset V(B) \), 
            \( E(B)\cap E(\partial Q) = \emptyset  \) and 
            so that \( B\cup \delta \) is  balanced.
        \item There is some subgraph \( B \) of \( G \)
            such that \( f(B) \leq 2 \), \( \{v_1,v_3\} \subset V(B) \)
            and \( E(B) \cap E(\partial Q) = \emptyset \).
    \end{enumerate}
    \label{lem_blocker231}
\end{lemma}

\begin{proof}
    Let \( z \) be the vertex of \( G_{v_1,v_3,Q} \)
    that corresponds to \( v_1,v_3 \). 
    Let \( e' \in E(G_{v_1,v_3,Q}) \) be the edge 
    corresponding to \( \{e_1,e_2\} \) and 
    let \( e'' \in E(G_{v_1,v_3,Q}) \) be the edge 
    corresponding to \( \{e_3,e_4\} \).
    Let \( A \)
    be a subgraph of \( G_{v_1,v_3,Q} \) that 
    violates the \( (2,3,1) \)-sparsity count and 
    choose \( A \) to be unbalanced if possible.
    Clearly \( z \in V(A) \) and 
    \( |E(A) \cap \{e',e''\}| \leq 1 \), otherwise
    \( G \) would also have a subgraph that violates
    the \( (2,3,1) \)-sparsity count.
    As pointed out above, if \( A \) is balanced
    we cannot assume that 
    it is an induced subgraph. On the other hand, if 
    \( A \) is unbalanced then we can and do assume that 
    it is induced.
    Let \( B \) be the 
    subgraph defined as follows. Let 
    \( V(B) = V(A) - \{z\} \cup \{v_1,v_3\} \) 
    and \( E(B) = E(A)- \{e',e''\} \cup F\)
    where \( F \subset E(\partial Q)\) contains
    \( \{e_1,e_2\} \), respectively \( \{e_3,e_4\} \), 
    if and only if \( e' \in E(A) \), respectively
    \( e'' \in E(A) \). 
    By construction 
    \( E(B) \cap E(\partial Q) \) is one of 
    the sets \( \emptyset \), \( \{e_1,e_2\} \)
    or \( \{e_3,e_4\} \) and
    \begin{equation}
        f(B) = f(A) +2 - n/2
        \label{eqn_count2}
    \end{equation}
    where \( n = |E(B) \cap E(\partial Q)| \).
    Observe that if \( B \cup \delta \) is balanced
    if and only if \( A \) is balanced. 

    Now suppose that \( n=2 \). In this case \( B \) is 
    balanced if and only if \( A \) is balanced and 
    \( f(B) = f(A) +1 \). Since \( A \) violated the 
    \( (2,3,1) \)-sparsity count, it follows that 
    \( B \) is \( (2,3,1) \)-tight and (1) is true.

    If \( n=0 \) then \( f(B) = f(A)+2 \). Now if \( A \) is 
    unbalanced then \( f(A) \leq 0 \) and (3) is true. 
    On the other hand if \( A \) is balanced, then
    as observed above \( B \cup \delta \) is balanced. 
    Moreover
    \( A \) has at least one edge and \( f(A) \leq 2 \)
    so (2) is true.
\end{proof}

We call the subgraph \( B \) whose existence is asserted 
by Lemma \ref{lem_blocker231} a blocker for the 
contraction \( G_{v_1,v_3,Q} \).
We call \( B \) a type 1/2/3 blocker according to which 
case of Lemma \ref{lem_blocker231} applies.
Again the diagrams in Figure \ref{fig_quadblockertypes}
serve as guides for the intuition regarding the 
topology of the various types of blocker.

\begin{lemma}
    \label{lem_231blockerconn}
    If \( B \) is a type 1 or type 2 blocker for \( G_{v_1,v_3,Q} \)
    then \( B \) is connected. If \( B \) is a type 3 blocker
    then either \( B \) is connected or it has precisely 
    two components both of which are unbalanced and
    \( (2,3,1) \)-tight.
\end{lemma}
\begin{proof}
    Using Equation (\ref{eqn_incexcl}) 
    and the \( (2,3,1) \)-sparsity of \( G \) 
    we see that if \( f(B) =1 \) then \( B \) is connected.
    Similarly if \( f(B) \leq 4 \) and \( B \) is balanced 
    and disconnected then \( B \) has no edges.
    Finally if \( f(B)=2 \) and \( B \) is disconnected
    then clearly, again using Equation (\ref{eqn_incexcl}),
    both components are unbalanced and \( (2,3,1) \)-tight.
\end{proof}

See Figure \ref{fig_disconnblocker} for an example of a disconnected 
type 3 blocker.

\begin{lemma}
    \label{lem_231_2neq4}
    Suppose that \( G \) is a \( (2,3,1) \)-sparse 
    \( \mathbb A \)-graph and \( D \) is a quadrilateral 
    face of \( G \) with \( v_1 \neq v_3 \). If 
    \( G_{v_1,v_3,Q} \) is not \( (2,3,1) \)-sparse 
    then \( v_2 \neq v_4 \). 
\end{lemma}
\begin{proof}
    Let \( B \) 
    be the blocker for the contraction \( G_{v_1,v_3,Q} \)
    and suppose that \( v_2 = v_4 \). 
    Then \( Q \) is 
    degenerate and \( \partial Q \) must be isomorphic
    to Figure \ref{fig_degenquads}(a). In particular 
    \( v_2 \) 
    separates \( v_1 \) from \( v_3 \) in \( G \)
    and \( |E(\partial Q)| = 4 \).
    
    If \( v_2 \not\in V(B) \) 
    then \( B \) is a type 2 or type 3 blocker.
    Moreover \( f(B\cup \partial Q) = f(B)-2\) since 
    \( \{v_1,v_3\} \subset V(B) \) and \( |E(\partial Q)| = 4 \).
    If \( B \) is type 3 then \( f(B \cup \partial Q)
    \leq 2 -2 = 0\) which is forbidden.
    If \( B \) is type 2 then \( f(B) \leq4 \). But 
    since \( v_2 \) separates \( v_1 \) from \( v_3 \) 
    and since \( v_2 \not\in B \) we see that \( B \) is a
    balanced disconnected graph containing at least one edge.
    It follows easily from Equation (\ref{eqn_incexcl})
    that \( f(B) \geq 5 \), a contradiction.

    If \( v_2 \in V(B) \) then \( V(\partial Q) 
    \subset V(B)\). But \( E(\partial Q) \not\subset E(B) \)
    so \( B \) is not an induced subgraph and so must 
    be balanced. Now if \( B\) is type 1 then \( v_2 \)
    is a cutvertex for the Laman graph \( B \) which
    is a contradiction. If \( B \) is type 2 or type 3
    then since \( V(Q) \subset V(B) \), 
    \( f(B\cup \partial Q) = f(B)-4 \leq 0\), again a contradiction. 
\end{proof}

For an edge \( e \in G = (D,\Phi) \) let \( \mathring e \)
be the relative interior of \( |e| \) in \( |D| \).

\begin{lemma}
    Suppose that \( B \) is a type 2 blocker or type 3 blocker
    for \( G_{v_1,v_3,Q} \) and that \( v_1 \) and \( v_3 \)
    lie in the same component of \( B \).
    Then \( V(B) \) separates \( \mathring e_1 \cup \mathring e_2 \)
    from \( \mathring e_3 \cup \mathring e_4 \).
    \label{lem_blockersep231}
\end{lemma}

\begin{proof}
    The proof of Lemma \ref{lem_blockersep} works here, mutatis mutandis.
\end{proof}

\begin{proposition}
    \label{prop_quad2}
    Suppose that \( G \) is a \( (2,3,1) \)-sparse 
    \( \mathbb A \)-graph and that \( Q \) is a quadrilateral face of \( G \) 
    with boundary vertices \( v_1,v_2,v_3,v_4 \) such that 
    \( v_1 \neq v_3 \). 
    If \( G_{v_1,v_3,Q} \) is not \( (2,3,1) \)-sparse then 
    \( G_{v_2,v_4,Q} \) is \( (2,3,1) \)-sparse.
\end{proposition}
\begin{proof}
    First note that \( v_2 \neq v_4 \) by Lemma \ref{lem_blocker231}
    so \( G_{v_2,v_4,Q} \) is well defined. On the other hand, 
    it is possible that \( Q \) is a degenerate quadrilateral
    isomorphic to (b) or (c) from Figure \ref{fig_degenquads}.
    Now suppose that \( G_{v_2,v_4,Q} \) is not \( (2,3,1) \)-sparse. 
    Then by 
    Lemma \ref{lem_blocker231} there are blockers \( B_1 \), respectively 
    \( B_2 \), for the contractions \(G_{v_1,v_3,Q} \), 
    respectively 
    \( G_{v_2,v_4,Q} \). 
    Observe that \( V(\partial Q) 
    \subset V(B_1 \cup B_2) \), 
    Therefore \begin{eqnarray}
        f(B_1 \cup B_2 \cup \partial Q) &=& f(B_1 \cup B_2) - d \notag \\
        &=& f(B_1) +f(B_2) -f(B_1 \cap B_2) -d
        \label{eqn_def231} 
    \end{eqnarray} 
    where
    \( d \) is the number of edges of \( \partial Q\) 
    that are not in \( B_1 \cup B_2\) and is determined
    by the types of the blockers.
    Now there are six cases to consider.
    In the following list ``Case \( (X, Y) \)'' means that \( B_1 \) is
    a type \(X\) blocker and \( B_2 \) is a type \( Y \) blocker.

    {Case \( (1,1)\) :} 
    In this case \( d = 1 \). However \( B_1 \cap B_2 \)
    contains an edge of \( \partial Q \) so 
    by Lemma \ref{lem_232tightunion}, \( B_1
    \cup B_2\) is \( (2,3,1) \)-tight. Furthermore, 
    clearly \( B_1 \cup B_2 \cup \partial Q \) is balanced
    if and only if \( B_1 \cup B_2 \) is balanced in this case.
    So (\ref{eqn_def231}) yields the required contradiction.

    {Case \( (1,2) \):} In this case \( d = 2 \) so,  
    using (\ref{eqn_def231}) and \( f(B_2) \leq 4 \), we have 
    \begin{equation}
        f(B_1 \cup B_2 \cup \partial Q) 
        \leq f(B_1) +2 - f(B_1 \cap B_2) \label{eqn_231case12}
    \end{equation}
    Now since \( B_1 \) is tight we have \( f(B_1) 
    \leq 3\). It follows from (\ref{eqn_231case12}) that \( f(B_1\cap
    B_2) \leq 4\) and so \( B_1 \cap B_2 \), which 
    is balanced, must either be connected 
    or a pair of isolated vertices.
    If \( B_1 \cap B_2 \) is connected then 
    it is easy to see that \( B_1 \cap (B_2 \cup \partial Q) \)
    is connected and so \( B_1 \cup B_2 \cup \partial Q \) is 
    balanced if and only if \( B_1  \) is balanced.
    Then we can proceed, 
    mutatis mutandis, as in case \( (1,2) \) of Proposition 
    \ref{prop_quad1}.

    If \( B_1 \cap B_2 \) is not connected then, since it is balanced, 
    \( f(B_1 \cap B_2) \geq 4 \). Now using 
    (\ref{eqn_231case12}), and since \( B_1 \) is 
    \( (2,3,1) \)-tight, we see that \( B_1 \) 
    must be balanced, \( B_1 \cap B_2 = \{u,v\}\)
    consists of two isolated vertices and 
    \( B_1 \cup B_2 \cup \partial Q \) is unbalanced.
    Let \( \delta \) be the diagonal of \( Q \) 
    joining \( v_2 \) and \( v_4 \). 
    As in the proof of Lemma \ref{lem_blockersep} there is
    a cycle \( C \subset B_2\cup \delta\) such that 
    \( \delta \in E(C) \) and \( |C| \) separates
    \( \mathring e_1 \cup \mathring e_4\) from 
    \( \mathring e_2 \cup \mathring e_3 \) in \( \mathbb A \).
    Let \( F',F'' \) be the two faces of \( C \) and 
    let \( B' \), respectively \( B'' \), be the subgraph 
    of \( B_1 \) that is disjoint from \( F' \), respectively \( F'' \). 
    Say \( e_1 \in E(B') \) and \( e_2 \in E(B'') \). 

    Now \( B' \cap B'' \subset C  \), so \( B'\cap B'' \subset
        B_1 \cap B_2 = 
    \{u,v\}\). But \( B_1 \) does not have a cutvertex since it is 
    a Laman graph, so \( B' \cap B'' = \{u,v\} \).
    Now using Equation (\ref{eqn_incexcl}) it follows easily 
    that \( \{f(B'),f(B'')\} = \{3,4\} \). But since  
    \( B' \) and \( B'' \) are both balanced and each contains an edge, 
    it follows that each of \( B' \) and \( B'' \) is connected.

    Now \( C \) is a balanced cycle since \( B_2 \cup \delta \)
    is balanced, so without loss of generality, suppose that 
    \( F' \) is the face of \( C \) that contains both 
    ends of \( \mathbb A \). Then \( C \cup B' \) is balanced
    and it follows easily that \( B_2 \cup \delta \cup B' \) is 
    balanced. 
    See Figure \ref{fig_231case12} for a schematic diagram of this 
    situation.

    Now \( B_1 \cup B_2 \cup \delta = 
    B_1 \cup (B_2 \cup \delta \cup B') \) and \( B_1 \cap 
    (B_2 \cup \delta \cup B') = B'\) which is connected. 
    Since \( B_1 \) and \( B_2 \cup \delta \cup B' \) are 
    both balanced and \( B' \) is connected, 
    it follows from Lemma \ref{lem_svk} that 
    \( B_1\cup B_2 \cup \delta \) is balanced. It follows
    easily that \( B_1 \cup B_2  \cup \partial Q \)
    is balanced, contradicting our earlier deduction.

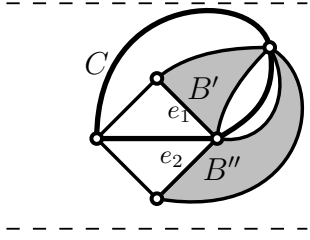
\begin{figure}
        \centering
\begin{tikzpicture}[x=1.0cm,y=1.0cm]
    \clip (-0.5,-1) rectangle (4.5,3.5);
    \coordinate (v4) at (1.2,1.2);
    \coordinate (v1) at (2,2.0);
    \coordinate (v2) at (2.8,1.2);
    \coordinate (v3) at (2,0.4);

    \coordinate (w) at ($(v2)+(60:1.4)$);  
    \draw [edge,line width=2pt] (v4) -- (v2) .. controls +(30:0.5) and +(-80:0.7) ..  (w) .. controls +(100:0.8) and +(90:2) .. (v4);
    \draw[edge,fill=gray!50] (v1) .. controls +(40:0.6) and +(180:0.4) .. (w) .. controls +(-135:0.4) and +(85:0.5) .. (v2) -- (v1);
    \draw[edge,fill=gray!50] (w) .. controls +(-60:0.7) and +(-10:0.7) .. (v2) -- (v3)
    .. controls +(-10:2) and +(-30:1) .. (w);

    (v1) .. controls +(40:0.6) and +(180:0.4) .. (w) .. controls +(-135:0.4) and +(85:0.5) .. (v2) -- (v1);

    \node at ($(v2)+(105:0.7)$){$B'$};
    \node at ($(v2)+(-80:0.4)$){$B''$};
    \node[vertex] at (w) {};

    \draw[cylinder] (0,0) -- (4,0) (0,3) -- (4,3);
    \draw[edge] (v1) -- (v2) node[pos=0.5,label={[xshift=-1mm,yshift=-5mm]:\footnotesize $e_1$}]{} -- (v3) node[pos=0.5,label={[xshift=-2mm,yshift=-3mm]:\footnotesize $e_2$}]{} -- (v4) -- (v1);
    \node[vertex] at (v1){};
    \node[vertex] at (v2){};
    \node[vertex] at (v3){};
    \node[vertex] at (v4){};
    \node at ($(v4) + (90:1)$) {$C$};
    
\end{tikzpicture}

\caption{Case \( (1,2) \) in the proof of Proposition \ref{prop_quad2}. The cycle \( C \) is constructed from a diagonal of \( Q \) and a path in \( B_2 \). The subgraphs \( B' \), respectively \( B'' \) are the parts of \( B_1 \) that lie outside, respectively inside, \( C \). }
        \label{fig_231case12}
    \end{figure}
%
%
%

    Case \( (1,3) \): 
    In this case \( d=2 \) so (\ref{eqn_def231}) 
    yields 
    \begin{equation}
        f(B_1\cup B_2 \cup \partial Q) \leq f(B_1) -f(B_1 \cap B_2)
        \label{eqn_231case13}
    \end{equation}
    Now \( B_1 \) is \( (2,3,1) \)-tight and \( B_1 \cap B_2 \)
    is not empty by Lemma \ref{lem_blockersep231}, so it 
    follows that \( B_1 \) is balanced and \( f(B_1 \cap B_2)
    \leq 2\). Since \( B_1 \cap B_2 \) is balanced, we conclude
    that \( B_1 \cap B_2 \) is a single vertex, which 
    without loss of generality we assume to be \( v_2 \).
    If \( v_2 \) and \( v_4 \) are in the same component of 
    \( B_2 \) then by Lemma \ref{lem_blockersep231}, 
    \( v_2 \) is a cutvertex for \( B_1 \) which 
    contradicts the fact that  \( B_1 \) is a Laman graph.
    On the other hand if \( v_2 \) and \( v_4 \) are in 
    different components of \( B_2 \) then \( B_2 \cup \partial Q\)
    must be embedded as shown in Figure \ref{fig_quadblockertypes}.
    Since 
    \( B_1 \) is balanced it is clear that \( v_2 \) 
    must separate \( v_1 \) and \( v_3 \) in \( B_1 \)
    again contradicting the fact that \( B_1 \) is Laman.

    Case \( (2,2) \): In this case, \( d = 4 \) and using 
    (\ref{eqn_def231}) we have \( f(B_1 \cup B_2
    \cup \partial Q) = 4 - f(B_1 \cap B_2)\) so 
    \( f(B_1 \cap B_2) \leq 3 \). But \( B_1 \cap B_2 \) 
    is balanced so \( f(B_1 \cap B_2) \in \{2,3\} \) and 
    \( B_1 \cup B_2 \cup \partial Q \) is unbalanced.

    Let \( \delta \) be the diagonal of \( Q \) joining 
    \( v_1 \) and \( v_3 \). By assumption \( B_1 \cup \delta \)
    is balanced and \( B_2 \) is also balanced.
    Now \( (B_1 \cup \delta) \cap B_2 = B_1 \cap B_2 \) 
    which must be  connected since \( f(B_1 \cap B_2) \leq 3 \). 
    So by Lemma 
    \ref{lem_svk}, \( B_1\cup B_2 \cup \delta \) is balanced. 
    Let \( F \) be the face of \( B_1 \cup B_2 \cup \delta \)
    that contains the ends of \( \mathbb A \).
    
    By Lemma \ref{lem_blockersep231}
    there is a cycle \( C \subset B_1 \cup \delta \) that 
    separates \( \mathring e_1 \cup \mathring e_2 \) 
    from \( \mathring e_3 \cup \mathring e_4 \). Now 
    \( C \) is balanced since \( B_1 \cup \delta \) is balanced, 
    so without loss of generality we can assume that 
    \( \mathring e_1 \cup \mathring e_2 \) lies in a face 
    of \( C \) that does not contain any of the 
    ends of \( \mathbb A \). Since \( C \subset B_1\cup B_2 
    \cup \delta\)  it follows that
    \(  \mathring e_1 \cup \mathring e_2 \cap F = \emptyset\).
    Therefore 
    \( (B_1 \cup B_2 \cup \delta) \cup \{e_1,e_2\} \) 
    is balanced. In particular \( B_1 \cup B_2 
    \cup \{e_1,e_2\}\) is balanced. Now by a similar 
    argument (using the other diagonal of \( Q \)) 
    we show that \( B_1 \cup B_2 \cup \{e_2,e_3\}
    \) is balanced. It follows, since \( B_1\cup B_2 \cup e_2 \)
    is connected, that 
    \( B_1 \cup B_2 \cup \{e_1,e_3,e_3\} \) is balanced 
    and then easily that \( B_1 \cup B_2 \cup \partial Q \) 
    is balanced, contradicting our earlier deduction.

    Case \( (2,3) \): Since \( d=4 \) in this case, 
    Equation (\ref{eqn_def231}) yields 
    \( f(B_1\cup B_2 \cup \partial Q) \leq 2 - f(B_1\cap B_2) \).
    Now since \( B_1 \cap B_2 \) is non-empty and balanced we have 
    the required contradiction.

    Case \( (3,3) \): Again \( d = 4 \) and 
    (\ref{eqn_def231}) yields \( 
    f(B_1 \cup B_2 \cup \partial Q)\leq  - f(B_1\cap B_2) \leq 0\). 
\end{proof}

\subsection{Proof of Theorems \ref{thm_cylinder_inductive232} and 
\ref{thm_cylinder_inductive231}}
First observe that if \( |V(G)| \leq 2 \) then 
\( G \) is isomorphic to \( K \) if it is balanced 
and \( L \) in the case \( l = 2 \), or \( M \) in the 
case \( l=1 \), if it is unbalanced.
Suppose that \( G \) has at least three vertices. 
If \( G \) is balanced then by 
Lemma \ref{lem_23leulercount} it has a triangular 
face. Now by Proposition \ref{prop_23ltriangle},
\( G_{e,T} \) is \( (2,3,l) \)-tight for some triangular
face \( T \) and \( e \in \partial T \). Moreover, it 
is clear that \( G_{e,T} \) is also balanced. The conclusion
follows by induction.

On the other hand, if \( G \) is unbalanced then by 
Lemma \ref{lem_23leulercount}, either it has 
a triangular face or a quadrilateral face. Now 
by Propositions \ref{prop_23ltriangle} and 
either Proposition \ref{prop_quad1} or Proposition \ref{prop_quad2} 
there is some contraction 
of \( G \) that is also \( (2,3,l) \)-tight. Again the 
required conclusion follows by induction.

\section{Completing sparse surface graphs to tight graphs}
\label{sec_completion}

Finally we consider the problem of adding 
edges to a sparse surface
graph to make it a tight surface graph.
We begin with the case of a $(2,3)$-sparse $\Sigma$-graph, 
where $\Sigma$ 
is a connected surface.

\begin{proposition} \label{prop_r2complete} Let $\Sigma$ be a 
    connected surface and let $G$ be a $(2,3)$-sparse $\Sigma$-graph. 
    Then there exists a $(2,3)$-tight $\Sigma$-graph $G'$ 
    such that \( G \) is a spanning 
    subgraph of \( G' \).
\end{proposition}
\begin{proof} 
    It suffices to show that if $|E(G)|<2|V(G)|-3$ 
    then we can add an edge $e$ within some
    face of $G$ so that $G\cup \{e\}$ is $(2,3)$-sparse. If $G$ is 
    disconnected, then we can clearly add such an edge since \( \Sigma \)
    is connected, 
    so we may assume that $G$ is connected.

    Let $B$ be a maximal $(2,3)$-tight subgraph of $G$ and 
    suppose that $E(B) \neq E(G)$.
    Since $G$ and \( B \) are both connected
    there exists a vertex $u\in V(B)$ that is 
    incident to an edge  \( e \in E(B) \) and also incident to 
    an edge \( f \in E(G) - E(B) \). 
    Clearly we can choose 
    \( e \) and \( f \) so that they are successive edges in the 
    boundary walk of some face \( F \) of \( G \). 
    Suppose that \( V(e) = \{u,v\} \) and 
    \( V(f) = \{u,w\} \).
    Now let \( \delta \) be a Jordan arc in \( \Sigma \)
    whose relative interior is contained in \( F \) and such that
    \( \kappa = u,e,v,\delta,w,f,u \) is the  
    boundary walk 
    of a triangular region properly contained within \( F \).
    We think of \( \delta \) as a new edge and claim that \( G\cup 
    \delta\) is \( (2,3) \)-sparse. 

    Suppose not. Then there must be a \( (2,3) \)-tight subgraph \( C \)
    of \( G \) containing \( \{v,w\} \). Since \( B \) is a maximal 
    \( (2,3) \)-tight subgraph \( G \), it follows that \( B\cup C \) is 
    not \( (2,3) \)-tight. Using (\ref{eqn_incexcl}) it follows
    that \( B \cap 
    C = \{v\}\) and \( f(B \cup C) = 4 \). But then 
    \( f \not\in B \cup C \) and \( B \cup C \cup f \) is \( (2,3) \)-tight
    contradicting the maximality of \( B \). 
\end{proof}

Now we prove Proposition \ref{prop_232complete}.
The case \( l=2 \) is quite similar to 
Proposition \ref{prop_r2complete}. On the other hand, the arguments
for the case \( l=1  \) are a little more delicate since
balanced \( (2,3,1) \)-tight subgraphs need not be induced.

\begin{proof}[Proof of Proposition \ref{prop_232complete}]
    Let \( B \) be a \( (2,3,l) \)-tight subgraph of \( G \) that is maximal with respect to 
    inclusion among all \( (2,3,l) \)-tight  subgraphs of \( G \).
    Construct \( u,v,w,e,f,\delta,\kappa \) exactly as described in 
    the proof of 
    Proposition \ref{prop_r2complete} (bearing 
        in mind that, a priori,  \( u,v,w \)
    need not be pairwise distinct).
    Suppose that \( G \cup \delta \) is not \( (2,3,l) \)-sparse.
    Then there must be some \( (2,3,l) \)-tight subgraph
    \( C \) of \( G \) such that
    \( \{v,w\} \subset V(C) \)
    and such that \( C \cup \delta \) is balanced if and 
    only if \( C \) is balanced.

    Suppose that \( l =2 \). 
    Then \( w \not\in B \) since \( B \) is an induced graph. 
    So \( C \not\subset B \) and since \( B  \) is maximal
    it follows that \( B \cup C \) is not \( (2,3,2) \)-tight.
    By Lemma \ref{lem_232tightunion}. \( B \cap C =\{v\} \).
    Now \( u \neq v \) since loop edges are forbidden in \( G \)
    so \( f \not\in C \). Thus \( f(B\cup C \cup f) 
    = f(B) +f(C) - 3 \). If \( C \) is unbalanced then \( 
    f(B \cup C\cup f) = f(B) -1 \leq 2\) and so \( B \cup
    C \cup f\) is \( (2,3,2) \)-tight, contradicting the 
    maximality of \( B \). On the other hand, if 
    \( C \) is balanced then
    \( C \cup \delta \) is balanced and so 
    \( C \cup \{e,f,\delta,u\}  \) is balanced 
    since \( \kappa \) is a boundary walk of 
    a cellular face. 
    Therefore \( C \cup  \{e,f,u\} \)
    is balanced and \( (2,3,2) \)-tight. It follows 
    from Lemma \ref{lem_232tightunion} that \( 
    B\cup C \cup f\) is \( (2,3,2) \)-tight, again 
    contradicting the maximality of \( B \).

    Now suppose that \( l =1 \). As previously observed,
    balanced \( (2,3,1) \)-tight subgraphs need not be induced. 
    However unbalanced \( (2,3,1) \)-tight subgraphs necessarily 
    are induced.
    Suppose, seeking a contradiction, that \( f \in E(C) \). 
    Then \( u, v \in V(C) \). Now  if \( C \) is balanced
    then \( C \cup \delta \) is balanced and since \( f \in E(C) \)
    and \( \kappa \) bounds a triangle, it follows that 
    \( C \cup \delta \cup e \) is balanced if \( C \) is balanced. 
    But \( C \) is 
    \( (2,3,1) \)-tight and so \( e \in E(C) \) if \( f \in C \). 
    However, by Lemma \ref{lem_231tightunion} it would follow, in that
    case, that \( B \cup C \) is \( (2,3,1) \)-tight contradicting 
    the maximality of \( B \) (since \( f \not\in B \)). 
    Thus we have shown that \( f \not\in E(C) \). 

    Now suppose, seeking a contradiction, that \( B \) is unbalanced.
    Then \( w \not\in B \) and so \( B\cup C \) is not 
    \( (2,3,1) \)-tight. By Lemma \ref{lem_231tightunion}
    it follows that \( B \cap C = \{v\} \) and that 
    \( C \) is balanced. 
    Since \( e,f, \not\in E(C) \) it follows that \( C \cup \{u,e,f\} \)
    is a \( (2,3,1) \)-tight subgraph of \( G \)
    (which is unbalanced if  and only if \( u = v \)). 
    By Lemma \ref{lem_231tightunion}, \(B \cup (C\cup \{u,e,f\}) = B \cup C \cup f\)
    is \( (2,3,1) \)-tight, contradicting the maximality of 
    \( B \). 

    Thus we can assume that \( B \) is balanced. Now
    suppose, seeking a contradiction, 
    that \( C \subset B \). Then \( C \) is 
    balanced and so \( C \cup \delta \) is also balanced.
    Thus \( B \cup \delta = B \cup (C \cup \delta) \) 
    is also balanced, using Lemma \ref{lem_svk}.
    Since \( \kappa \) bounds a triangle, it 
    follows that \( B \cup \delta \cup f \) is balanced.
    But then, since \( B \) is tight, we have \( f \in 
    E(B)\), contradicting our earlier deduction.

    Thus \( C \not\subset B \) and so 
    \( B \cup C \) is not 
    \( (2,3,1) \)-tight. It follows from Lemma
    \ref{lem_231tightunion} that \( B\cap C \)
    has no edges and at most two vertices. If 
    \( |V(B\cap C)| = 2 \) then \( B \cup C \cup f\)
    is unbalanced and \( (2,3,1) \)-tight, contradicting 
    the maximality of \( B \). 
    If \( |V(B \cap C)| = 1 \) then 
    \begin{center}
        \begin{tabular}{rcl}
            \( B\cup C \) is balanced & \( \Leftrightarrow \) & \( C \) is balanced \\
            & \(\Leftrightarrow\) & \( C \cup \delta \) is balanced \\
            & \(\Leftrightarrow\) & \( B \cup C \cup \delta \) is balanced \\
            & \(\Leftrightarrow\) & \( B \cup C \cup \{\delta,f\} \) is balanced \\
            & \(\Rightarrow\) & \( B \cup C \cup f\) is balanced 

        \end{tabular}
    \end{center}
    Again we conclude that  \( B \cup C \cup f \) is 
    \( (2,3,1) \)-tight,
    contradicting the maximality of \( B \).
\end{proof}


\end{document}